\documentclass[11pt]{amsart}

\usepackage{amsmath,amssymb}
\usepackage{graphicx}

\usepackage{fullpage}
\usepackage{enumerate}
\usepackage{pgf,tikz}
\usetikzlibrary{patterns}
\usetikzlibrary{arrows}
\usetikzlibrary{positioning}

\def\COMMENT#1{}
\let\COMMENT=\footnote
\def\TASK#1{}

\def\noproof{{\unskip\nobreak\hfill\penalty50\hskip2em\hbox{}\nobreak\hfill%
        $\square$\parfillskip=0pt\finalhyphendemerits=0\par}\goodbreak}
\def\endproof{\noproof\bigskip}
\newdimen\margin   
\def\textno#1&#2\par{%
    \margin=\hsize
    \advance\margin by -4\parindent
           \setbox1=\hbox{\sl#1}%
    \ifdim\wd1 < \margin
       $$\box1\eqno#2$$%
    \else
       \bigbreak
       \hbox to \hsize{\indent$\vcenter{\advance\hsize by -3\parindent
       \sl\noindent#1}\hfil#2$}%
       \bigbreak
    \fi}

\newtheorem{thm}{Theorem}[section]
\newtheorem{examp}[thm]{Example}
\newtheorem{lemma}[thm]{Lemma}
\newtheorem{problem}[thm]{Problem}
\newtheorem{claim}[thm]{Claim}

\newtheorem{col}[thm]{Corollary}
\newtheorem{conj}[thm]{Conjecture}
\newtheorem{prop}[thm]{Proposition}
\newtheorem{remark}[thm]{Remark}
\newtheorem*{thm*}{Theorem}
\newtheorem*{define*}{Definition}
\newtheorem*{examp*}{Example}
\newtheorem*{lem*}{Lemma}
\newtheorem*{claim*}{Claim}
\newtheorem*{fact*}{Fact}
\newtheorem*{col*}{Corollary}
\newtheorem*{conj*}{Conjecture}
\newtheorem{exex}{Extremal Example}
\def\eps{\varepsilon}

\theoremstyle{definition}
\newtheorem{define}[thm]{Definition}

\begin{document}

\title{Dirac-type results for tilings and coverings in ordered graphs}

\author{Andrea Freschi}
\author{Andrew Treglown}
\thanks{AF: School of Mathematics, University of Birmingham, United Kingdom, {\tt axf079@bham.ac.uk}. \\
\indent AT: School of Mathematics, University of Birmingham, United Kingdom, {\tt a.c.treglown@bham.ac.uk}.\\
}

\date{\today}
\begin{abstract}
A recent paper of Balogh, Li and Treglown~\cite{blt} initiated the study of Dirac-type problems for ordered graphs. In this paper we prove a number of results in this area.
In particular, we determine asymptotically the minimum degree threshold for forcing
\begin{itemize}
    \item[(i)] a
    perfect $H$-tiling in an ordered graph, for any fixed ordered graph $H$ of interval chromatic number at least $3$;
    \item[(ii)]  an $H$-tiling in an ordered graph $G$ covering a fixed proportion of the vertices of $G$ (for any fixed ordered graph $H$);
  
     \item[(iii)] an $H$-cover in an ordered graph (for any fixed ordered graph $H$).
\end{itemize}
The first two of these results resolve questions of  Balogh, Li and Treglown whilst (iii) resolves a question of Falgas-Ravry. Note that (i) combined with a result from \cite{blt} completely determines the asymptotic minimum degree threshold for forcing a perfect $H$-tiling. Additionally, we prove a result that combined with a theorem of Balogh, Li and Treglown,  asymptotically  determines the minimum degree threshold for forcing an almost perfect $H$-tiling in an ordered graph (for any fixed ordered graph $H$). Our work therefore provides ordered graph analogues of the seminal
tiling theorems of K\"uhn and Osthus~[Combinatorica 2009] and of Koml\'os~[Combinatorica 2000].
Each of our results exhibits some curious, and perhaps unexpected, behaviour. Our solution to (i) makes use of a novel absorbing argument.
\end{abstract}
\maketitle

\section{Introduction} \label{Introduction}
In recent years there has been a significant effort to develop both Tur\'an and Ramsey theories in the setting of  \emph{vertex ordered graphs}. 
A \emph{(vertex) ordered graph} or \emph{labelled graph} $H$ on $h$ vertices is a graph whose vertices have been labelled with $[h]:=\{1,\dots,h\}$.
An ordered graph $G$ with vertex set $[n]$ \emph{contains} an ordered graph $H$ on $[h]$ if (i) there is an injection $\phi : [h] \rightarrow [n]$ such that
$\phi (i) < \phi (j)$ for all $1\leq i < j \leq h$ and  (ii) $\phi(i)\phi(j)$ is an edge in $G$ whenever $ij$ is an edge in $H$.

Tur\'an-type problems concern edge density conditions that force a fixed graph $H$ as a subgraph in a host graph $G$. Whilst the  Erd\H{o}s--Stone--Simonovits theorem~\cite{es1, es2} determines, up to a quadratic error term, the number of edges in the densest $H$-free $n$-vertex graph, there is still active interest in the Tur\'an problem for bipartite $H$. Indeed, for bipartite $H$ the error term in the Erd\H{o}s--Stone--Simonovits theorem is in fact the dominant term, so more refined results are sought. 
Similarly, a result of Pach and Tardos~\cite{pt} determines asymptotically the number of edges an ordered graph requires to force a copy of a fixed ordered graph $H$ with so-called interval chromatic number $\chi_<(H)$ at least $3$. Therefore, again there is significant interest in the `bipartite' case of this problem (i.e., when $\chi_<(H)=2$); see Tardos~\cite{tardosbcc} for a recent survey on such results.

The study of Ramsey theory for ordered graphs has also gained significant traction. For example, results of Conlon, Fox, Lee and Sudakov~\cite{con}, and of
 Balko,  Cibulka,  Kr\'al and Kyn\v{c}l~\cite{balko}
 demonstrate that there are ordered graphs $H$ for which the behaviour of the Ramsey number is vastly different to their underlying unordered graph.

Other than Tur\'an and Ramsey problems, another central branch of extremal graph theory concerns Dirac-type results; that is, minimum degree conditions that force fixed (spanning) structures in graphs.
In a recent paper, Balogh, Li and the second author~\cite{blt} 
 initiated the study of Dirac-type results for ordered graphs. 
 Their main focus was on 
 \emph{perfect $H$-tilings}, though they also raised other Dirac-type problems (see \cite[Section 8]{blt}).
In both the ordered and unordered settings, 
an \emph{$H$-tiling} in a graph $G$ 
is a collection of vertex-disjoint copies of $H$ contained in $G$.
 An
$H$-tiling is \emph{perfect} if it covers all the vertices of $G$.
Perfect $H$-tilings are also often referred to as \emph{$H$-factors}, \emph{perfect $H$-packings} or \emph{perfect $H$-matchings}. 
$H$-tilings can be viewed as generalisations of both the notion of a matching (which corresponds to the case when $H$ is a single edge) and the Tur\'an problem (i.e.,~a copy of $H$ in $G$ is simply an $H$-tiling of size one).

Balogh, Li and Treglown~\cite{blt} raised the following question.
\begin{problem}\cite{blt}\label{mainprob}
Given any ordered graph $H$ and any $n \in \mathbb N$ divisible by $|H|$, determine
the smallest integer $\delta_{<}(H, n)$ such that every $n$-vertex ordered graph $G$  with $\delta(G)\geq \delta_{<}(H, n)$ contains a perfect $H$-tiling. 
\end{problem}
The analogous problem in the (unordered) graph setting had been studied since the 1960s (see, e.g.,~\cite{alonyuster, cor, hs, kssAY, kuhn, kuhn2}) and forty-five years later 
 a complete solution, up to an additive constant term, was obtained via a theorem of K\"uhn and Osthus~\cite{kuhn2}. We will discuss this result further when comparing this problem with Problem~\ref{mainprob}.

In~\cite[Theorem 1.9]{blt}, Balogh, Li and Treglown asymptotically resolved Problem~\ref{mainprob} for
$H$ with $\chi_<(H)=2$. Further, they developed  approaches to the absorbing and regularity methods for ordered graphs, including providing  general absorbing and almost perfect tiling lemmas. In this paper we build on their  results to asymptotically resolve Problem~\ref{mainprob} in all remaining cases (i.e., for all $H$ with interval chromatic number at least $3$). Our main result shows that Problem~\ref{mainprob} does exhibit a somewhat different behaviour when $\chi _<(H)\geq 3$ compared to when $\chi_<(H) = 2$; we discuss this further in Section~\ref{compare}.

In addition to this result, we also resolve the Dirac-type problems for \emph{$H$-covers} in ordered graphs (see Section~\ref{sec:cover}) and $H$-tilings covering a fixed proportion $x$ of the vertices of an ordered graph for $x \in (0,1)$ (see Section~\ref{sec:x}).

\subsection{A Dirac-type theorem for perfect $H$-tilings}
In this subsection we state Theorem~\ref{mainthm}, which asymptotically resolves Problem~\ref{mainprob}  when $\chi _<(H)\geq 3$. This result will depend on several definitions and parameters which we now introduce.

\begin{define}[Interval chromatic number]
Given  $r \in \mathbb N$, an \emph{interval $r$-colouring} of an ordered graph $H$ is 
a partition of the vertex set $[h]$ of $H$ into $r$ (possibly empty) intervals so that  no two vertices belonging
to the same interval are adjacent in $H$.
The \textit{interval chromatic number} $\chi_{<}(H)$ of an ordered graph $H$ is the smallest $r\in \mathbb N$ such that there exits an interval $r$-colouring of $H$.
\end{define}
One can think of the interval chromatic number as the natural ordered analogue of the chromatic number of a graph. Moreover, whilst the Erd\H{o}s--Stone--Simonovits theorem~\cite{es1, es2} asserts that $1-1/(\chi (H)-1)+o(1)$ is the edge density threshold for ensuring a copy of a graph $H$ in $G$, the 
\emph{Erd\H{o}s--Stone--Simonovits theorem for ordered graphs} due to Pach and Tardos~\cite{pt} asserts the corresponding threshold in the ordered graph setting is 
$1-1/(\chi_< (H)-1)+o(1)$.

The next definition is the relevant parameter for studying \emph{almost} perfect $H$-tilings in graphs.
\begin{define}[Critical chromatic number]
The \textit{critical chromatic number} $\chi_{cr}(F)$ of an unordered graph $F$ is defined as 
\[
\chi_{cr}(F):=(\chi(F) - 1)\frac{|F|}{|F| - \sigma(F)},
\]
where $\sigma(F)$ denotes the size of the smallest possible colour class in any $\chi(F)$-colouring of $F$.
\end{define}
Note that $\chi(H)-1<\chi_{cr}(H)\leq \chi(H)$ for all graphs $H$, and $\chi_{cr}(H)=\chi(H)$ precisely when every $\chi(H)$-colouring of $H$ has colour classes of equal size.

We informally refer to an $H$-tiling in an  $n$-vertex (ordered) graph $G$ as an
\emph{almost perfect $H$-tiling} if it covers all but at most $o(n)$ vertices of $G$.
Koml\'os~\cite{komlos} proved that $(1-1/\chi_{cr}(H))n$ is the minimum degree threshold for forcing an almost perfect $H$-tiling in an $n$-vertex graph $G$. In fact, it was later shown~\cite{szhao} that such graphs $G$ contain $H$-tilings covering all but a constant number of the vertices in $G$.
In the setting of ordered graphs,  a related parameter $\chi^*_{cr}(H)$  turns out to be the relevant parameter for forcing an almost perfect $H$-tiling. To introduce this parameter we need the following definitions.

For two subsets $X, Y$ of $[n]$, we write $X<Y$ if $x<y$ for all $x\in X$ and $y\in Y$. 
Let $B$ be a complete $k$-partite unordered graph with parts $U_1,\ldots, U_k$,
and $\sigma$ be a permutation of the set $[k]$. 
An \textit{interval labelling} of $B$ with respect to $\sigma$ is a bijection $\phi: V(B)\rightarrow [|B|]$ such that $\phi(U_i)<\phi(U_j)$ if $\sigma(i)<\sigma(j)$; that is, $$\phi(U_{\sigma^{-1}(1)})<\cdots<\phi(U_{\sigma^{-1}(k)}).$$
For brevity, we will usually drop $\phi$ and just write $U_{\sigma^{-1}(1)}<\cdots<U_{\sigma^{-1}(k)}.$
Given $t \in \mathbb N$, write $B(t)$ for the \emph{blow-up  of $B$} with vertex set $\bigcup_{x\in V(B)}V_x$, where the $V_x$'s are sets of $t$ independent vertices; so there are all possible edges between $V_x$ and $V_y$ in $B(t)$ if $xy \in E(B)$.  Given an interval labelling $\phi$ of $B$,
let $(B(t), \phi)$ be the ordered graph obtained from $B(t)$ by equipping $V(B(t))$ with a vertex ordering, satisfying $V_x < V_y$ for every $x, y\in V(B)$ with $\phi(x)<\phi(y)$.
We refer to $(B(t), \phi)$ as an \textit{ordered blow-up}  of $B$.

\begin{define}[Bottlegraph]\label{bottlegraphdef}
For an ordered graph $H$, we say that a complete $k$-partite unordered graph $B$ is a \textit{bottlegraph of $H$}, if for every permutation $\sigma$ of $[k]$ and every interval labelling $\phi$ of $B$ with respect to $\sigma$, there exists a constant $t=t(B, H, \phi)$ such that the ordered blow-up $(B(t), \phi)$ contains a perfect $H$-tiling. We say that $B$ is a \textit{simple bottlegraph of $H$} if for any choice of $\sigma$ and $\phi$ we can take $t=1$.
\end{define}
Note that in Definition~\ref{bottlegraphdef} we did not impose any restriction on the size of the parts of a bottlegraph. However, as we will see in Proposition~\ref{propbottle},
it suffices to consider bottlegraphs $B'$ where  all parts are of the same size except for perhaps one smaller part. More precisely, given any bottlegraph $B$ of $H$, there is another bottlegraph $B'$ with this structure such that $\chi _{cr}(B')=\chi_{cr}(B)$. This bottle-like structure 
is where the name {bottlegraph} is derived, and was first used by Koml\'os~\cite{komlos}
 in the setting of unordered graphs.

\begin{define}[Ordered critical chromatic number]
The \textit{ordered critical chromatic number} $\chi^*_{cr}(F)$ of an ordered graph $F$ is defined as 
$$\chi_{cr}^*(F):=\inf\{\chi_{cr}(B):B\text{ is a bottlegraph of }F\}.$$
We say a bottlegraph $B$ of $F$ is \emph{optimal} if $\chi_{cr}(B)=\chi_{cr}^*(F)$.
\end{define}

Notice that $\chi _<(H)-1 \leq \chi_{cr}^*(H)$ for all ordered graphs $H$ as each bottlegraph $B$ of $H$ must have chromatic number at least $\chi_<(H)$ and so 
$\chi _<(H)-1 < \chi _{cr}(B)$. In fact, Proposition~\ref{stronglowerb} in Section~\ref{sec:extremalconstructions} yields a stronger lower bound on $\chi_{cr}^*(H)$.
On the other hand, (in contrast to $\chi_{cr}(F)$ for unordered graphs $F$) we will also see examples of ordered graphs where $\chi_{cr}^*(H)$ is much larger than $\chi_<(H)$.
Note though that $\chi_{cr}^*(H)\leq h$ for any ordered graph $H$ on $[h]$ as $K_{h}$ is a bottlegraph of $H$. In fact, this upper bound is attained when $H$ is such that $1$ and $2$ are adjacent or $h-1$ and $h$ are adjacent; this is an immediate consequence of Proposition~\ref{lowerbound}. 
To aid the reader's intuition, in Section~\ref{examp3} we give examples of ordered graphs $H$ where we compute $\chi^*_{cr}(H)$. Various  bounds on $\chi^*_{cr}(H)$ are given in Section~\ref{sec:properties}.

\smallskip

The next result, a simple corollary of~\cite[Theorem~4.3]{blt}, shows that 
$\chi^*_{cr}(H)$ is a relevant parameter for forcing an almost perfect $H$-tiling in an ordered graph.\footnote{The reader may not immediately see why Theorem~\ref{BalLiTre1} is a corollary of Theorem 4.3 from \cite{blt} as there is an error term in the minimum degree condition in the latter theorem. However, a standard trick (see, e.g., the proof of Corollary~6.6 in~\cite{forum}) 
always allows one to omit an error term in a minimum degree condition that forces an almost perfect $H$-tiling.}

\begin{thm}[Balogh, Li and Treglown \cite{blt}]\label{BalLiTre1}
Let $H$ be an ordered graph.  Then for every $\eta>0$, there exists an integer $n_0=n_0(H,\eta)$ so that every ordered graph $G$ on $n\geq n_0$ vertices with
$$\delta(G)\geq\left(1-\frac{1}{\chi^*_{cr}(H)}\right)n$$
contains an $H$-tiling covering all but at most $\eta n$ vertices.
\end{thm}
At first sight it is not  clear if the minimum degree threshold in Theorem~\ref{BalLiTre1} is best possible. However, 
Theorem~\ref{aptsthm} in Section~\ref{lowersec} shows that Theorem~\ref{BalLiTre1} is best possible in the sense that one cannot replace the $1-{1}/{\chi^*_{cr}(H)}$ term in the minimum degree condition with any other fixed constant term $a< 1-{1}/{\chi^*_{cr}(H)}$.  Thus, Theorem~\ref{aptsthm} and Theorem~\ref{BalLiTre1} provide an analogue of
Koml\'os' theorem in the ordered setting.

Unusually, in the proof of Theorem~\ref{aptsthm}, for most ordered graphs $H$ we do not 
 simply produce an explicit extremal example. 
  Indeed, if one has not explicitly computed the value of $\chi^*_{cr}(H)$ and the `reason' why it takes this value, then it seems difficult to produce such an explicit extremal example. Instead, the proof splits into a few cases and uses various tools and results that we introduce in the paper.

\smallskip

At this point the reader may wonder if the conclusion of Theorem~\ref{BalLiTre1} can
be strengthened to ensure a \emph{perfect} $H$-tiling.
For some ordered graphs $H$ this is possible. However, for other ordered graphs one will require a significantly higher minimum degree condition.
The following definition is the critical concept for articulating this dichotomy
for $H$ with $\chi _<(H) \geq 3$.

\begin{define}[Local barrier]\label{localbarrier}
Let $H$ be an ordered graph on $[h]$ with $r:=\chi_<(H)\geq 2$. We say that $H$ has a \emph{local barrier} if for some fixed $i\ne j\in[r+1]$ the following condition holds.
Given any interval $(r+1)$-colouring of $H$ with colour classes 
$V_1<\cdots<V_{r+1}$  such that $V_i=\{v\}$ is a singleton class, there is at least one edge between $v$ and $V_j$ in $H$.
\end{define}
Note that in this definition we may have that a colour class $V_k$ is empty.
If $H$ is the ordered complete graph on $r$ vertices then $H$ does not have a local barrier; it is also easy to check that $\chi^*_{cr}(H)=\chi_<(H)=r$. Given $r \geq 2,$
let $H'$ be any complete $r$-partite (unordered) graph with at least $2$ vertices in each colour class. Let $H$ be any ordered graph obtained from $H'$ by assigning an interval labelling to $H'$; so $\chi _<(H)=r$. Then one can check that $H$ has a local barrier with parameters $i=1$ and $j=r+1$ as in Definition~\ref{localbarrier}.

\smallskip

We are now able to state our main result which resolves Problem~\ref{mainprob} for all
ordered graphs $H$ with $\chi_<(H) \geq 3$.

\begin{thm}\label{mainthm}
Let  $H$ be an ordered graph with $\chi_<(H)\geq 3$. Given any $\eta >0$, there exists an integer $n_0=n_0(H,\eta)$ so that if $n\geq n_0$ and $|H|$ divides $n$ then
\begin{itemize}
    \item[(i)] $ \left (1-\frac{1}{\chi^*_{cr}(H)}  \right )n \leq  \delta_<(H,n)\leq\left (1-\frac{1}{\chi^*_{cr}(H)} +\eta \right )n $ \ \  if
    $\chi^*_{cr}(H) \geq \chi_<(H)$;
\item[(ii)] $ \left (1-\frac{1}{\chi_<(H)} \right )n< \delta_<(H,n)\leq   \left (1-\frac{1}{\chi_<(H)} +\eta \right )n $  \ \  if 
 $\chi^*_{cr}(H) < \chi_<(H)$ and $H$ has a local barrier;
 \item[(iii)] $ \left (1-\frac{1}{\chi^*_{cr}(H)}  \right )n \leq  \delta_<(H,n)\leq\left (1-\frac{1}{\chi^*_{cr}(H)} +\eta \right )n $ \ \  if 
    $\chi^*_{cr}(H) < \chi_<(H)$ and $H$ has no local barrier.
    
\end{itemize}
\end{thm}
Therefore Problem~\ref{mainprob} is now asymptotically resolved. 
The reader might find it hard to see why the value of $\delta_<(H,n)$ behaves as in Theorem~\ref{mainthm}, and indeed, it took the authors quite some time to discover the correct behaviour of this problem. 
In Section~\ref{compare} we give further intuition on this result.
In Section~\ref{sec:examples1} we give examples  of $H$ in each case (i)--(iii) of the theorem.
In Section~\ref{sec:extremalconstructions} we give the extremal constructions for Theorem~\ref{mainthm}. In particular, in cases (i) and (iii) the extremal examples are `bottlegraphs' -- complete multipartite ordered graphs where each part has the same size except at most one smaller part; in Section~\ref{sec:properties} we explicitly  compute the value of $\chi^*_{cr}(H)$  for a range of $H$, and thus the minimum degree threshold in Theorem~\ref{mainthm} also. The explanation for why these extremal `bottlegraphs' 
do not have perfect $H$-tilings revolve around `space barrier' constraints, (i.e., one runs out of space in some subset of vertices in the extremal bottlegraph).\footnote{See~\cite{keevashmycroft} for a discussion on space barriers.} However, the `reason' for obtaining a space barrier can be somewhat involved (and unlike any other space barrier extremal example we have ever seen before); we discuss this in Section~\ref{examp1}.

The proof of Theorem~\ref{mainthm} applies an absorbing theorem from~\cite[Theorem 4.1]{blt} and Theorem~\ref{BalLiTre1} above. The main novelty is to prove an absorbing theorem for ordered graphs $H$ as in Theorem~\ref{mainthm}(iii). Whilst our argument  makes use of a lemma of Lo and Markstr\"om~\cite{lo}, and seems rather natural, it is  different to any absorbing proof we have previously seen  (in particular, we do not use \emph{local-global absorbing} as in~\cite{blt}). See Section~\ref{sec:proofsketch} for an overview of our absorbing strategy.

\subsection{A Dirac-type theorem for vertex covers}\label{sec:cover}
Given  (ordered) graphs $H$ and $G$, we say that $G$ has an \emph{$H$-cover} if every vertex in $G$ lies in a copy of $H$. Note that the notion of an $H$-cover is an `intermediate' between seeking a single copy of $H$ and a perfect $H$-tiling; in particular, a perfect $H$-tiling in $G$ is itself an $H$-cover.
Given any $n \in \mathbb N$ and any (ordered) graph $H$, let $\delta_{\text{cov}}(H,n)$
denote the smallest integer $k$ such that  every $n$-vertex (ordered) graph $G$  with $\delta(G)\geq k$ contains an  $H$-cover.
As noted in~\cite{orekot}, an easy application of Szemer\'edi's regularity lemma asymptotically determines $\delta_{\text{cov}}(H,n)$ for all graphs $H$.
\begin{prop}\cite[Proposition~6]{orekot}\label{prop1}
For every graph $H$ and every $\eta >0$, there exists an integer $n_0=n_0(H,\eta)$ so that if $n \geq n_0$ then
$$ \left (1-\frac{1}{\chi(H)-1}  \right )n-1
\leq \delta_{\text{cov}}(H,n)\leq \left (1-\frac{1}{\chi(H)-1} +\eta \right )n.$$
\end{prop}
Proposition~\ref{prop1} implies that, asymptotically, the minimum degree threshold for ensuring an $H$-cover in a graph $G$ is the same as the minimum degree threshold for ensuring a \emph{single copy} of $H$ in $G$. 
In~\cite[Theorem~5]{orekot}, K\"uhn, Osthus and Treglown asymptotically determined the Ore-type degree condition that forces an $H$-cover for any fixed graph $H$. 
There has also been several recent papers concerning  minimum $\ell$-degree conditions that force $H$-covers in $k$-uniform hypergraphs; see, e.g.,~\cite{vfrzhao2, vfrzhao, hanetc}.

Falgas-Ravry~\cite{vfrpc} raised the question of determining $\delta_{\text{cov}}(H,n)$ for all ordered graphs $H$. Our next result asymptotically answers this question.
\begin{thm}\label{thmcover}
Let  $H$ be an ordered graph and  $\eta >0$. Then there exists an integer $n_0=n_0(H,\eta)$ so that if $n \geq n_0$ then
\begin{itemize}
    \item[(i)] $ \left (1-\frac{1}{\chi_<(H)-1}  \right )n
< \delta_{\text{cov}}(H,n)\leq \left (1-\frac{1}{\chi_<(H)-1} +\eta \right )n $ \ \ if $H$ has no local barrier;
\item[(ii)] $ \left (1-\frac{1}{\chi_<(H)}  \right )n
<\delta_{\text{cov}}(H,n)\leq \left (1-\frac{1}{\chi_<(H)} +\eta \right )n $ \ \ if $H$ has a local barrier.
\end{itemize}
\end{thm}
 Theorem~\ref{thmcover} is a direct consequence of some of the auxiliary results we use in the proof of Theorem~\ref{mainthm}.
Note that the behaviour of the threshold in Theorem~\ref{thmcover} is perhaps unexpected.
Indeed, unlike in the unordered setting, Theorem~\ref{thmcover} and the Erd\H{o}s--Stone--Simonovits theorem for ordered graphs imply that the asymptotic minimum degree thresholds for forcing a copy of $H$ and an $H$-cover are different if $H$ has a local barrier.

Furthermore, a key moral of the Erd\H{o}s--Stone--Simonovits theorem (and Proposition~\ref{prop1}) is that 
once a graph $G$ is dense enough (or has large enough minimum degree) so as to ensure a copy of $K_r$ (or a $K_r$-cover) then $G$ must contain every fixed graph $H$ (or an $H$-cover) for every $H$ of chromatic number $r$. An intuition for this comes from Szemer\'edi's regularity lemma.
However, the analogous moral is not true for $H$-covers in ordered graphs. Indeed, if $H$ is an ordered complete graph on $r$ vertices (so $\chi_<(H)=r$ and $H$ has no local barrier) then Theorem~\ref{thmcover} tells us the minimum degree threshold for forcing an $H$-cover in an $n$-vertex ordered graph is $(1-\frac{1}{r-1}+o(1)   )n$ whilst
the corresponding threshold for any `blow-up' $H'$ of $H$ (so $\chi_<(H')=\chi_<(H)=r$ and $H$ has a local barrier) is significantly higher, namely 
$(1-\frac{1}{r}+o(1)  )n$. This should hint to the reader that the regularity method behaves differently in the ordered setting; in particular, if $H$ has a  local barrier this provides an  obstruction when applying the regularity lemma. More discussion on the regularity method for ordered graphs can be found in~\cite[Section 3.1]{blt}.

\subsection{Intuition behind the threshold in Theorem~\ref{mainthm}}\label{compare}
In this subsection we build up further intuition behind the threshold in 
Theorem~\ref{mainthm}. For this it will be useful to first take a step back and consider perfect $H$-tilings in unordered graphs. In this setting the Dirac-type threshold is governed by two factors:
\begin{itemize}
    \item[$(C1)$] The minimum degree needs to be large enough to force an almost perfect $H$-tiling. 
    \item[$(C2)$] The minimum degree must be large enough to prevent `divisibility' barriers within the host graph that constrain us from turning an almost perfect $H$-tiling into a perfect $H$-tiling.
\end{itemize}
This is made precise by the following theorem of K\"uhn and Osthus~\cite{kuhn2}. 
(See~\cite[Section 1.2]{kuhn2} for the definition of $\text{hcf}(H)=1$.)
\begin{thm}[K\"uhn and Osthus~\cite{ kuhn2}]\label{kothm}
Let $\delta(H, n)$ denote the smallest integer $k$ such that every graph $G$ whose order $n$ is divisible by $|H|$ and with $\delta(G)\geq k$ contains a perfect $H$-tiling. 
For every unordered graph $H$, 
\[
\delta(H, n)=\left(1 - \frac{1}{\chi^*(H)}\right)n +O(1), 
\]
where  $\chi^*(H):=\chi_{cr} (H)$ if $\text{hcf}(H)=1$ and $\chi^*(H):=\chi (H)$ otherwise.
\end{thm}
Recall that every graph $H$ satisfies $\chi_{cr} (H)\le \chi(H)$. So by Koml\'os' aforementioned almost perfect tiling theorem~\cite{komlos}, the minimum degree condition in Theorem~\ref{kothm} is enough to ensure $(C1)$ holds.
Meanwhile those graphs $H$ with $\text{hcf}(H)=1$ are precisely the graphs for which, at the almost  perfect tiling threshold, $(C2)$ is satisfied. Furthermore, for graphs $H$ with $\text{hcf}(H)\ne 1$, 
 $(C2)$ is only guaranteed to be satisfied once the $n$-vertex host graph has minimum degree around $(1-1/\chi(H))n$.
 
We do not state the precise definition 
 of $\text{hcf}(H)=1$ here, however, the following example is instructive. Let $H$ be any connected bipartite graph and let $n \in \mathbb N $ be divisible by $|H|$. Consider the $n$-vertex graph $G$ that  consists of two disjoint cliques whose sizes are as equal as possible so that neither is divisible by $|H|$.
 Then whilst $G$ contains an almost perfect $H$-tiling, the divisibility constraint on the clique sizes prevents a perfect $H$-tiling.
 Thus all such $H$ are examples of graphs with $\text{hcf}(H) \ne 1$.
In particular, $\delta(G)=n/2-O(1)=(1-1/\chi(H))n-O(1)$, so $G$ is an extremal example for 
 Theorem~\ref{kothm} in this case.
 
 As mentioned earlier, another necessary condition for a Dirac-type threshold for perfect $H$-tilings is the following:
\begin{itemize}
    \item[$(C3)$] The minimum degree needs to be large enough to force an $H$-cover. 
\end{itemize}
Condition $(C3)$, however, does not factor into the statement of Theorem~\ref{kothm} as
Proposition~\ref{prop1} shows that one can ensure an $H$-cover `earlier' than 
an almost perfect $H$-tiling (recall that $\chi(H)-1< \chi_{cr}(H)$).

Interestingly, the opposite is true in the ordered graph setting when $H$ has a local barrier and $\chi_<(H)\geq 3$ is such that $\chi _<(H)> \chi_{cr}^* (H)$.
Indeed, in this case, Theorem~\ref{mainthm}(ii) essentially states that 
the $H$-cover condition is the `last' of  conditions $(C1)$--$(C3)$  to be satisfied.
In particular,
Extremal Example~1 in Section~\ref{sec:extremalconstructions} shows that for every $H$ satisfying Definition~\ref{localbarrier}, there are $n$-vertex ordered graphs with $\delta(G)> (1-1/\chi_<(H))n-1$ for which a certain vertex does not lie in a copy of $H$.

In all other cases when $\chi_<(H)\geq 3$, Theorem~\ref{mainthm} essentially states that 
the almost perfect tiling condition is the `last' of  conditions $(C1)$--$(C3)$  to be satisfied. 
Therefore,  surprisingly (at least to the authors!), divisibility barriers play no role in Problem~\ref{mainprob} for $H$ with 
$\chi_<(H)\geq 3$.
In contrast, Theorem~1.9  in~\cite{blt} shows that in the case when $\chi _<(H)=2$,   each of  $(C1)$, $(C2)$ and $(C3)$ can be the condition
that governs the Dirac-type threshold for perfect $H$-tiling, depending on the choice of $H$.

\subsection{A Dirac-type theorem for $H$-tilings}\label{sec:x}
In addition to determining the Dirac-type threshold for almost perfect tilings in 
unordered graphs, Koml{\'o}s~\cite{komlos} provided a best-possible minimum degree condition for forcing an $H$-tiling covering a certain proportion of the vertices in a graph $G$.

\begin{define}[$(x,H)$-tilings]
Let $G$ and $H$ be (ordered) graphs and $x\in[0,1]$. An \emph{$(x,H)$-tiling in  $G$} is an $H$-tiling covering at least $x|G|$  vertices. So a $(1,H)$-tiling is simply a perfect $H$-tiling.
\end{define}
\begin{thm}[Koml{\'o}s \cite{komlos}]\label{Komlos}
Let $H$ be a graph and $x\in (0,1)$. Define
$$g(x,H):=\left(1-\frac{1}{\chi(H)-1}\right)(1-x)+\left(1-\frac{1}{\chi_{cr}(H)}\right)x.$$
Given any $\eta >0$,
there exists some $n_0=n_0(x,H,\eta)\in\mathbb{N}$ such that if $G$ is a graph on $n$ vertices where $n\geq n_0$ and $\delta(G)\geq g(x,H)\cdot n$ 
then there exists an  $(x-\eta,H)$-tiling in $G$. 
\end{thm}
Note that the minimum degree condition in Theorem~\ref{Komlos} is best possible in the sense that given any fixed \(H\) and \(x \in (0,1)\), one cannot replace \(g(x,H)\) with any fixed \(g'(x,H)<g(x,H)\) (see~\cite[Theorem~7]{komlos} for a proof of this).

The function $g(x,H)$ is quite well-behaved. Indeed, for fixed $H$, $g(x,H)$ grows linearly in $x$. Note that $g(0,H)\cdot n$ and $g(1,H)\cdot n$ are the asymptotic minimum degree thresholds for ensuring an $n$-vertex graph contains a copy of $H$ and an almost perfect $H$-tiling respectively. From this prospective, the function $g(x,H)$ can be viewed as a linear interpolation of these two thresholds. 

The question of obtaining an ordered graph analogue of Theorem~\ref{Komlos} was raised in~\cite[Question~8.2]{blt}. We provide an answer to this problem; for this we require the following definitions.

\begin{define}[$x$-bottlegraphs]\label{defx}
Let $H$ be an ordered graph and $x\in(0,1]$. An unordered graph $B$ is an \emph{$x$-bottlegraph of $H$} if it satisfies the following properties:
\begin{itemize}
\item[(i)] $B$ is a complete $k$-partite graph with parts $U_1,U_2,\dots,U_k$, for some $k\in\mathbb{N}$.
\item[(ii)] There exists some $m\in\mathbb{N}$ such that $|U_1|\leq m$ and $|U_i|=m$ for every $i>1$.
\item[(iii)] Given any permutation $\sigma$ of $[k]$ and any interval labelling of $B$ with respect to $\sigma$, the resulting ordered graph
 contains an $(x,H)$-tiling.
\end{itemize}
\end{define}

\begin{define}
Let $H$ be an ordered graph and $x\in(0,1]$. 
We define  $\chi_{cr}^*(x,H)$ as
$$\chi_{cr}^*(x,H):=\inf \{\chi_{cr}(B) \, : \, B \, \text{is an $x$-bottlegraph of $H$}\}.$$
\end{define}
Given any ordered graph $H$,
 if $B$ is an $x$-bottlegraph of $H$ then $\chi (B) \geq \chi _<(H)$. This implies
that $\chi _{cr}(B) > \chi_<(H)-1$ and so
\begin{align}\label{goodlowerbound}
    \chi ^* _{cr}(x,H) \geq \chi _< (H)-1.
\end{align}

An application of Theorem~\ref{Komlos}
together with a tool from~\cite[Lemma~6.2]{blt} 
yields the following minimum degree condition for the existence of $(x,H)$-tilings in ordered graphs.

\begin{thm}\label{OrderedKomlos}
Let $H$ be an ordered graph, $x\in(0,1)$ and define
$$f(x,H):=\left(1-\frac{1}{\chi_{cr}^*(x,H)}\right).$$
Given any $\eta>0$, there exists some $n_0=n_0(x,H,\eta)\in\mathbb{N}$ such that if $G$ is an ordered graph on $n$ vertices with $n\geq n_0$ and $\delta(G)\geq (f(x,H)+\eta) n$ 
then $G$ contains an $(x,H)$-tiling.
\end{thm}
The minimum degree condition in Theorem~\ref{OrderedKomlos} is best possible in the following sense.
Let $H$ and  $x\in(0,1)$ be fixed. Given any $0<a <1-1/\chi_{cr}^*(x,H)$, and any sufficiently large
 $n \in \mathbb N$, consider any $n$-vertex graph $B$ that satisfies (i) and (ii) of Definition~\ref{defx} for some choice of $k,m \in \mathbb N$
 and where
 $$an < \delta (B) = n-m= \left(1-\frac{1}{\chi_{cr}(B)}\right)n< f(x,H)\cdot n.$$
So $\chi_{cr}(B) <\chi_{cr}^*(x,H)$.
 (Note such a graph $B$ exists for any choice of  $0<a <1-1/\chi_{cr}^*(x,H)$.)
 Then by definition of $\chi_{cr}^*(x,H)$ there is a permutation $\sigma$ of $[k]$ and an interval labelling $\phi$ of $B$ with respect to $\sigma$, such the resulting ordered graph $(B,\phi)$ does not
 contain an $(x,H)$-tiling.


A draw-back of Theorem~\ref{OrderedKomlos} is that 
it seems hard to compute $\chi_{cr}^*(x,H)$ in general. However, in Section~\ref{sec:xtiling}
we describe the behaviour of the function $f(x,H)$ for some fixed ordered graphs $H$. In particular,
akin to Theorem~\ref{Komlos}, if $H$ has $\chi_<(H)=2$ then $f(x,H)$ is linear in $x$.
Perhaps surprisingly though, there are ordered graphs where $f(x,H)$ is only piecewise linear. We also compute $f(x,H)$ for every ordered graph $H$ and every $x$ that is not too big.

\subsection{Organisation of the paper}
The paper is organised as follows. In Section~\ref{sec:extremalconstructions} we give the extremal constructions for Theorems~\ref{mainthm} and~\ref{thmcover}. 
In Section~\ref{sec:examples1} we give some examples of ordered graphs $H$ that fall into each of the three cases of Theorem~\ref{mainthm}.
In Section~\ref{sec:proofmainthm}
we state a new absorbing theorem (Theorem~\ref{absorbingthm}) and an absorbing theorem from~\cite{blt}
and combine them with Theorem~\ref{BalLiTre1} to prove Theorem~\ref{mainthm}. The subsequent sections therefore build up tools for the proof of Theorem~\ref{absorbingthm}: in Section~\ref{sec:bottle} we state a couple of useful properties of bottlegraphs;
in Section~\ref{sec:regularitylemma} we introduce Szemer\'edi's regularity lemma and related useful results; some tools for absorbing are given in Section~\ref{sec:absorbingtools}; Section~\ref{sec:usefulresults} 
contains several results which  give  flexibility in how one can interval colour certain ordered graphs $H$. 

In Section~\ref{sec:proofsketch} we give a sketch of the proof of Theorem~\ref{absorbingthm} before proving it and Theorem~\ref{thmcover} in Section~\ref{sec:proofs}.
In Section~\ref{sec:properties} we give  general upper and lower bounds on $\chi^*_{cr}(H)$ and also compute $\chi^*_{cr}(H)$ for a few general classes of ordered graphs $H$.

In Section~\ref{lowersec} we prove that the minimum degree condition in Theorem~\ref{BalLiTre1} is best possible. 
The proof of Theorem~\ref{OrderedKomlos} is given in Section~\ref{sec:ok}; in the subsequent section we describe the behaviour of the function $f(x,H)$ for some choices of $H$. We conclude the paper with some open problems in Section~\ref{sec:conc}.

\subsection{Notation}
Given  $n\in \mathbb N$, let  $[n]:=\{1, \ldots, n\}$.
A \emph{nearly balanced interval partition} of $[n]$ is a partition of $[n]$ into intervals
$W_1<\dots <W_t$ where $||W_i|-|W_j||\leq 1$ for every $1\leq i, j \leq t$.
Similarly, a $t$-partite graph with vertex classes $V_1,\ldots, V_t$ is \textit{nearly balanced} if $||V_i|-|V_j||\leq 1$ for every $1\leq i, j \leq t$.

If $G$ is an (ordered) graph, $|G|$ denotes the size of its vertex set, and $e(G)$ denotes the number of edges in $G$.
Given $A\subseteq V (G)$, the \emph{induced subgraph $G[A]$} is the subgraph of $G$ whose vertex set is $A$ and whose edge set consists of all of the edges of $G$ with both endpoints in $A$. We define $G\setminus A:=G[V(G)\setminus A]$.
For two disjoint subsets $A, B\subseteq V (G)$, the \emph{induced bipartite subgraph $G[A, B]$} is the subgraph of $G$ whose vertex set is $A\cup B$ and whose edge set
consists of all of the edges of $G$ with one endpoint in $A$ and the other endpoint in $B$. 
We write $e(A,B):=e(G[A,B])$.

For an (ordered) graph $G$ and a vertex $x \in V(G)$, we define 
$N_G(x)$ as the set of neighbours of $x$ in $G$ and
$d_G (x):=|N_G(x) |$. For $X \subseteq V(G)$ we define $d_G (x,X):=|N_G(x)\cap X |$.
Given  (ordered) graphs $G$ and $H$ and $X \subseteq V(G)$ we say that $G[X]$ \emph{spans a copy of $H$ in $G$} if $H$ is a spanning subgraph of $G[X]$.


Given an ordered graph $G$ we say that \emph{$V_1<\dots<V_r$ is an interval $r$-colouring of $G$}
to mean that there is an interval $r$-colouring of $G$ with colour classes $V_1< \dots <V_r$.
We say that an ordered graph $G$ is \emph{complete $r$-partite} if there exists an interval $r$-colouring $V_1<\cdots<V_r$ such that $xy\in E(G)$ for every $x\in V_i$ and $y\in V_j$ with $i\not=j$. We refer to the $V_i$'s as the \emph{parts} of $G$.

Given an unordered graph $G$ and a positive integer $t$, let $G(t)$ be the graph obtained from $G$ by replacing every vertex $x\in V(G)$ by a set $V_x$ of $t$ vertices spanning an independent set, and joining $u\in V_x$ to $v\in V_y$ precisely when $xy$ is an edge in $G$; that is, we replace the edges of $G$ by copies of $K_{t,t}$. We will refer to $G(t)$ as a \textit{blown-up} copy of $G$.
If $U_i$ is a vertex class in $G$ then we write $U_i(t)$ for the corresponding vertex class in $G(t)$. We use analogous notation when considering  blown-up copies of complete $k$-partite ordered graphs. In particular, given a
complete $k$-partite ordered graph $B$ with parts $B_1<\dots < B_k$, the \emph{ordered blow-up $B(t)$ of $B$} consists of parts $B_1(t)<\dots < B_k(t)$ where $|B_i(t)|=t|B_i|$ for all $i \in [k]$.

Throughout the paper, we
omit all floor and ceiling signs whenever these are not crucial.
The constants in the hierarchies used to state our results are chosen from right to left.
For example, if we claim that a result holds whenever $0< a\ll b\ll c\le 1$, then 
there are non-decreasing functions $f:(0,1]\to (0,1]$ and $g:(0,1]\to (0,1]$ such that the result holds
for all $0<a,b,c\le 1$  with $b\le f(c)$ and $a\le g(b)$. 
Note that $a \ll b$ implies that we may assume in the proof that, e.g., $a < b$ or $a < b^2$.

\section{Extremal constructions}\label{sec:extremalconstructions}
In this section we provide the extremal examples for Theorems~\ref{mainthm} and~\ref{thmcover}. 
First, consider the case when $H$ has a local barrier. We now construct an $n$-vertex ordered graph  which does not contain an $H$-cover (and thus no perfect $H$-tiling), and whose minimum degree is more than $(1-1/\chi_<(H))n-1$, thereby giving the lower bounds in Theorem~\ref{mainthm}(ii) and Theorem~\ref{thmcover}(ii).

\begin{exex}\label{ex1}
Let $n,r\in\mathbb{N}$ and $i,j\in[r+1]$ with $i\not=j$. Let $F_1(n,r,i,j)$ be an $n$-vertex ordered graph  consisting of vertex classes $U_1<\cdots<U_{r+1}$ which satisfy the following conditions: 
\begin{itemize}
\item $U_i=\{u\}$ is a singleton class while the remaining vertex classes are as equally sized as possible, and in particular,  $|U_j|=\left\lfloor\frac{n-1}{r}\right\rfloor$;
\item $F_1(n,r,i,j) \setminus\{u\}$ is a complete $r$-partite ordered graph with parts $U_1,\dots, U_{i-1}, U_{i+1}, \dots U_{r+1}$;
\item $u$ is adjacent to all other vertices except those in  $U_j$.
\end{itemize}
\end{exex}
Note that
\begin{align}\label{eq10}
\delta(F_1(n,r,i,j))=n-1-|U_j|=n-1-\left\lfloor\frac{n-1}{r}\right\rfloor >\left(1-\frac{1}{r}\right)n-1.
\end{align}
Furthermore, we now prove that $F_1(n,r,i,j)$ does not contain an $H$-cover (nor a perfect $H$-tiling) provided that   $\chi _<(H)=r$ and $H$ has a local barrier with respect to parameters $i,j\in[r+1]$. 

\begin{lemma}\label{ex1lemma}
Let $H$ be an ordered graph, let $r:=\chi_<(H)$ and let $n\in\mathbb{N}$. If $H$ has a local barrier then there exist $i,j\in\mathbb{N}$, with $i\not=j$, and a vertex $u\in F_1(n,r,i,j)$ such that there is no copy of $H$ in $F_1(n,r,i,j)$ covering the vertex $u$. In particular, $F_1(n,r,i,j)$ does not contain an $H$-cover nor a perfect $H$-tiling.
\end{lemma}
\begin{proof}
Suppose $H$ has a local barrier with respect to $i\ne j\in[r+1]$, as defined in Definition~\ref{localbarrier}. Let $u$ be the vertex in the singleton class $U_i$ of $F_1(n,r,i,j)$. Suppose there is a copy of $H$ in $F_1(n,r,i,j)$ covering the vertex $u$. Then the interval $(r+1)$-colouring  $U_1<\cdots<U_{r+1}$ of $F_1(n,r,i,j)$   induces an interval $(r+1)$-colouring $V_1<\cdots<V_{r+1}$ of $H$ such that $V_i=\{v\}$ is a singleton class and there is no edge between $v$ and $V_j$. This contradicts the assumption that $H$ has a local barrier with respect to $i,j$; thus, there is no copy of $H$ in $F_1(n,r,i,j)$ covering the vertex $u$.
\end{proof}
We immediately obtain the following corollary of Lemma~\ref{ex1lemma} and (\ref{eq10}).
\begin{col}\label{ex1col}
Let $H$ be an ordered graph and let $n\in \mathbb N$. If $H$ has a local barrier then 
$$\delta_{cov}(H,n)>\left(1-\frac{1}{\chi_<(H)}\right)n$$
and, if $|H|$ divides $n$
$$\delta_<(H,n)>\left(1-\frac{1}{\chi_<(H)}\right)n.$$ 
\qed
\end{col}

Next we prove a general lower bound on $\delta_<(H,n)$ which  is asymptotically sharp if the ordered graph $H$ does not have a local barrier. Similarly to before, we  construct an $n$-vertex ordered graph  which does not contain a perfect $H$-tiling and whose minimum degree is at least $(1-1/\chi_{cr}^*(H))n-1$, thereby giving the lower bound in cases (i) and (iii) of Theorem~\ref{mainthm}. 

\begin{exex}\label{ex2}
Let $H$ be an ordered graph and $n\in\mathbb{N}$. Set $\ell:=\left\lfloor\frac{n}{\chi_{cr}^*(H)}+1\right\rfloor$. Define $F_2(H,n)$ to be the unordered complete $\lceil n/\ell\rceil$-partite graph on $n$ vertices such that all classes have size $\ell$ except for one class of size at most $\ell$.
\end{exex}

It is easy to check that the minimum degree of $F_2(H,n)$ is
\begin{align}\label{eq101}
\delta(F_2(H,n))=n-\ell=n-\left\lfloor\frac{n}{\chi_{cr}^*(H)}+1\right\rfloor\geq\left(1-\frac{1}{\chi_{cr}^*(H)}\right)n-1.
\end{align}
Additionally, there exists a certain ordering of the vertices of $F_2(H,n)$ such that the resulting ordered graph  does not contain a perfect $H$-tiling:

\begin{lemma}\label{ex2lemma}
Let $H$ be an ordered graph and $n\in\mathbb{N}$ such that $|H|$ divides $n$. There exists an interval labelling $\phi$ of $F_2(H,n)$ such that the ordered graph $(F_2(H,n),\phi)$ does not contain a perfect $H$-tiling.
\end{lemma}
\begin{proof}
The critical chromatic number of $F_2(H,n)$ is
$$\chi_{cr}(F_2(H,n))=\frac{n}{\ell}<\frac{n}{n/\chi_{cr}^*(H)}=\chi_{cr}^*(H).$$
It follows that $F_2(H,n)$ is not a bottlegraph of $H$. Hence, by definition, there exists a permutation $\sigma$ of $[\lceil n/\ell \rceil]$ and an interval labelling $\phi$ of $F_2(H,n)$ with respect to $\sigma$ such that $(F_2(H,n),\phi)$ does not contain a perfect $H$-tiling.
\end{proof}
Lemma~\ref{ex2lemma} and~(\ref{eq101}) immediately imply the following corollary.
\begin{col}\label{ex2col}
Let $H$ be an ordered graph. Then given any $n \in \mathbb N$ divisible by $|H|$,
$$\delta_<(H,n)\geq \delta(F_2(H,n))+1 \geq\left(1-\frac{1}{\chi_{cr}^*(H)}\right)n.$$\qedhere
\end{col}

Next we give a general lower bound for $\delta_{cov}(H,n)$ which  is asymptotically sharp if the ordered graph $H$ does not have a local barrier, thereby giving the lower bound in Theorem~\ref{thmcover}(i). 

\begin{exex}\label{ex3}
Let $H$ be an ordered graph and $n\in\mathbb{N}$. 
Let $F_3(H,n)$ be the complete $(\chi_<(H)-1)$-partite ordered graph on $n$ vertices with parts of size as equal as possible.
\end{exex}

It is easy to check that the minimum degree of $F_3(H,n)$ is
$$\delta(F_3(H,n))=n-\left\lceil\frac{n}{\chi_<(H)-1}\right\rceil>\left(1-\frac{1}{\chi_<(H)-1}\right)n-1.$$
As $\chi_<(F_3(H,n))< \chi_<(H)$, $F_3(H,n)$  does not contain a copy of $H$ and thus does not contain an $H$-cover. We therefore obtain the following result.

\begin{lemma}\label{ex3lemma}
Let $H$ be an ordered graph and $n \in \mathbb N$. Then 
$$\delta_{cov}(H,n)\geq \delta(F_3(H,n))+1 >\left(1-\frac{1}{\chi_<(H)-1}\right)n.$$\qed
\end{lemma}

For $n$ divisible by $\chi_<(H)-1$,
 $F_3(H,n)$ also shows  that the minimum degree threshold that  ensures an almost perfect $H$-tiling in an $n$-vertex ordered graph is  more than $(1-1/(\chi_<(H)-1))n$. Thus, combined with Theorem~\ref{BalLiTre1} this immediately implies that, for all ordered graphs $H$,
 \begin{align}\label{cruciallower}
    \chi_<(H)-1 < \chi^*_{cr}(H).
 \end{align}
Actually, we close the section by proving an even stronger lower bound on $\chi_{cr}^*(H)$.

\begin{prop}[A lower bound for $\chi_{cr}^*(H)$]\label{stronglowerb}
Let $H$ be an ordered graph on $h$ vertices and $r:=\chi_<(H)$. Then, $$\chi_{cr}^*(H)\geq(r-1)+\frac{r-1}{h-1}.$$
\end{prop}
\begin{proof}
Let $B$ be an arbitrary bottlegraph of $H$. It suffices to show that $\chi_{cr}(B)\geq(r-1)+\frac{r-1}{h-1}$. If $\chi_{cr}(B)\geq r$ then we are done (since $h\geq r$), so for the rest of the proof we assume that $\chi_{cr}(B)<r$. In particular, as (\ref{cruciallower}) implies that $\chi_{cr}(B)>r-1$, this means that $B$ has exactly $r$ parts. Let $B_1$ denote the part of $B$ of smallest size. Pick any interval labelling $\phi$ of $B$; then there exists some $t\in\mathbb{N}$ such that the ordered blow-up $(B(t),\phi)$ contains a perfect $H$-tiling $\mathcal{H}$. Since $B$ has exactly $r$ parts, it follows that every copy of $H$ in $(B(t),\phi)$ intersects all parts of $B$. Hence,
$$|B_1(t)|\geq|\mathcal{H}|=\frac{|B(t)|}{|H|}=t|B|/h\implies|B_1|\geq|B|/h,$$
and so
$$\chi_{cr}(B)=(r-1)\frac{|B|}{|B|-|B_1|}\geq(r-1)\frac{|B|}{|B|-|B|/h}=(r-1)+\frac{r-1}{h-1}.$$
\end{proof}
In Section~\ref{examp3} we give a family of ordered graphs $H$ for which the lower bound on $\chi_{cr}^*(H)$ in Proposition~\ref{stronglowerb} is tight.

\section{Motivating examples}\label{sec:examples1}
\subsection{An example for Theorem~\ref{mainthm}(i)}\label{examp1}
Recall that Extremal Example~2 yields the lower bound in cases (i) and (iii) of Theorem~\ref{mainthm}.
The argument in Lemma~\ref{ex2lemma} is rather straightforward. This is because of the definition of $\chi^*_{cr}(H)$; if one takes a complete multipartite graph $G$ with 
$\chi_{cr}(G)<\chi^*_{cr}(H)$, then by definition there is a vertex labelling of $G$ so that the resulting ordered graph does not contain a perfect $H$-tiling.

Therefore, if one provides an argument that justifies why a bottlegraph of $H$ is optimal, this equivalently can be translated into an argument which explains why an ordered graph is an extremal example for cases (i) and (iii) of 
Theorem~\ref{mainthm}. In this way, one can view $\chi^*_{cr}(H)$ as `encoding' properties of the extremal example.

In Section~\ref{sec:properties} we will compute $\chi^*_{cr}(H)$ for various classes of ordered graphs $H$.
Often these arguments will be somewhat involved; thus, in these cases  
 the reason why the extremal example for Theorem~\ref{mainthm} does not contain a perfect $H$-tiling is also `involved'. That is, in general the reason why  extremal examples do not contain perfect $H$-tilings is not as immediate as Lemma~\ref{ex2lemma} might suggest.
We illustrate this point through the following example.

\begin{examp}[An example for Theorem~\ref{mainthm}(i)]\label{example1}
Let  $\ell\geq2$ and let $H$ be the complete $3$-partite ordered graph with parts $H_1<H_2<H_3$ of size $\ell,1,\ell$ respectively. 
\end{examp}
For $H$ as in Example~\ref{example1}, 
we have that $\chi_{cr}^*(H)=(4\ell^2-1)/\ell^2>3=\chi_<(H)$. (In fact, in Proposition~\ref{comp3partitecase} we compute $\chi_{cr}^*(F)$ for all complete $3$-partite ordered graphs $F$.) Thus, for such $H$ we are in case (i) of Theorem~\ref{mainthm} and so
$$\delta_<(H,n)=\left(1-\frac{\ell^2}{4\ell^2-1}+o(1)\right)n.$$

We now describe an extremal example for Theorem~\ref{mainthm} for such $H$.
Let $n\in\mathbb{N}$ such that $|H|$ divides $n$ and $n\geq 20$. Let $G$ be the complete $4$-partite ordered graph on $n$ vertices with parts $G_1<G_2<G_3<G_4$ where
$$|G_1|=|G_2|=|G_3|=\left\lfloor\frac{n\ell^2}{4\ell^2-1}\right\rfloor+1\quad\text{and}\quad|G_4|=n-3\left\lfloor\frac{n\ell^2}{4\ell^2-1}\right\rfloor-3.
$$
Note that $G_4$ is the smallest part since $|G_i|\geq n/4$ for $i=1,2,3$ and $|G_4|\leq n/4$. In particular,
\begin{align*}
\delta(G)=n-\left\lfloor\frac{n\ell^2}{4\ell^2-1}\right\rfloor-1\geq n-\frac{n\ell^2}{4\ell^2-1}-1=\left(1-\frac{\ell^2}{4\ell^2-1}\right)n-1.
\end{align*}
Suppose for a contradiction that $G$ contains a perfect $H$-tiling $\mathcal{H}$. Let $\mathcal{A}\subseteq\mathcal{H}$ be the set of copies of $H$ in $\mathcal{H}$ which have exactly $\ell$ vertices in $G_1$ and set $\mathcal{B}:=\mathcal{H}\setminus\mathcal{A}$. This immediately implies
\begin{align}\label{eqAB2}
|G_1|\geq \ell|\mathcal{A}|.
\end{align}
Note that if $H'\in\mathcal{A}$ then $H'$ has at most $\ell+1$ vertices in $G_1\cup G_2$ while if $H'\in\mathcal{B}$ then $H'$ has at most $\ell$ vertices in $G_1\cup G_2$. It follows that
\begin{align}\label{eqAB1}
|G_1|+|G_2|\leq (\ell+1)|\mathcal{A}|+\ell|\mathcal{B}|.
\end{align}
Combining (\ref{eqAB2}) and (\ref{eqAB1}) yields the following:
\begin{align*}
|\mathcal{A}|+|\mathcal{B}|\stackrel{(\ref{eqAB1})}{\geq}&\frac{|G_1|+|G_2|-|\mathcal{A}|}{\ell}\stackrel{(\ref{eqAB2})}{\geq}\frac{\ell|G_1|+\ell|G_2|-|G_1|}{\ell^2} \\
=&\frac{2\ell-1}{\ell^2}\left(\left\lfloor\frac{n\ell^2}{4\ell^2-1}\right\rfloor+1\right)>\frac{2\ell-1}{\ell^2}\left(\frac{n\ell^2}{4\ell^2-1}\right)=\frac{n}{2\ell+1}.
\end{align*}
The above is a contradiction since
$$|\mathcal{A}|+|\mathcal{B}|=|\mathcal{H}|=\frac{|G|}{|H|}=\frac{n}{2\ell+1}.$$
Hence, $G$ does not contain a perfect $H$-tiling.

\smallskip

Note that $G$ is a `space barrier' construction as our argument tells us that $G_1 \cup G_2$ is `too big' to ensure a perfect $H$-tiling in $G$; moreover, the reason why $G_1 \cup G_2$ is `too big', whilst not difficult, is not at first sight, obvious (i.e., we needed to consider how two types of copies of $H$ intersect $G_1\cup G_2$).

Space barrier constructions occur in many other settings too (e.g., the K\"uhn--Osthus perfect tiling theorem~\cite{kuhn}). However,
all previous graph space barrier constructions 
we are aware of have a different flavour to the above space barrier $G$. Indeed, previously known examples  fail to contain the desired substructure due to some very immediate property that means 
 \emph{one} vertex class is `too small' or `too big'.
 
 In Section~\ref{sec:properties} we compute $\chi^*_{cr}(H)$ precisely for several classes of ordered graphs.
In particular,
we  give other ordered graphs $H$ that fall into case (i) of Theorem~\ref{mainthm}, namely all complete $3$-partite ordered graphs and all complete $r$-partite ordered graphs whose smallest part is the first or last part (see Propositions~\ref{complete1class} and~\ref{comp3partitecase}).

\subsection{An example for Theorem~\ref{mainthm}(ii)}
The next example provides a family of ordered graphs that fall into case (ii) of Theorem~\ref{mainthm}. 

\begin{examp}[An example for Theorem~\ref{mainthm}(ii)]\label{example2}
Let $r,k\geq 3$ and let $H$ be the ordered graph with vertex set $V(H)=[(r-1)k+2]$ and edge set $E(H)=\{(1,(r-1)k+2)\}\cup\{((s-1)k+2,sk+2):s\in[r-1]\}$ (see Figure~\ref{fig:1}). So $\chi_<(H)=r$.
\end{examp}

\begin{figure}[!h]
\centering
\begin{tikzpicture}
\filldraw[black] (0,0) circle (3pt) node[below=3pt]{$1$};
\filldraw[black] (1,0) circle (3pt) node[below=3pt]{$2$};
\filldraw[black] (2,0) circle (3pt) node[below=3pt]{$3$};
\filldraw[black] (3,0) circle (3pt) node[below=3pt]{$4$};
\filldraw[black] (4,0) circle (3pt) node[below=3pt]{$5$};
\filldraw[black] (5,0) circle (3pt) node[below=3pt]{$6$};
\filldraw[black] (6,0) circle (3pt) node[below=3pt]{$7$};
\filldraw[black] (7,0) circle (3pt) node[below=3pt]{$8$};
\filldraw[black] (8,0) circle (3pt) node[below=3pt]{$9$};
\filldraw[black] (9,0) circle (3pt) node[below=3pt]{$10$};
\filldraw[black] (10,0) circle (3pt) node[below=3pt]{$11$};
\draw (0,0) arc (135:45:7.071);
\draw (1,0) arc (135:45:2.121);
\draw (4,0) arc (135:45:2.121);
\draw (7,0) arc (135:45:2.121);
\end{tikzpicture}
\caption{The ordered graph $H$ as in Example~\ref{example2} for $r=4$ and $k=3$.}
\label{fig:1}
\end{figure}

Let $B$ the complete $r$-partite graph with parts $B_1,\dots,B_r$ where $|B_i|=k$ for $i\in[r-1]$ and $|B_r|= 2$. Observe that $\chi_{cr}(B)=(r-1)+2/k$. It is straightforward to check that for any permutation $\sigma$ of $[r]$ and any interval labelling $\phi$ of $B$ with respect to $\sigma$, the ordered graph $(B,\phi)$ contains a spanning copy of $H$; hence $B$ is a simple bottlegraph of $H$. It follows that
$$\chi_{cr}^*(H)\leq\chi_{cr}(B)=(r-1)+\frac{2}{k}<r=\chi_<(H).$$
Furthermore, $H$ has a local barrier: for any interval $(r+1)$-colouring $\{1\}<V_1<\cdots<V_r$ of $H$ we have that $(r-1)k+2\in V_r$ and thus there is one edge between $\{1\}$ and $V_r$. 

\subsection{An example for Theorem~\ref{mainthm}(iii)}\label{examp3}
Next we consider a family of ordered graphs which fall into case (iii) of Theorem~\ref{mainthm}. 

\begin{examp}[An example for Theorem~\ref{mainthm}(iii)]\label{example3}
Let $r,k\geq 2$ and let $H$ be the ordered graph with vertex set $V(H)=[(r-1)k+1]$ and edge set $E(H)=\{((s-1)k+1,sk+1):s\in[r-1]\}$. So $H$ is  the path $(1)(k+1)(2k+1)\dots((r-1)k+1)$ and $\chi_<(H)=r$ (see Figure~\ref{fig:2}).
\end{examp}
\begin{figure}[!h]
\centering
\begin{tikzpicture}
\filldraw[black] (0,0) circle (3pt) node[below=3pt]{$1$};
\filldraw[black] (2,0) circle (3pt) node[below=3pt]{$2$};
\filldraw[black] (4,0) circle (3pt) node[below=3pt]{$3$};
\filldraw[black] (6,0) circle (3pt) node[below=3pt]{$4$};
\filldraw[black] (8,0) circle (3pt) node[below=3pt]{$5$};
\filldraw[black] (10,0) circle (3pt) node[below=3pt]{$6$};
\filldraw[black] (12,0) circle (3pt) node[below=3pt]{$7$};
\draw (0,0) arc (135:45:2.828);
\draw (4,0) arc (135:45:2.828);
\draw (8,0) arc (135:45:2.828);
\end{tikzpicture}
\caption{The ordered graph $H$ as in Example~\ref{example3} for $r=4$ and $k=2$.}
\label{fig:2}
\end{figure}

We will explicitly compute $\chi^*_{cr}(H)$. In particular, we prove that $\chi^*_{cr}(H)<\chi_<(H)$ and that $H$ does not have a local barrier. 
We first construct a bottlegraph of $H$. Let $B$ denote the complete $r$-partite graph with parts $B_1,\dots,B_r$ where $|B_i|=k$ for $i\in[r-1]$ and $|B_r|= 1$. Observe that $\chi_{cr}(B)=(r-1)+1/k$. It is straightforward to check that for any permutation $\sigma$ of $[r]$ and any interval labelling $\phi$ of $B$ with respect to $\sigma$, the ordered graph $(B,\phi)$ contains a spanning copy of $H$. Thus, $B$ is a simple bottlegraph of $H$ and so $\chi_{cr}^*(H)\leq\chi_{cr}(B)=(r-1)+1/k$.
Moreover,  Proposition~\ref{stronglowerb} implies that in fact 
$\chi_{cr}^*(H)= (r-1)+1/k$.



Finally, we show that $H$ does not have a local barrier.
 Let $i\not=j\in[r+1]$.
If $i\not\in\{1,r+1\}$, there exists an interval $(r+1)$-colouring $V_1<\cdots<V_{r+1}$ of $H$ such that $V_i=\{x\}$ with $x\not=(s-1)k+1$ for every $s\in[r]$. Then $x$ is isolated in $H$ and so clearly there is no edge between $x$ and $V_j$.  
If $i=1$, there exists an interval $(r+1)$-colouring $V_1<\cdots<V_{r+1}$ of $H$ such that $V_i=\{1\}$ and  $V_j=\emptyset$; so again, there is no edge between $V_i$ and $V_j$. The case $i=r+1$ is analogous; we take $V_i=\{(r-1)k+1\}$ and $V_j=\emptyset$.


\section{Proof of Theorem~\ref{mainthm}}\label{sec:proofmainthm}
In this section we present some intermediate results and explain how they imply Theorem~\ref{mainthm}. Crucial to our approach will be the use of the \emph{absorbing method}, a technique that was introduced systematically by R\"odl, Ruci\'nski
and Szemer\'edi~\cite{rrs2}, but that has roots in earlier work (see, e.g.,~\cite{kriv}).
Given ordered graphs $G,H$ and a set $W\subseteq V(G)$, a set $S\subseteq V(G)$ is called an \emph{$H$-absorbing set for $W$} if both $G[S]$ and $G[W\cup S]$ contain  perfect $H$-tilings. In \cite[Theorem~4.1]{blt}, Balogh, Li and the second author provided a minimum degree condition that ensures an ordered graph $G$ contains a set $Abs$ that is an $H$-absorbing set for {every} not too large set $W\subseteq V(G)\setminus Abs$.

\begin{thm}[Balogh, Li and Treglown \cite{blt}]
\label{BalLiTre2}
Let $H$ be an ordered graph on $h$ vertices and let $\eta>0$. Then there exists an $n_0\in\mathbb{N}$ and $0<\nu\ll\eta$ so that the following holds. Suppose that $G$ is an ordered graph on $n\geq n_0$ vertices and 
$$\delta(G)\geq\left(1-\frac{1}{\chi_<(H)}+\eta\right)n.$$
Then $V(G)=[n]$ contains a set $Abs$ so that
\begin{itemize}
\item $|Abs|\leq\nu n$;
\item $Abs$ is an $H$-absorbing set for every $W\subseteq V (G)\setminus Abs$ such that $|W|\in h\mathbb{N}$ and $|W|\leq\nu^3n$.
\end{itemize}
\end{thm}

 Theorems~\ref{BalLiTre1} and~\ref{BalLiTre2} can be combined to yield a minimum degree condition that forces a perfect $H$-tiling. Indeed, let $G$ and $H$ be ordered graphs and suppose  that
 $$\delta(G)\geq\left(1-\frac{1}{\max\{\chi_<(H),\chi_{cr}^*(H)\}}+o(1)\right)n.$$
  We first invoke Theorem~\ref{BalLiTre2} to find a set $Abs\subseteq V(G)$ which is an $H$-absorbing set for any not too large set $W\subseteq V(G)\setminus Abs$. Then we apply Theorem~\ref{BalLiTre1} to $G\setminus Abs$ to find an $H$-tiling $\mathcal{M}_1$ which covers all but a small proportion of vertices in $G\setminus Abs$. Let $W$ denote the set of such vertices in $G\setminus Abs$. Since $W$ is relatively small, $Abs$ is an $H$-absorbing set for $W$, and thus $G[W\cup Abs]$ contains a perfect $H$-tiling $\mathcal{M}_2$. Finally, observe that $\mathcal{M}_1\cup\mathcal{M}_2$ is a perfect $H$-tiling in $G$.

Thus we have proven that
$$\delta_<(H,n)\leq\left(1-\frac{1}{\max\{\chi_<(H),\chi_{cr}^*(H)\}}+o(1)\right)n.$$ 
In particular, this is asymptotically sharp if $\chi_{cr}^*(H)\geq\chi_<(H)$ (by Corollary~\ref{ex2col}) or if $\chi_{cr}^*(H)<\chi_<(H)$ and $H$ has a local barrier (by Corollary~\ref{ex1col}), therefore proving cases (i) and (ii) of Theorem~\ref{mainthm}. However,  if $\chi_{cr}^*(H)<\chi_<(H)$ and $H$ does not have a local barrier then this minimum degree condition can be substantially lowered. To achieve this, we need a new absorbing result:

\begin{thm}[Absorbing theorem for non-local barriers]\label{absorbingthm}
Let $H$ be an ordered graph on $h$ vertices with $\chi_<(H)\geq 3$ and let $\eta>0$. If $H$ does not have a local barrier and $\chi_{cr}^*(H)<\chi_<(H)$, then there exists an $n_0\in\mathbb{N}$ and $0<\nu\ll\eta$ so that the following holds. Suppose that $G$ is an ordered graph on $n\geq n_0$ vertices and 
$$\delta(G)\geq\left(1-\frac{1}{\chi_<(H)-1}+\eta\right)n.$$
Then $V(G)=[n]$ contains a set $Abs$ so that
\begin{itemize}
\item $|Abs|\leq\nu n$;
\item $Abs$ is an $H$-absorbing set for every $W\subseteq V(G)\setminus Abs$ such that $|W|\in h\mathbb{N}$ and $|W|\leq\nu^3n$.
\end{itemize}
\end{thm}
Note that the statement of Theorem~\ref{absorbingthm} is false if one allows $\chi_<(H)=2$; indeed,  the conclusion of the theorem fails for  so-called divisibility barriers $H$.\footnote{More precisely, here \emph{divisibility barrier} means an ordered graph $H$ with $\chi_<(H)=2$ that satisfies Property B as defined in~\cite[Page 3]{blt}.}  However,
one can adapt our proof and relax the hypothesis of Theorem~\ref{absorbingthm} to  $\chi_<(H)\geq 2$ if one additionally assumes $H$ is not a divisibility barrier.
We will not do this in this paper, however, as~\cite[Theorem~1.9]{blt} already resolves the perfect $H$-tiling problem for ordered graphs $H$ with  $\chi_<(H)=2$.

We postpone the proof of Theorem~\ref{absorbingthm} to Section~\ref{sec:proofs}.
With Theorem~\ref{absorbingthm} at hand, we can now give 
the  proof of Theorem~\ref{mainthm}.

\begin{proofofmainthm}
First note that the lower bounds in parts (i)--(iii) of the theorem follow immediately from Corollary~\ref{ex2col} (for (i) and (iii)) and Corollary~\ref{ex1col} (for (ii)). Thus it remains to prove the upper bounds.

Let  $H$ be an ordered graph with $\chi_<(H)\geq 3$ and let $\eta>0$. Let $n\in\mathbb{N}$ be sufficiently large and such that $|H|$ divides $n$. Let $G$ be an ordered graph on $n$ vertices with minimum degree so that
\begin{itemize}
    \item[(i)] $  \delta(G)\geq\left (1-\frac{1}{\chi^*_{cr}(H)}+\eta\right )n $ \ \  if
    $\chi^*_{cr}(H) \geq \chi_<(H)$;
\item[(ii)] $  \delta(G)\geq\left (1-\frac{1}{\chi_<(H)}+\eta\right )n $  \ \  if 
 $\chi^*_{cr}(H) < \chi_<(H)$ and $H$ has a local barrier;
 \item[(iii)] $\delta(G)\geq\left (1-\frac{1}{\chi^*_{cr}(H)}+\eta\right )n $ \ \  if 
    $\chi^*_{cr}(H) < \chi_<(H)$ and $H$ has no local barrier.
\end{itemize}
Recall that $\chi^* _{cr}(H) > \chi_<(H)-1$. Thus,
by Theorem \ref{BalLiTre2} (for cases (i) and (ii)) and Theorem \ref{absorbingthm} (for case (iii)), there exists some $0<\nu\ll \eta$ and a set $Abs\subseteq V(G)$ such that
\begin{itemize}
\item $|Abs|\leq\nu n$;
\item $Abs$ is an $H$-absorbing set for every $W\subseteq V (G)\setminus Abs$ such that $|W|\in |H|\mathbb{N}$ and $|W|\leq\nu^3n$.
\end{itemize}
Let $G':=G\setminus Abs$. In all cases we have that
$\delta(G')\geq  (1-{1}/{\chi^*_{cr}(H)} )|G'| .$

Since $n$ was chosen to be sufficiently large, by Theorem~\ref{BalLiTre1} there exists an $H$-tiling $\mathcal{M}_1$ in $G'$ covering all but at most $\nu^3n$ vertices. Let $W\subseteq V(G')$ denote the set of vertices which are not covered by $\mathcal{M}_1$. Since $|H|$ divides $n$, $|V(\mathcal{M}_1)|$ and $|Abs|$ we have that $|H|$ divides $|W|$ too. Also, $|W|\leq\nu^3n$, hence $G'[W\cup Abs]$ contains a perfect $H$-tiling $\mathcal{M}_2$. Finally, observe that $\mathcal{M}_1\cup\mathcal{M}_2$ is a perfect $H$-tiling of $G$, as desired.\qed

\end{proofofmainthm}

\section{Bottlegraphs}\label{sec:bottle}
In the following proposition we show that
it suffices to consider bottlegraphs where  all parts are of the same size except for perhaps one smaller part.
\begin{prop}\label{propbottle}
Let $H$ be an ordered graph and
 $B$ be a bottlegraph of $H$. There exists a bottlegraph $B'$ of $H$ and an integer $m\in\mathbb{N}$ such that $\chi_{cr}(B')=\chi_{cr}(B)$ and all parts of $B'$ have size $m$ except for one part with size at most $m$.
\end{prop}

\begin{proof}
Let $B$ be a bottlegraph of $H$; so $B$ is a complete $k$-partite unordered graph with parts $B_1,\dots,B_k$ for some $k\in\mathbb{N}$. Without loss of generality, we may assume that $|B_1|\leq|B_i|$ for every $i>1$. Let $B'$ be the complete $k$-partite unordered graph with parts $B'_1,\dots,B'_k$ where
$$|B'_i|=\begin{cases}
(k-1)|B_1| &\text{ if $i=1$}\\
|B|-|B_1| &\text{ otherwise.}
\end{cases}$$
So $|B'|=(k-1)|B|$. Furthermore, we have
$$|B'_1|=(k-1)|B_1|=k|B_1|-|B_1|\leq|B|-|B_1|=|B'_i|$$
for every $i>1$. It follows that
$$\chi_{cr}(B')=\frac{(k-1)|B'|}{|B'|-|B'_1|}=\frac{(k-1)^2|B|}{(k-1)|B|-(k-1)|B_1|}=
\frac{(k-1)|B|}{|B|-|B_1|}=\chi_{cr}(B).$$
It remains to show that $B'$ is a bottlegraph of $H$. Observe that the vertices in $B'_1$ can be partitioned into $(k-1)$ sets of size $|B_1|$ while the vertices in $B'_i$ can be partitioned into $(k-1)$ sets of sizes $|B_2|,\dots,|B_k|$ respectively, for every $i>1$. This implies  that $B'$ contains a perfect $B$-tiling consisting of $(k-1)$ copies of $B$. 

Let $\{C_1,\dots,C_{k-1}\}$ be a perfect $B$-tiling in $B'$ (i.e., each $C_i$ is a copy of $B$ in $B'$). Let $\sigma$ be a permutation of $[k]$ and let $\phi$ be an interval labelling of $B'$ with respect to $\sigma$. For every $C_i$, $\phi$ induces an interval labelling $\phi_i$ of $C_i$ with respect to some permutation $\sigma_i$ of $[k]$. Since $B$ is a bottlegraph,  given any $i\in[k-1]$, there exists some $t_i\in\mathbb{N}$ such that the ordered blow-up $(C_i(t_i),\phi_i)$ contains a perfect $H$-tiling.
Set $t:=t_1t_2\dots t_{k-1}$. Then the ordered blow-up $(C_i(t),\phi_i)$ contains a perfect $H$-tiling $\mathcal{M}_i$, for each $i\in[k-1]$.

Finally, $\mathcal{M}_1\cup\cdots\cup\mathcal{M}_{k-1}$ is a perfect $H$-tiling of the ordered blow-up $(B'(t),\phi)$. Since $\sigma,\phi$ were arbitrary, $B'$ is a bottlegraph of $H$.
\end{proof}
Note that the notion of a bottlegraph of $H$ (Definition~\ref{bottlegraphdef}) and $1$-bottlegraph (Definition~\ref{defx}) are not quite the same.
However, the next result implies that $\chi^*_{cr}(H)=\chi^*_{cr}(1,H)$.
\begin{prop}\label{propeq}
Let $H$ be an ordered graph. Then $\chi^*_{cr}(H)=\chi^*_{cr}(1,H)$.
\end{prop}
\begin{proof}
Let $\mathcal X:= \{ \chi_{cr}(B) : \ B \text{ is a bottlegraph of } H\}$ and 
$\mathcal X_1:= \{ \chi_{cr}(B) : \ B \text{ is a $1$-bottlegraph of $H$}\}.$
Thus, $\inf \mathcal X= \chi^*_{cr} (H)$ and $\inf \mathcal X_1= \chi^*_{cr} (1,H)$.
By definition of a bottlegraph and $1$-bottlegraph we have that $\mathcal X_1 \subseteq \mathcal X$; so to prove the proposition it suffices to show that $\mathcal X \subseteq \mathcal X_1$.

Given any bottlegraph $B$ of $H$, let $B'$ be the bottlegraph of $H$ obtained by applying Proposition~\ref{propbottle}. So $B'$ satisfies conditions (i) and (ii) in the definition of a $1$-bottlegraph of $H$ and $\chi_{cr}(B')=\chi_{cr}(B)$. As $B'$ is a bottlegraph of $H$, there is some $t\in \mathbb N$ so that $B'(t)$ satisfies condition (iii) of the definition of a $1$-bottlegraph of $H$.
Then $B'(t)$ is a $1$-bottlegraph of $H$ with $\chi_{cr}(B'(t))=\chi_{cr}(B')=\chi_{cr}(B)$.
Thus, $\mathcal X \subseteq \mathcal X_1$, as desired.

\end{proof}

\section{The regularity lemma}\label{sec:regularitylemma}
In the proof of Theorem~\ref{absorbingthm} we will make use of the regularity method.
In this section we state a multipartite version of  Szemer{\'e}di's regularity lemma and some other related tools. First, we introduce some basic notation.

The \emph{density} of an (ordered) bipartite graph with vertex classes~$A$ and~$B$ is
defined to be
$$d(A,B):=\frac{e(A,B)}{|A||B|}.$$
Given $\eps>0$, a graph $G$ and two disjoint sets $A, B\subset V(G)$, we say that the pair $(A, B)_G$ (or simply $(A, B)$ when the underlying graph is clear) is \emph{$\eps$-regular} if for all sets
$X \subseteq A$ and $Y \subseteq B$ with $|X|\ge \eps |A|$ and
$|Y|\ge \eps |B|$, we have $|d(A,B)-d(X,Y)|< \eps$. 
Given $d\in [0, 1]$, the pair $(A, B)_G$ is \emph{$(\eps,d)$-regular} if $G$ is $\eps$-regular, and $d(A,B)\geq d$.

We now state some well-known properties of $\eps$-regular pairs.
The first (see, e.g.,~\cite[Fact~1.5]{KomlosSimonovits}) implies that
one can delete many vertices from an $(\eps,d)$-regular pair and still retain
such a regularity property.

\begin{lemma}[Slicing lemma]\label{slicinglemma}
Let $(A,B)_G$ be an $\eps$-regular pair of density $d$, and for some $\alpha>\eps$, let $A'\subseteq A$, $B'\subseteq B$ with $|A'|\geq\alpha|A|$ and $|B'|\geq\alpha|B|$. Then $(A', B')_G$ is $(\eps', d-\eps)$-regular with $\eps':=\max\{\eps/\alpha, 2\eps\}$.
\end{lemma}


The following theorem is a multipartite version of Szemer{\'e}di's regularity lemma~\cite{sze} (presented, e.g., as Lemma~5.5 in~\cite{blt}).

\begin{thm}[Multipartite regularity lemma]\label{multiregularity}
Given any integer $t\geq2$, any $\eps>0$ and any $\ell_0\in\mathbb{N}$ there exists $L_0=L_0(\eps, t, \ell_0)\in\mathbb{N}$ such that for every $d\in (0, 1]$ and for every nearly balanced $t$-partite graph $G=(W_1,\dots, W_t)$ on $n\geq L_0$ vertices, there exists an $\ell\in\mathbb{N}$, a partition $W^0_i, W^1_i,\dots, W^\ell _i$ of $W_i$ for each $i\in[t]$ and a spanning subgraph $G'$ of $G$, such that the following conditions hold:
\begin{enumerate}
\item $\ell_0\leq \ell\leq L_0$;
\item $d_{G'}(x)\geq d_G(x)-(d+\eps)n$ for every $x\in V(G)$;
\item $|W^0_i|\leq\eps n/t$ for every $i\in[t]$;
\item $|W^j_i| = |W^{j'}_{i'}|$ for every $i,i'\in[t]$ and $j,j'\in[\ell]$;
\item for every $i,i'\in[t]$ and $j,j'\in[\ell]$ either $(W^j_i, W^{j'}_{i'})_{G'}$ is an $(\eps,d)$-regular pair or $G'[W^j_i, W^{j'}_{i'}]$ is empty.
\end{enumerate}
\end{thm}

We call the $W^j_i$ \emph{clusters}, the $W^0_i$ the \emph{exceptional sets} and the vertices in the $W^0_i$ \emph{exceptional vertices}. We refer to $G'$ as the pure graph. The \emph{reduced graph $R$ of $G$ with parameters $\eps$, $d$ and $\ell_0$} is the graph whose vertices are the $W^j_i$ (where $i \in[t]$ and $j\in[\ell])$ and in which $W^j_i W^{j'}_{i'}$ is an edge precisely when $(W^j_i, W^{j'}_{i'})_{G'}$ is $(\eps, d)$-regular.
The following well-known corollary of the regularity lemma shows that the reduced graph almost inherits the minimum degree of the original graph.

\begin{prop}\label{reducedgraph}
Let $0<\eps,d,k<1$ and let $G$ be an $n$-vertex graph with $\delta(G)\geq kn$. If $R$ is the reduced graph of $G$ obtained by applying Theorem~\ref{multiregularity} with parameters $\eps,d,\ell_0$, then $\delta(R)\geq(k-2\eps-d)|R|$.
\end{prop}

A useful tool to embed subgraphs into $G$ using the reduced graph $R$ is the so-called \emph{key lemma}. 

\begin{lemma}[Key lemma~\cite{KomlosSimonovits}]\label{keylemma} 
Let $0<\eps<d$ and $q,t\in\mathbb{N}$. Let $R$ be a graph with $V(R)=\{v_1,\dots,v_k\}$. We construct a graph $G$ as follows: replace every vertex $v_i\in V(R)$ with a set $V_i$ of $q$ vertices and replace each edge of $R$ with an $(\eps,d)$-regular pair. For each $v_i\in V(R)$, let $U_i$
denote the set of $t$ vertices in $R(t)$ corresponding to $v_i$. Let $H$ be a subgraph of $R(t)$ on $h$ vertices with maximum degree $\Delta$. Set $\delta:=d-\eps$ and $\eps_0:=\delta^\Delta/(2+\Delta)$. If $\eps\leq\eps_0$ and $t-1\leq\eps_0q$ then there are at least $(\eps_0q)^h$ labelled copies of $H$ in $G$ so that if $x\in V(H)$ lies in $U_i$ in $R(t)$, then $x$ is embedded into $V_i$ in $G$.
\end{lemma}
As in~\cite{blt}, some of our applications of  Lemma~\ref{keylemma} will take the following form: suppose within an ordered graph $G$ we have vertex classes $V_1<\ldots<V_k$ so that each pair $(V_i,V_j)_G$ is $(\varepsilon,d)$-regular. Then Lemma~\ref{keylemma} tells us  $G$ contains many copies of any fixed size ordered graph $H$ with
 $\chi _< (H)=k$, where the $i$th vertex class of each such copy of $H$ is embedded into $V_i$.

\section{Absorbing tools}\label{sec:absorbingtools}
In this section we state a couple of results which are  useful for proving the existence of $H$-absorbing sets. The first  result is the following crucial
lemma of Lo and Markstr\"om~\cite{lo}; we present the ordered version of their result which appeared as Lemma~7.1 in~\cite{blt}.

\begin{lemma}[Lo and Markstr\"om~\cite{lo}]\label{absorbingtool} 
Let $h,s\in\mathbb{N}$ and $\xi>0$. Suppose that $H$ is an ordered graph on $h$ vertices. Then there exists an $n_0\in\mathbb{N}$ such that the following holds. Suppose that $G$ is an ordered graph on $n\geq n_0$ vertices so that, for any $x,y\in V(G)$, there are at least $\xi n^{sh-1}$
$(sh-1)$-sets $X\subseteq V(G)$ such that both $G[X\cup\{x\}]$ and $G[X\cup\{y\}]$ contain perfect $H$-tilings. Then $V(G)$ contains a set $M$ so that
\begin{itemize}
\item $|M|\leq(\xi/2)^hn/4$;
\item $M$ is an $H$-absorbing set for any $W\subseteq V (G)\setminus M$ such that $|W|\in h\mathbb{N}$ and $|W|\leq(\xi/2)^{2h}n/(32s^2h^3)$.
\end{itemize}
\end{lemma}
Informally, we will sometimes refer to a set $X$ satisfying the assumptions of Lemma~\ref{absorbingtool} as a \emph{chain} of size $|X|$ between vertices $x$ and $y$. The next lemma states that it is in some sense possible to concatenate chains.

\begin{lemma}\label{chainlemma}
Let $H$ be an (ordered) graph on $h$ vertices,
let $\alpha,\beta,\gamma >0$  and $s_1,s_2, n\in\mathbb{N}$ where
$$0<1/n \ll \gamma\ll\alpha,\beta, 1/h, 1/s_1, 1/s_2 .$$
Let $G$ be an (ordered) graph on $n$ vertices, $x,y\in V(G)$ and $A\subseteq V(G)\setminus\{x,y\}$ where $|A|\geq\alpha n$. Suppose that for every $z\in A$ there exist at least $\beta n^{s_1h-1}$ $(s_1h-1)$-sets $X\subseteq V(G)$ such that both $G[X\cup\{x\}]$ and $G[X\cup\{z\}]$ contain  perfect $H$-tilings and similarly there exist at least $\beta n^{s_2h-1}$ $(s_2h-1)$-sets $Y\subseteq V(G)$ such that both $G[Y\cup\{y\}]$ and $G[Y\cup\{z\}]$ contain  perfect $H$-tilings. Then there exist at least $\gamma n^{(s_1+s_2)h-1}$ $((s_1+s_2)h-1)$-sets $Z\subseteq V(G)$ such that both $G[Z\cup\{x\}]$ and $G[Z\cup\{y\}]$ contain  perfect $H$-tilings.
\end{lemma}

\begin{proof}
Let $z\in A$ and let $X,Y$ be an $(s_1h-1)$-set and an $(s_2h-1)$-set respectively which satisfy the above properties, with $X,Y$ disjoint so that $y \not \in X$ and $x \not \in Y$. Then $X\cup\{z\}\cup Y$ is an $((s_1+s_2)h-1)$-set and both $G[(X\cup\{z\}\cup Y)\cup\{x\}]$ and $G[(X\cup\{z\}\cup Y)\cup\{y\}]$ contain  perfect $H$-tilings. Thus, it is enough to lower bound the number of such triples $(z,X,Y)$. There are at least $\alpha n$ choices for $z$. Given a fixed choice of $z$ there are at least $\beta n^{s_1h-1}-n^{s_1h-2}\geq \beta n^{s_1h-1}/2$ suitable choices for $X$. For a fixed $X$ there are at least $\beta n^{s_2h-1}-s_1hn^{s_2h-2}\geq \beta n^{s_2h-1}/2$ choices for $Y$. Therefore,
in total there are at least 
$$\frac{(\alpha n)\cdot(\beta n^{s_1h-1}/2)\cdot(\beta n^{s_2h-1}/2)}{((s_1+s_2)h-1)!}\geq\gamma n^{(s_1+s_2)h-1}$$
choices for $Z$, where the denominator here is because different choices of the triple $(z,X,Y)$ can yield the same set $Z$.
\end{proof}




\section{Flexible colouring}\label{sec:usefulresults}
In this section we prove several auxiliary results which will be particularly important for the proofs of Theorem~\ref{absorbingthm} and Theorem~\ref{aptsthm}. We start with the following definition.

\begin{define}\label{flexibilitydef}
Let $H$ be an ordered graph and let $r:=\chi_<(H)$. We say that $H$ is \emph{flexible} if, for every $i\in[r-1]$, there exists an interval $(r+1)$-colouring
$$V_1<\cdots<V_i<\{x\}<V_{i+1}<\cdots<V_r$$
of $H$ such that both $V_1<\cdots<V_i\cup\{x\}<V_{i+1}<\cdots<V_r$ and $V_1<\cdots<V_i<V_{i+1}\cup\{x\}<\cdots<V_r$ are interval $r$-colourings of $H$.
\end{define}

In the next lemma we show that given a flexible ordered graph $H$, there exists a complete $\chi_<(H)$-partite ordered graph $F$ such that $F$ contains a perfect $H$-tiling and any complete $r$-partite ordered graph $F'$, whose parts have approximately the same sizes as the corresponding parts in $F$, contains a perfect $H$-tiling too. 

\begin{lemma}\label{flexibilitylemma1}
Let $H$ be an ordered graph on $h$ vertices and let $r:=\chi_<(H)$. If $H$ is flexible then the following holds. For every $i\in[r-1]$, let $V_1^i<\cdots<V_i^i<\{x_i\}<V_{i+1}^i<\cdots<V_r^i$ be an interval $(r+1)$-colouring of $H$ as described in Definition~\ref{flexibilitydef}. Let $F$ be the complete $r$-partite ordered graph with parts $F_1<\cdots<F_r$ such that 
$$|F_k|=\begin{cases}
rh+2rh\cdot\sum\limits_{i=1}^{r-1}|V_k^i| \text{\quad for $k\in\{1,r\};$} \\
2rh+2rh\cdot\sum\limits_{i=1}^{r-1}|V_k^i| \text{\quad otherwise.}
\end{cases}$$
Thus,  $|F|=2r(r-1)h^2$.
Let $s_1,\dots,s_r\in\mathbb{Z}$ such that $s_1+\cdots+s_r=0$ and $|s_i|\leq h$ for every $i\in[r]$. Let $F'$ be the complete $r$-partite ordered graph with parts $F'_1<\cdots<F'_r$ such that $|F'_k|=|F_k|+s_k$ for every $k\in[r]$. Then both $F$ and $F'$ contain  perfect $H$-tilings.
\end{lemma}

\begin{proof}
First, we explicitly construct a perfect $H$-tiling in $F$. For every $i\in[r-1]$, consider $rh$ copies of the interval $r$-colouring $V_1^i<\cdots<V_i^i\cup\{x_i\}<V_{i+1}^i<\cdots<V_r^i$ of $H$ and $rh$ copies of the interval $r$-colouring  $V_1^i<\cdots<V_i^i<V_{i+1}^i\cup\{x_i\}<\cdots<V_r^i$ of $H$. For every $k\in[r]$ we take the union of all the $k$th colour classes from the interval $r$-colourings above to obtain $r$ new classes. It is easy to check that the size of the new $k$th class is exactly the size of $F_k$, so we have just constructed a perfect $H$-tiling in $F$. In particular, as there are $2r(r-1)h$ copies of $H$ in this tiling, $|F|=2r(r-1)h^2$.

Note that for every pair of consecutive classes $(F_k,F_{k+1})$ we could independently move $rh$ vertices from $F_k$ to $F_{k+1}$ to yield a new complete $r$-partite ordered graph which still contains a perfect $H$-tiling; similarly, we could move $rh$ vertices from $F_{k+1}$ to $F_k$ (these $2rh$ vertices fulfill the role of $x_k$ in their respective interval $r$-colourings of $H$). We are going to use this observation to construct a perfect $H$-tiling in $F'$. 

Set $t_k:=s_1+\cdots+s_k$ for $k\in[r]$ and $t_0:=0$. For every pair of consecutive classes $(F_k,F_{k+1})$, if $t_k\geq0$ move $t_k$ vertices from $F_{k+1}$ to $F_k$, otherwise move $-t_k$ vertices from $F_k$ to $F_{k+1}$. This is possible since $|t_k|\leq|s_1|+\cdots+|s_r|\leq rh$ by assumption. Note that the size of the new $k$th class is $|F_k|+t_k-t_{k-1}=|F_k|+s_k=|F'_k|$, hence we just constructed a perfect $H$-tiling in $F'$.
\end{proof}

Observe that the ordered graphs $F$ and $F'$ in Lemma~\ref{flexibilitylemma1} have the same number of vertices. In the next corollary we ease this restriction and allow $F'$ to have a few more vertices than $F$.



\begin{col}\label{flexibilitycol1}
Let $H$ be an ordered graph on $h$ vertices and let $r:=\chi_<(H)$. If $H$ is flexible then the following holds. Let $F$ be the complete $r$-partite ordered graph with parts $F_1<\cdots<F_r$ as in Lemma~\ref{flexibilitylemma1} and let $t\in\mathbb{N}$. For any $s_1,\dots,s_r,\ell\in\mathbb{N}\cup \{0\}$ such that $s_1+\cdots+s_r=\ell h\leq th$ and any complete $r$-partite ordered graph $F'$ with parts $F'_1<\cdots<F'_r$ of size $|F'_k|=t|F_k|+s_k$ for every $k\in[r]$, both  $F(t)$ and $F'$ contain perfect $H$-tilings.
\end{col}

\begin{proof}
By Lemma~\ref{flexibilitylemma1}, $F$ contains a perfect $H$-tiling and thus clearly  the ordered blow-up $F(t)$ contains a perfect $H$-tiling too.

We prove that $F'$ contains a perfect $H$-tiling by induction on $\ell$. If $\ell=0$ then $F'=F(t)$ and so $F'$ contains a perfect $H$-tiling.

Assume $\ell>0$. For every $k\in[r]$, let $s'_k\in\mathbb{N}\cup \{0\}$ such that $s'_k\leq s_k$ and $s'_1+\cdots+s'_r=h$. Let $Q_1<\cdots<Q_r$ be an interval $r$-colouring of $H$. Notice that $V(F')$ can be partitioned into four sets $X,Y,W,Z$ such that
$$|X\cap F'_k|=(t-\ell)|F_k|,\quad|Y\cap F'_k|=(\ell-1)|F_k|+(s_k-s'_k),\quad|W\cap F'_k|=|Q_k|,\quad|Z|=|F_k|+s'_k-|Q_k|,$$
for every $k\in[r]$. Observe that
\begin{itemize}
    \item If non-empty then $F'[X]$ is a copy of $F(t-\ell)$, thus $F'[X]$ contains a perfect $H$-tiling.
    \item If $\ell=1$ then $s'_k=s_k$ for every $k\in[r]$ and so $F'[Y]$ is the empty graph. Otherwise, $F'[Y]$ contains a perfect $H$-tiling by the inductive hypothesis since $(s_1-s'_1)+\cdots+(s_r-s'_r)=(\ell-1)h$
    and $s_k-s'_k\geq0$ for every $k\in[r]$.
    \item $F'[W]$ spans a copy of $H$.
    \item $F'[Z]$ contains a perfect $H$-tiling by Lemma~\ref{flexibilitylemma1} since 
    $$\sum_{k=1}^r (s'_k-|Q_k|)=h-h=0$$
    and $|s'_k-|Q_k||\leq h$ for every $k\in[r]$.
\end{itemize}
Altogether this clearly implies $F'$ contains a perfect $H$-tiling.
\end{proof}

Our goal now is to show that if $\chi_{cr}^*(H)<\chi_<(H)$ then $H$ is flexible. This property is crucial 
for the proof of Theorem~\ref{absorbingthm}, and is a corollary of the following lemma.

\begin{lemma}\label{flexibilitylemma2}
Let $H$ be an ordered graph with vertex set $[h]$ and let $r:=\chi_<(H)$. If $H$ is not flexible, there exist some $i\in[r-1]$ such that the number of vertices lying in the first $i$ intervals of any given interval $r$-colouring of $H$ is fixed.
\end{lemma}
\begin{proof}
Since $H$ is not flexible, there exist some $i\in[r-1]$ such that there is no interval $(r+1)$-colouring
$V_1<\cdots<V_i<\{x\}<V_{i+1}<\cdots<V_r$
of $H$ such that both $V_1<\cdots<V_i\cup\{x\}<V_{i+1}<\cdots<V_r$ and $V_1<\cdots<V_i<V_{i+1}\cup\{x\}<\cdots<V_r$ are interval $r$-colourings of $H$.

Let $U_1<\cdots<U_r$ be an interval $r$-colouring of $H$. Let $y$ be the  largest vertex in $U_i$ and let $z$ be the smallest vertex in $U_{i+1}$; so  $z=y+1$. Note that $y$ must be adjacent to some vertex in $U_{i+1}$, as otherwise the interval $(r+1)$-colouring $U_1<\cdots<U_i\setminus\{y\}<\{y\}<U_{i+1}<\cdots<U_r$ satisfies the flexibility property, a contradiction. Let $y'$ be the  smallest vertex in $U_{i+1}$ adjacent to $y$. Similarly, let $z'$ be the  largest vertex in $U_i$ adjacent to $z$.

\begin{claim}
Given any $r$-colouring $Q_1<\cdots<Q_r$ of $H$, $y \in Q_i$ and $z \in Q_{i+1}.$
\end{claim}
Fix an arbitrary interval $r$-colouring $Q_1<\cdots<Q_r$ of $H$; say that $z'$ lies in the $k$th interval $Q_k$. Note  that  $Q_1<\cdots<Q_k\cap[z']<(U_i\setminus[z'])<U_{i+1}<\cdots<U_r$ is an interval colouring of $H$. If (i) $k<i-1$ or (ii) $k=i-1$ and $z'=y$ then $\chi_<(H)<r$, a contradiction. If $k=i-1$ and $z'\not=y$ then the interval $(r+1)$-colouring $Q_1<\cdots<Q_{i-1}\cap[z']<(U_i\setminus[z'])<\{z\}<U_{i+1}\setminus\{z\}<\cdots<U_r$ satisfies the flexibility property, a contradiction to our initial assumption in the proof. Thus, $k\geq i$, which implies that $z$ is contained in the $(i+1)$th interval $Q_{i+1}$ or above since $z$ and $z'$ are adjacent. Similarly, we can show that $y$ is contained in the $i$th interval $Q_i$ or below. This implies that, since $y,z$ are consecutive vertices in $H$,  $y$ lies in $Q_i$ and $z$ in $Q_{i+1}$. This completes the proof of the claim.

\smallskip

By the claim above, and as $z=y+1$, we conclude that $y$ is the largest vertex in $Q_i$. Hence, the number of vertices in the first $i$ intervals of any given interval $r$-colouring of $H$ is exactly $y$, yielding the required result.
\end{proof}

\begin{col}\label{flexibilitycol2}
Let $H$ be an ordered graph with vertex set $[h]$.  If $\chi_{cr}^*(H)<\chi_<(H)$, then $H$ is flexible.
\end{col}
\begin{proof}
Let $r:=\chi_<(H)$.
Suppose for a contradiction that $H$ is not flexible. By Lemma~\ref{flexibilitylemma2}, there exist some $i\in[r-1]$ such that the number of vertices lying in the first $i$ intervals of any given interval $r$-colouring of $H$ is, say, $y$ for some fixed $y\in\mathbb{N}$. Consider a bottlegraph $B$ of $H$ with $\chi_{cr}(B)<r$ (which exists by definition of $\chi^*_{cr}(H)$ and as $\chi^*_{cr}(H)<r$); note that $B$ must consist of precisely $r$ parts. By Proposition~\ref{propbottle} we may assume that $B$ has parts $B_1,\dots,B_r$ where $|B_1|<m$ and $|B_i|=m$ for some $m\in\mathbb{N}$. Let $\sigma$ be a permutation of $[r]$ and $\phi$ be an interval labelling of $B$ with respect to $\sigma$. Recall that the ordered graph $(B,\phi)$ has parts $B_{\sigma^{-1}(1)}<\cdots<B_{\sigma^{-1}(r)}$. Then, there exists some $t\in\mathbb{N}$ such that the ordered blow-up $(B(t),\phi)$  contains a perfect $H$-tiling. In particular, this perfect $H$-tiling consists of $t|B|/h$ copies of $H$. By assumption, each copy of $H$ has exactly $y$ vertices lying in the first $i$ parts of $(B(t),\phi)$ which implies that the first $i$ parts contain exactly $yt|B|/h$ vertices; that is, the total number of vertices in these parts is independent of the choice of $\sigma$. This is clearly a contradiction
as if we pick $\sigma$ such that $\sigma^{-1}(1)=1$ then there are fewer than $itm$ vertices in the first $i$ parts of $(B(t),\phi)$, while if $\sigma^{-1}(r)=1$ then there are precisely 
$itm$ vertices in the first $i$ parts of $(B(t),\phi)$.
\end{proof}
Note that one does really require  $\chi^*_{cr}(H)<\chi _<(H)$ in the statement of Corollary~\ref{flexibilitycol2}; consider, e.g., when $H$ is a complete balanced $r$-partite ordered graph with  parts of size at least $2$.

\section{Proof sketch of Theorem~\ref{absorbingthm}}\label{sec:proofsketch}
In this section we briefly present the main ideas behind the proof of Theorem~\ref{absorbingthm} before giving a  rigorous argument in Section~\ref{sec:proofabsorbingthm}. Throughout this section we set $r:=\chi_<(H)\geq 3$.

Let $H$ be an ordered graph such that $\chi_{cr}^*(H)<\chi_<(H)=r$ and $H$ has no local barrier. Let $G$ be an $n$-vertex graph with $n$ sufficiently large and minimum degree 
$$\delta(G)\geq\left(1-\frac{1}{r-1}+o(1)\right)n.$$
Our ultimate goal is to construct many chains between any given pair of vertices $x,y\in V(G)$; then Lemma~\ref{absorbingtool} will conclude the proof.

To achieve this, we first divide $[n]$ into many nearly balanced intervals, remove all edges in $G$ which lie completely in some interval and then apply the multipartite version of the regularity lemma to obtain a reduced graph $R$. This preconditioning process will make it convenient to work with cliques in the reduced graph: if two clusters $W$ and $W'$ are adjacent in $R$ then by construction either $W<W'$ or $W>W'$.

We then show that given an arbitrary pair of clusters $W$ and $W'$ in $R$, for \emph{almost} every pair of vertices $x\in W$ and $y\in W'$ we can find many chains between $x$ and $y$. We achieve this gradually through various steps:
\begin{itemize}
\item Given a copy $T$ of $K_r$ in the reduced graph $R$ and an arbitrary cluster $W$ in $T$, we prove that for \emph{almost} every pair of vertices $x,y\in W$ we can find many chains between $x$ and $y$. This is quite straightforward and the only property of $H$ we use is that $\chi_<(H)=r$.
\item Given a copy $T$ of $K_r$ in the reduced graph $R$ and two arbitrary clusters $W,W'$ in $T$, we prove that for \emph{almost} every pair of vertices $x\in W$ and $y\in W'$ we can find many chains between $x$ and $y$. The ``flexibility" guaranteed by the condition $\chi_<(H)<\chi_{cr}^*(H)$, formally stated in Corollary~\ref{flexibilitycol2}, will be the main ingredient here.
\item Finally, we show that given any two arbitrary clusters $W$ and $W'$, for \emph{almost} every pair of vertices $x\in W$ and $y\in W'$ we can find many chains between $x$ and $y$. Since $r\geq 3$, this fairly straightforwardly follows from the minimum degree of $R$ and the previous point. 
\end{itemize} 
Next, given an arbitrary vertex $x\in V(G)$, we use the minimum degree condition to find a particular structure $\mathcal{L}$ in $G$ containing $x$ and resembling an extremal construction for graphs which have a local barrier, namely  Extremal Example~\ref{ex1}; in particular, the vertex classes of Extremal Example~\ref{ex1} are replaced by  some of the clusters in $R$. Assuming $H$ does not have a local barrier, and using Corollary~\ref{flexibilitycol1}, we find many chains between $x$ and \emph{almost} every vertex lying in a certain cluster $W$ in $\mathcal{L}$. 

Analogously, given an arbitrary $y\in V(G)\setminus\{x\}$, we find many chains between $y$ and \emph{almost} every vertex lying in a certain cluster $W'$. By the previous steps, there exist many chains between almost every vertex in $W$ and almost every vertex in $W'$. By concatenation, this ensures many chains  between $x$ and $y$ and so the hypothesis of Lemma~\ref{absorbingtool} is  satisfied, thereby yielding Theorem~\ref{absorbingthm}.

\section{Proof of Theorem~\ref{absorbingthm} and Theorem~\ref{thmcover}}\label{sec:proofs}

\subsection{Proof of Theorem~\ref{absorbingthm}}\label{sec:proofabsorbingthm}
Let $H$ be an ordered graph on $h$ vertices such that $\chi_{cr}^*(H)<\chi_<(H)$, $\chi_<(H)\geq 3$ and $H$ has no local barrier. Set $r:=\chi_<(H)$.

Given any $\eta >0$, define additional constants
 $d,\varepsilon_1,\varepsilon_2,\varepsilon_3\in\mathbb{R}$ and $\ell_0\in\mathbb{N}$ so that 
\begin{align}\label{hier1}
0<1/\ell_0\ll\varepsilon_1\ll\varepsilon_2\ll\varepsilon_3\ll d\ll\eta, 1/(rh).
\end{align}
Set $t:=\lceil 4/\eta\rceil$ and let $L_0=L_0(\varepsilon_1,t,\ell_0)\geq\ell_0$ be as in Theorem~\ref{multiregularity}. Additionally, let $\xi_1,\xi_2,\xi_3,\xi_4, \xi_5\in\mathbb{R}$ and $n\in\mathbb{N}$ so that
\begin{align}\label{hier2}
0<1/n\ll\xi_5 \ll \xi_4\ll\xi_3\ll\xi_2\ll\xi_1\ll 1/L_0.
\end{align}
Define $\nu:= (\xi_5/2)^h/4$ and $s:=2r(r-1)h+1$.

Let $G$ be an ordered graph on $n\geq L_0$ vertices with minimum degree
$$\delta(G)\geq\left(1-\frac{1}{r-1}+\eta\right)n.$$ 
Let $W_1<\cdots<W_t$ be a nearly balanced interval partition of $[n]$. Let $G'$ be the ordered graph obtained by deleting all edges lying in each $W_i$ from $G$. As
$t=\lceil 4/\eta \rceil$ we have that $|W_i|\leq \lceil \eta n/4 \rceil$ for all $i \in [t]$ and so
\begin{align}\label{minG}
  \delta(G')\geq\left(1-\frac{1}{r-1}+\frac{\eta}{2}\right)n.   
\end{align}

Apply Theorem~\ref{multiregularity} to $G'$ with parameters $\varepsilon_1,t,\ell_0$ and $d$ to obtain a pure graph $G''$ and a partition $W^0_i,\dots,W^\ell_i$ for every $W_i$, where $\ell_0\leq \ell\leq L_0$. All non-exceptional clusters $W^j_i$ have size $m$ where 
\begin{align}\label{hier3}
m\geq\frac{1}{\ell}\left(\left\lfloor\frac{n}{t}\right\rfloor-\frac{\varepsilon_1 n}{t}\right)\geq\frac{\eta n}{5L_0}.
\end{align}
Let $R$ be the corresponding reduced graph of $G'$. By (\ref{minG}), Proposition~\ref{reducedgraph} implies that 
\begin{align}\label{eq2}
\delta(R)\geq\left(1-\frac{1}{r-1}+\frac{\eta}{3}\right)|R|.
\end{align}
Crucially, observe that if $W_i^jW_{i'}^{j'}$ is an edge in $R$ then, by construction of $G'$, $i\not=i'$ and so either $W_i^j<W_{i'}^{j'}$ or $W_i^j>W_{i'}^{j'}$.  

First, we prove that given a copy $T$ of $K_r$ in $R$ and an arbitrary cluster $W$ in $T$, for \emph{almost} every pair of vertices $x,y\in W$ there exist many chains between $x$ and $y$. 

\begin{claim}\label{claim1}
Let $T_1<\cdots<T_r$ be $r$ clusters which form a copy of $K_r$ in $R$.  Given any $i\in[r]$, there exists a set $A_i\subseteq T_i$ of size $|A_i|\geq(1-\varepsilon_1 r)|T_i|$ such that the following holds: for every $x\in A_i$ there exists a set $C_x\subseteq T_i$ of size $|C_x|\geq(1-2\varepsilon_2 r)|T_i|$ such that for every $y\in C_x$ there are at least $\xi_1 n^{h-1}$ $(h-1)$-sets $X\subseteq V(G)$ so that both $G[X\cup\{x\}]$ and $G[X\cup\{y\}]$ contain a copy of $H$.
\end{claim}

\begin{proof}
Fix $i\in[r]$. For every $k\not=i$, let 
$$L_k:=\{v\in T_i:d_{G''}(v,T_k)\leq(d-\varepsilon_1)|T_k|\}.$$
Observe that $d_{G''}(L_k,T_k)\leq d-\varepsilon_1$ for every $k\not=i$. Suppose $|L_{k_0}|\geq\varepsilon|T_i|$ for some $k_0\not=i$. Since $(T_i,T_{k_0})_{G''}$ is $(\varepsilon_1,d)$-regular then 
$$|d_{G''}(L_{k_0},T_{k_0})-d_{G''}(T_i,T_{k_0})|<\varepsilon_1$$
and so $d_{G''}(L_{k_0},T_{k_0})>d-\varepsilon_1$, a contradiction. Thus $|L_k|<\varepsilon_1|T_i|$ for every $k\not=i$. In particular, 
$$|L_1\cup\cdots\cup L_r|<\varepsilon_1 (r-1)|T_i|<\varepsilon_1 r|T_i|.$$
Let $A_i:=T_i\setminus(L_1\cup\cdots\cup L_r)$. It follows from the previous inequality that
$$|A_i|\geq(1-\varepsilon_1 r)|T_i|.$$ 

Let $x\in A_i$. Define $T'_k:=T_k\cap N_{G''}(x)$ for every $k\not=i$ and $T'_i:=T_i\setminus\{x\}$. By the slicing lemma (Lemma~\ref{slicinglemma}), the pair $(T'_a,T'_b)_{G''}$ is $(\varepsilon_2,d/2)$-regular for every  $a\not=b\in[r]$. For every $k\not=i$, let 
$$L'_k:=\{v\in T'_i:d_{G''}(v,T'_k)\leq(d/2-\varepsilon_2)|T'_k|\}.$$ 
By the same argument as before, 
$$|L'_1\cup\cdots\cup L'_r|<\varepsilon_2 r|T'_i|.$$
Let $C_x:=T'_i\setminus(L'_1\cup\cdots\cup L'_r)$. The above inequality implies that
$$|C_x|\geq(1-\varepsilon_2 r)|T'_i|\geq(1-2\varepsilon_2 r)|T_i|.$$

Let $y\in C_x$. Define $T''_k:=T'_k\cap N_{G''}(y)=T_k \cap N_{G''}(x) \cap N_{G''}(y)$ for every $k\not=i$ and $T''_i:=T'_i\setminus\{y\}=T_i \setminus \{x,y\}$. By the slicing lemma, the pair $(T''_a,T''_b)_{G''}$ is $(\varepsilon_3,d/3)$-regular for every  $a\not=b\in[r]$. Furthermore, by construction, $x$ and $y$ are adjacent to all vertices in $T''_k$ for every $k\not=i$.

Let $V_1<\cdots<V_r$ be an interval $r$-colouring of $H$. Pick any $v\in V_i$  and let $H'$ be the complete $r$-partite ordered graph
with parts $V_1<\cdots<V_i\setminus\{v\}<\cdots<V_r$. Using (\ref{hier1}), (\ref{hier2}) and (\ref{hier3}), the key lemma (Lemma~\ref{keylemma}) implies that there exist at least $\xi_1 n^{h-1}$ vertex sets  $X$ in $G$ so that $G[X]$ spans a copy of $H'$ where  $V_k\subseteq T''_k$ for every $k\not=i$ and $V_i\setminus\{v\}\subseteq T''_i$. Because $x$ and $y$ are adjacent to all vertices in $T''_k$ for every $k\not=i$, both $G[X\cup\{x\}]$ and $G[X\cup\{y\}]$ are spanned by a copy of $H$.
\end{proof}

We now prove that given a copy $T$ of $K_r$ in the reduced graph $R$ and two arbitrary clusters $W$ and $W'$ in $T$, for \emph{almost} every pair of vertices $x\in W$ and $y\in W'$ there exist many chains between $x$ and $y$. Recall $s=2r(r-1)h+1$.

\begin{claim}\label{claim2}
Let $T_1<\cdots<T_r$ be $r$ clusters which form a copy of $K_r$ in $R$. 
For every $i,j\in[r]$, there exist sets $A_i\subseteq T_i$ and $A_j\subseteq T_j$ of size $|A_i|\geq(1-\varepsilon_1 r)|T_i|$ and $|A_j|\geq(1-\varepsilon_1 r)|T_j|$ such that for every $x\in A_i$ and $y\in A_j$ there are at least $\xi_2 n^{sh-1}$ $(sh-1)$-sets $X\subseteq V(G)$ for which both $G[X\cup\{x\}]$ and $G[X\cup\{y\}]$ contain perfect $H$-tilings.
\end{claim}

\begin{proof}
Fix $i,j\in[r]$. As in the proof of Claim~\ref{claim1}, for every $k\not=i$, we define
$$L_k:=\{v\in T_i:d_{G''}(v,T_k)\leq(d-\varepsilon_1)|T_k|\}$$
and set  $A_i:=T_i\setminus(L_1\cup\cdots\cup L_r)$.
In the proof of Claim~\ref{claim1} we saw that 
$$|L_1\cup\cdots\cup L_r|<\varepsilon_1 (r-1)|T_i|<\varepsilon_1 r|T_i| \ \
\text{and thus} \ \
|A_i|\geq(1-\varepsilon_1 r)|T_i|.$$ 

Let $x\in A_i$.  As in the proof of Claim~\ref{claim1}, we define $T'_k:=T_k\cap N_{G''}(x)$ for every $k\not=i$ and $T'_i:=T_i\setminus\{x\}$. The pair $(T'_a,T'_b)_{G''}$ is $(\varepsilon_2,d/2)$-regular for every  $a\not=b\in[r]$. 

For every $k\not=j$, let 
$$L'_k:=\{v\in T'_j:d_{G''}(v,T'_k)\leq(d/2-\varepsilon_2)|T'_k|\}.$$ 
Note that $L'_k$ is not quite the same as the corresponding set given in the proof 
of Claim~\ref{claim1} (it is defined it terms of $j$ not $i$); however,
as in the proof of Claim~\ref{claim1} we have that
$$
|L'_1\cup\cdots\cup L'_r|<\varepsilon_2 r|T'_j|.$$
Let $C:=T'_j\setminus(L'_1\cup\cdots\cup L'_r)$. The previous inequality implies that
\begin{align}\label{Cbound}
|C|\geq(1-\varepsilon_2 r)|T'_j|\geq(1-\varepsilon_2 r)(d-\varepsilon_1)|T_j|\stackrel{(\ref{hier1})}{\geq}\frac{d}{2}|T_j|.
\end{align}

Let $z\in C$. Define  $T''_k:=T'_k\cap N_{G''}(z)$ for every $k\not=j$ and $T''_j:=T'_j\setminus\{z\}$. By the slicing lemma and (\ref{hier1}), the pair $(T''_a,T''_b)_{G''}$ is $(\varepsilon_3,d/3)$-regular for every pair $a\not=b\in[r]$. Furthermore, by construction, $x$ is adjacent to all vertices in $T''_k$ for every $k\not=i$ while $z$ is adjacent to all vertices in $T''_k$ for every $k\not=j$.

Since $\chi_{cr}^*(H)<\chi_<(H)$, by Corollary~\ref{flexibilitycol2} $H$ is flexible. Let $F$ be the complete $r$-partite ordered graph with parts $F_1<\cdots<F_r$ as defined in Lemma~\ref{flexibilitylemma1}. Recall that $|F|=2r(r-1)h^2=(s-1)h$. Pick any $v\in F_i$ and let $F^*$ be the ordered graph obtained by removing the vertex $v$ from $F$. In particular, $F^*$ is a complete $r$-partite ordered graph with parts $F_1<\cdots<F_i\setminus\{v\}<\cdots<F_r$. By (\ref{hier1}), (\ref{hier2}) and (\ref{hier3}), the key lemma (Lemma~\ref{keylemma}) implies that there exist at least $\xi_1 n^{|F|-1}$ vertex sets  $Y$ in $G$ so that $G[Y]$ spans a copy of $F^*$ where $F_k\subseteq T''_k$ for every $k\not=i$ and $F_i\setminus\{v\}\subseteq T''_i$. 

Given any such set $Y$,
since $x$ is adjacent to all vertices in $T''_k$ (for all $k\not=i$) then $G[Y\cup\{x\}]$ is spanned by a copy of $F$ and thus it contains a perfect $H$-tiling. Similarly, since $z$ is adjacent to all vertices in $T''_k$ (for all $k\not=j$) then $G[Y\cup\{z\}]$ is spanned by a complete $r$-partite ordered graph $F'$ on $|F|$ vertices where each part of $F'$ differs in size by at most one compared to the corresponding part in $F$. Thus, Lemma~\ref{flexibilitylemma1} implies $G[Y\cup\{z\}]$ contains a perfect $H$-tiling. 

Let $A_j\subseteq T_j$ be as in the statement of Claim~\ref{claim1}; so $|A_j|\geq(1-\varepsilon_1 r)|T_j|$. Let $y\in A_j$. By Claim~\ref{claim1}, there exists a set $C_y\subseteq T_j$ of size 
\begin{align}\label{Cybound}
|C_y|\geq(1-2\varepsilon_2 r)|T_j|
\end{align}
such that for any $z\in C_y$ there exist at least $\xi_1 n^{h-1}$ $(h-1)$-sets $Z\subseteq V(G)$ so that both $G[Z\cup\{y\}]$ and $G[Z\cup\{z\}]$ contain spanning copies of $H$. 

In summary, given any $x \in A_i$ and $y \in A_j$, we have shown that for every $z\in C\cap C_y$ there exist at least $\xi_1 n^{|F|-1}$ $(|F|-1)$-sets $Y$ such that both $G[Y\cup\{x\}]$ and $G[Y\cup\{z\}]$ contain perfect $H$-tilings and similarly there exist at least $\xi_1 n^{h-1}$ $(h-1)$-sets $Z$ such that both $G[Z\cup\{y\}]$ and $G[Z\cup\{z\}]$ contain perfect $H$-tilings. Furthermore, as $C,C_y\subseteq T_j$ note that
$$|C\cap C_y|=|C|+|C_y|-|C\cup C_y|\geq|C|+|C_y|-|T_j|\stackrel{(\ref{hier1}),(\ref{Cbound}),(\ref{Cybound}) }{\geq}\varepsilon_1|T_j|=\varepsilon_1 m\stackrel{(\ref{hier3})}{\geq}\frac{\varepsilon_1\eta n}{5L_0}.$$
 Applying Lemma~\ref{chainlemma} with $C \cap C_y,  \varepsilon_1\eta /({5L_0}),\xi_1, \xi_2$ playing the roles of $A,  \alpha, \beta , \gamma$, we conclude that 
 there exist at least $\xi_2 n^{|F|+h-1}=\xi_2 n^{sh-1}$ $(sh-1)$-sets $X\subseteq V(G)$ such that both $G[X\cup\{x\}]$ and $G[X\cup\{y\}]$ contain perfect $H$-tilings, as desired.
\end{proof}

In the next claim we show that given any arbitrary pair of clusters $W$ and $W'$ in the reduced graph $R$, for \emph{almost} every pair of vertices $x\in W$ and $y\in W'$ there exist many chains between $x$ and $y$. The assumption $r\geq 3$ is crucial here.

\begin{claim}\label{claim3} 
For every pair of clusters $W$ and $W'$ in $R$, there exist sets $A\subseteq W$ and $A'\subseteq W'$ of size $|A|\geq(1-\varepsilon_1 r)|W|$ and $|A'|\geq(1-\varepsilon_1 r)|W'|$ satisfying the following: for every pair of vertices $x\in A$ and $y\in A'$ there are at least $\xi_3 n^{2sh-1}$ $(2sh-1)$-sets $X\subseteq V(G)$,  such that both $G[X\cup\{x\}]$ and $G[X\cup\{y\}]$ contain perfect $H$-tilings.
\end{claim}

\begin{proof}
Recall the reduced graph $R$ has minimum degree
$$\delta(R)\geq\left(1-\frac{1}{r-1}+\frac{\eta}{3}\right)|R|.$$
Using the above minimum degree condition, given any two adjacent clusters in $R$, one can greedily construct  a copy of $K_r$ in $R$ containing them both.

Since $r\geq 3$, $\delta(R)>|R|/2$ and so the clusters $W$ and $W'$ have a common neighbour $U$ in $R$. Let $K$ and $K'$ be two copies of $K_r$ in $R$ containing $W,U$ and $W',U$ respectively. 

Apply Claim~\ref{claim2} with $K$, $W$, $U$ playing the roles of $K_r$, $T_i$ and $T_j$ 
to obtain 
sets $A\subseteq W$ and $D\subseteq U$ with $|A|\geq (1-\eps _1 r)|W|$ and $|D|\geq (1-\eps _1 r)|U|$. 
Similarly, apply Claim~\ref{claim2} with $K'$, $W'$, $U$ playing the roles of $K_r$, $T_i$ and $T_j$
to obtain 
sets $A'\subseteq W'$ and $D'\subseteq U$ with $|A'|\geq (1-\eps _1 r)|W|$ and $|D'|\geq (1-\eps _1 r)|U|$. 

 By Claim~\ref{claim2}, for any $x\in A$, $y\in A'$ and $z\in D\cap D'$, there exist at least $\xi_2 n^{sh-1}$ $(sh-1)$-sets $Y\subseteq V(G)$ such that both $G[Y\cup\{x\}]$ and $G[Y\cup\{z\}]$ contain perfect $H$-tilings and $\xi_2 n^{sh-1}$ $(sh-1)$-sets $Z\subseteq V(G)$ such that both $G[Z\cup\{y\}]$ and $G[Z\cup\{z\}]$ contain perfect $H$-tilings. 

Furthermore, since $D,D'\subseteq U$ then
$$|D\cap D'|\geq|D|+|D'|-|U|\geq(1-2\varepsilon_1 r)|U|=(1-2\varepsilon_1 r)m\stackrel{(\ref{hier3})}{\geq}\frac{\eta n}{6L_0}.$$

Applying Lemma~\ref{chainlemma} with $D \cap D',  \eta /({6L_0}), \xi_2, \xi_3$ playing the roles of $A,  \alpha , \beta , \gamma$, we conclude that 
 there exist at least $\xi_3 n^{2sh-1}$ $(2sh-1)$-sets $X\subseteq V(G)$ such that both $G[X\cup\{x\}]$ and $G[X\cup\{y\}]$ contain perfect $H$-tilings, as desired.
\end{proof}

In the next claim we use the minimum degree condition of the reduced graph $R$ to find a structure $\mathcal{L}$ containing an arbitrary vertex $x$ and resembling the extremal construction for an ordered graph which has a local barrier. Furthermore, we prove that there exist many chains between $x$ and \emph{almost} every vertex in some  cluster in $\mathcal{L}$. 

\begin{claim}\label{claim4}
Let $x\in V(G)$. There exists some cluster $W\in V(R)$ and a set $A\subseteq W$ of size $|A|\geq(1-\varepsilon_2 r)|W|$ satisfying the following: for every $y\in A$ there exist at least $\xi_1 n^{sh-1}$ $(sh-1)$-sets $X\subseteq V(G)$ such that both $G[X\cup\{x\}]$ and $G[X\cup\{y\}]$ contain perfect $H$-tilings.
\end{claim}

\begin{proof}
Let $x\in V(G)$ and define
$$N^*(x):=\{W^i_j\in V(R):d_{G'}(x,W^i_j)\geq\eta|W^i_j|/100\}.$$
Recall that every non-exceptional cluster $W$ has size $m$. Hence, if $W\in N^*(x)$ then there are at most $m$ neighbours of $x$ in $W$, while if $W\not\in N^*(x)$ then there are at most $\eta m/100$ neighbours of $x$ in $W$. Finally, there are at most $\varepsilon_1 n$ vertices lying in exceptional clusters. Therefore,
$$\left(1-\frac{1}{r-1}+\frac{\eta}{2}\right)n \stackrel{(\ref{minG})}{\leq} d_{G'}(x)\leq m\cdot|N^*(x)|+\eta n/100+\varepsilon_1 n.$$
and thus
\begin{equation}\label{eq1}
\left(1-\frac{1}{r-1}+\frac{\eta}{3}\right)|R|\leq |N^*(x)|.
\end{equation}
Using  (\ref{eq2}) and (\ref{eq1}), we can greedily find $r$ clusters $T_1,\dots,T_r$ in $R$ such that
\begin{itemize}
    \item the $T_k$'s span a copy of $K_r$ in $R$;
    \item $T_k\in N^*(x)$ for every $k\in[r-1]$;
    \item if $x\in W_{k_0}$ with $k_0\in[t]$ then $T_k\not\subseteq  W_{k_0}$ for all $k\in[r]$. 
\end{itemize}
By the properties above, we may relabel indices so that there is an $i \in [r+1]$ and $j \in [r]$ for which $$T_1<\cdots<T_{i-1}<x<T_i<\cdots<T_r$$
and $T_k\in N^*(x)$ for every $k\not=j$.

Define $T'_k:=T_k\cap N_{G'}(x)$ for $k\not=j$ and $T'_j:=T_j$. By construction, $|T'_k|\geq\eta m/100$ for every $k\in[r]$. Thus, by the slicing lemma (Lemma~\ref{slicinglemma}), $(T'_a,T'_b)_{G''}$ is $(\varepsilon_2,d/2)$-regular for every  $a\not=b\in[r]$. 

We will show that the cluster $W:=T_j$ is as desired for the claim.
For every $k\not=j$, let 
$$L_k:=\{v\in T'_j:d_{G''}(v,T'_k)\leq(d/2-\varepsilon_2)|T'_k|\}.$$ 
It is easy to show that $|L_k|<\varepsilon_2|T'_j|$. Let $A:=T'_j\setminus(L_1\cup\cdots\cup L_r)$, then
$$|A|\geq(1-\varepsilon_2 r)|T'_j|=(1-\varepsilon_2 r)|T_j|.$$

Let $y\in A$. Define $T''_k:=T'_k\cap N_{G''}(y)$ for every $k\not=j$ and $T''_j=T'_j\setminus\{y\}$. By the slicing lemma, $(T''_a,T''_b)_{G''}$ is $(\varepsilon_3,d/3)$-regular for every  $a\not=b\in[r]$. Furthermore, $x$ and $y$ are  adjacent to all vertices in $T''_k$ for every $k\not=j$ (see Figure~\ref{fig:3}).

\begin{figure}[!h]
\centering
\begin{tikzpicture}
\fill[pattern=north east lines, pattern color=black] (0,0.2)--(9,0.2)--(9,-0.2)--(0,-0.2)--cycle;
\fill[pattern=north east lines, pattern color=black] (0,0.5) arc (135:45:6.364)--(9,0) arc (45:135:6.364)--cycle;

\fill[black!50] (0,-2)--(3.5,0)--(4.5,0)--cycle;
\fill[black!50] (0,-2)--(8.5,0)--(9.5,0)--cycle;

\fill[black!50] (6.5,-2)--(3.5,0)--(4.5,0)--cycle;
\fill[black!50] (6.5,-2)--(8.5,0)--(9.5,0)--cycle;

\filldraw[color=black, fill=white] (0,0) circle (1) node{$T_1\setminus\{y\}$};
\filldraw[black] (0,-2) circle (3pt) node[below=3pt]{$y$};
\filldraw[color=black, fill=white] (4,0) circle (0.5) node{$T''_2$};
\filldraw[black] (6.5,-2) circle (3pt) node[below=3pt]{$x$};
\filldraw[color=black, fill=white] (9,0) circle (0.5) node{$T''_3$};
\end{tikzpicture}
\caption{In this picture, we take $r=3$, $T_1<T_2<x<T_3$ and $j=1$. Note that $(T_1\setminus\{y\},T''_2)$, $(T_1\setminus\{y\},T''_3)$ and $(T''_2,T'_3)$ are regular pairs, while $x$ and $y$ are adjacent to all vertices in $T''_2$ and $T''_3$.}
\label{fig:3}
\end{figure}
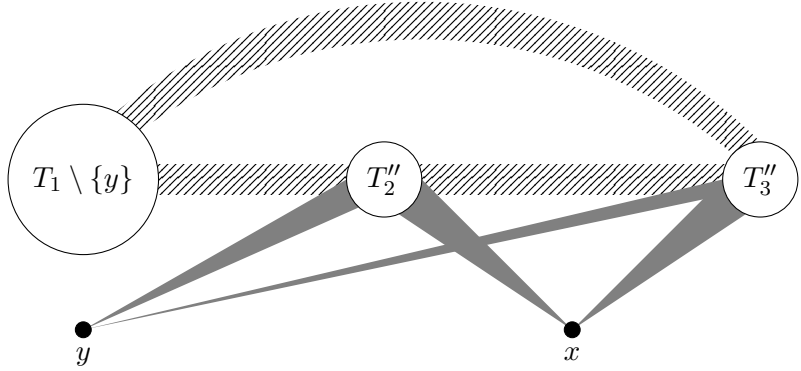

Since $H$ has no local barrier, there exists an interval $(r+1)$-colouring $$V_1<\cdots<V_{i-1}<\{v\}<V_i<\cdots<V_r$$ 
of $H$ such that there is no edge between $v$ and $V_j$. Furthermore, as $\chi_{cr}^*(H)<\chi_<(H)$, $H$ is flexible by Corollary~\ref{flexibilitycol2}. Let $F$ be the complete $r$-partite ordered graph with parts $F_1<\cdots<F_r$ as defined in Lemma~\ref{flexibilitylemma1} and let $F'$ be the complete $r$-partite ordered graph with parts $F'_1<\cdots<F'_r$ where $$|F'_k|=\begin{cases}
|F_k|+|V_k|+1 &\text{if $k=j$,}\\
|F_k|+|V_k| &\text{otherwise.}
\end{cases}$$
Note that
$$1+\sum_{k=1}^r |V_k|=1+(h-1)=h$$
thus, both $F$ and $F'$ contain perfect $H$-tilings by Corollary~\ref{flexibilitycol1}. 

Pick any $u\in F'_j$. Note that $F'\setminus\{u\}$ is the complete $r$-partite ordered graph such that the $k$th part has size $|F_k|+|V_k|$ for every $k\in[r]$ and 
 $|F'\setminus\{u\}|=|F|+h-1=sh-1$. By the key lemma (Lemma~\ref{keylemma}), there exist at least $\xi_1 n^{sh-1}$ $(sh-1)$-sets $X\subseteq V(G)$ such that $G[X]$ spans a copy of $F'\setminus\{u\}$ with $F'_k\subseteq T''_k$ for every $k\not=j$ and $F'_j\setminus\{u\}\subseteq T''_j$. It remains to show that both $G[X\cup\{x\}]$ and $G[X\cup\{y\}]$ contain perfect $H$-tilings.

First, we consider $G[X\cup\{x\}]$. Since $G[X]$ spans a copy of $F'\setminus\{u\}$ with $F'_k\subseteq T''_k$ for every $k\not=j$ and $F'_j\setminus\{u\}\subseteq T''_j$, then $X$ can be partitioned into two sets $Y,Z$ such that
$$|Y\cap T''_k|=|V_k|\ \ \text{and} \ \ |Z\cap T''_k|=|F_k|$$
for every $k\in[r]$. Note that $G[Z]$ spans a copy of $F$, thus $G[Z]$ contains a perfect $H$-tiling. On the other hand, $G[Y]$ spans a copy of $H\setminus\{v\}$. Recall that $x$ is adjacent to all vertices in $T''_k$ for every $k\not=j$ and $T''_{i-1}<x<T''_i$. Thus,  $G[Y\cup\{x\}]$ spans a copy of $H$ in $G$
where $x$ plays the role of $v$. It follows that $G[X\cup\{x\}]$ contains a perfect $H$-tiling.

Next, we consider $G[X\cup\{y\}]$. Since $y\in T_j$ and $y$ is adjacent to all vertices in $T''_k$ for $k\not=j$, then $G[X\cup\{y\}]$ spans a copy of $F'$ and so $G[X\cup\{y\}]$ contains a perfect $H$-tiling. 
\end{proof}

\begin{remark}\label{remark1}
Note that 
 Claim~\ref{claim4} holds even if we relax the hypothesis of Theorem~\ref{absorbingthm} 
to allow $\chi _<(H)\geq 2$; that is, we did not use the condition that $r \geq 3$ anywhere in the proof of this claim. This fact will be useful in the proof of Theorem~\ref{covercol} in the next subsection.
\end{remark}

Our final claim states that given arbitrary vertices $x,y\in V(G)$ there exist many chains of bounded size between $x$ and $y$.

\begin{claim}\label{finalclaim}
Let $x,y\in V(G)$. There exist at least $\xi_5 n^{4sh-1}$ $(4sh-1)$-sets $X\subseteq V(G)$ such that both $G[X\cup\{x\}]$ and $G[X\cup\{y\}]$ contain perfect $H$-tilings.
\end{claim}

\begin{proof}
By Claim~\ref{claim4} there exist clusters $W,W'\in V(R)$ and sets $A\subseteq W$, $A'\subseteq W'$ of size $|A|\geq(1-\varepsilon_2 r)|W|$ and $|A'|\geq(1-\varepsilon_2 r)|W'|$ such that for every $z\in A$ and $z'\in A'$ there exist at least $\xi_1 n^{sh-1}$ $(sh-1)$-sets $Y\subseteq V(G)$ such that both $G[Y\cup\{x\}]$ and $G[Y\cup\{z\}]$ contain perfect $H$-tilings and at least $\xi_1 n^{sh-1}$ $(sh-1)$-sets $Y'\subseteq V(G)$ such that both $G[Y'\cup\{y\}]$ and $G[Y'\cup\{z'\}]$ contain perfect $H$-tilings. 

Furthermore, by Claim~\ref{claim3} there exist sets $D\subseteq W$, $D'\subseteq W'$ of size $|D|\geq(1-\varepsilon_1 r)|W|$ and $|D'|\geq(1-\varepsilon_1 r)|W'|$ such that for every $z\in D$ and $z'\in D'$ there exist at least $\xi_3 n^{2sh-1}$ $(2sh-1)$-sets $Z\subseteq V(G)$ such that both $G[Z\cup\{z\}]$ and $G[Z\cup\{z'\}]$ contain perfect $H$-tilings.
Note that
$$|A\cap D|\geq(1-2\varepsilon_2 r)|W|=(1-2\varepsilon_2 r)m\stackrel{(\ref{hier3})}{\geq}\frac{\eta n}{6L_0},$$
$$|A'\cap D'|\geq(1-2\varepsilon_2 r)|W'|=(1-2\varepsilon_2 r)m\stackrel{(\ref{hier3})}{\geq}\frac{\eta n}{6L_0}.$$
Apply Lemma~\ref{chainlemma} to $y$ and any $z\in A\cap D$ with $A'\cap D',  \eta /({6L_0}), \xi_3, \xi _4$ playing the roles of $A,  \alpha, \beta , \gamma$; we conclude that  there are at least ${\xi_4} n^{3sh-1}$ $(3sh-1)$-sets $X'\subseteq V(G)$ such that both $G[X'\cup\{z\}]$ and $G[X'\cup\{y\}]$ contain perfect $H$-tilings. 

Next apply 
 Lemma~\ref{chainlemma} to $x$ and $y$, with $A\cap D, \eta /({6L_0}), \xi_4, \xi_5$ playing the roles of $A,  \alpha, \beta, \gamma$; this yields at least $\xi_5 n^{4sh-1}$ $(4sh-1)$-sets $X\subseteq V(G)$ such that both $G[X\cup\{x\}]$ and $G[X\cup\{y\}]$ contain perfect $H$-tilings, as desired.
\end{proof}

Observe that, by Claim~\ref{finalclaim}, the hypothesis of Theorem~\ref{absorbingtool} is satisfied with
$4s$ and $\xi _5$ playing the roles of $s$ and $\xi$. Thus, $V(G)$ contains a set $Abs$ so that
\begin{itemize}
\item $|Abs|\leq(\xi_5/2)^hn/4$;
\item $Abs$ is an $H$-absorbing set for any $W\subseteq V (G)\setminus Abs$ such that $|W|\in h\mathbb{N}$ and $|W|\leq(\xi_5/2)^{2h}n/(512s^2h^3)$.
\end{itemize}
Recall that $\nu= (\xi_5/2)^h/4$. Therefore we have that
$|Abs|\leq \nu n$ and by (\ref{hier1}) and (\ref{hier2}) we have that $\nu ^3 <(\xi_5/2)^{2h}/(512s^2h^3) $; so $Abs$ is as desired.
\qed

\subsection{Proof of Theorem~\ref{thmcover}}
By arguing as in Claim~\ref{claim4}, we obtain the following result.

\begin{thm}\label{covercol}
Let $H$ be an ordered graph that does not have a local barrier and let $\eta>0$. 
There exists an $ n_0\in \mathbb N$ so that the following holds. If $G$ is an ordered graph  on $n\geq n_0$ vertices with
$$\delta(G)\geq\left(1-\frac{1}{\chi_<(H)-1}+\eta\right)n,$$
 then for any vertex $x\in V(G)$ there exists a copy of $H$ in $G$ covering the vertex $x$.
\end{thm}
\proof
Let $r:=\chi_<(H)$.
Given such an ordered graph $G$, define constants as in (\ref{hier1}) and (\ref{hier2}). 
As in the proof of Theorem~\ref{absorbingthm}, 
apply the regularity lemma, and then argue as in the proof of
Claim~\ref{claim4} to obtain the following:
given any
$x\in V(G)$, there are  disjoint  $T'_1,\dots,T'_r\subseteq V(G)$ where $|T'_k|\geq \eta^2  n/(500L_0)$ and  $i \in [r+1]$, $j \in [r]$ 
such that
$$T'_1<\cdots<T'_{i-1}<x<T'_i<\cdots<T'_r;$$
$(T'_a,T'_b)$ is $(\eps_2,d/2)$-regular for every $a\not=b\in[r]$; $x$ is adjacent to all vertices in $T'_k$ for any $k\not=j$.

As $H$ does not have a local barrier, there exists an interval $(r+1)$-colouring
$$V_1<\cdots<V_{i-1}<\{v\}<V_i<\cdots<V_r$$
of $H$ such that there is no edge between $v$ and $V_j$. By the key lemma, there exists some set $X\subseteq V(G)$ such that $G[X]$ spans a copy of $H\setminus\{v\}$ with $V_k\subseteq T'_k$ for every $k\in[r]$. Since $x$ is adjacent to all vertices in $T'_k$ for every $k\not=j$, $G[X\cup\{x\}]$ spans a copy of $H$ in $G$, as desired.
\endproof
We are now ready to prove Theorem~\ref{thmcover} using Theorem~\ref{BalLiTre2} and Theorem~\ref{covercol}.

\begin{proofofthmcover}
Note that the lower bounds stated in Theorem~\ref{thmcover} follow immediately from  Corollary~\ref{ex1col} and Lemma~\ref{ex3lemma}. It remains to prove the upper bounds.

Let $H$ be an ordered graph and $\eta>0$. Let $G$ be an ordered graph on $n$ vertices with $n$ sufficiently large and minimum degree
\begin{itemize}
    \item[(i)] $  \delta(G)\geq\left (1-\frac{1}{\chi_<(H)}+\eta\right )n $ \ \  if $H$ has a local barrier;
\item[(ii)] $  \delta(G)\geq\left (1-\frac{1}{\chi_<(H)-1}+\eta\right )n $  \ \  if $H$ does not have a local barrier.
\end{itemize}
In case (ii) we obtain an $H$-cover by Theorem~\ref{covercol}.
In case (i), consider any
 $x\in V(G)$. Let $y\in V(G)\setminus \{x\}$. By Theorem~\ref{BalLiTre2} there exists a set $X\subseteq V(G)$ such that both $G[X\cup\{x\}]$ and $G[X\cup\{y\}]$ contain perfect $H$-tilings. In particular, there exists a copy of $H$ in $G$ covering  $x$. Thus, $G$ has an $H$-cover, as desired.
\qed
\end{proofofthmcover}

\section{Properties of $\chi^*_{cr}(H)$}\label{sec:properties}
In this section, we provide various bounds on the parameter $\chi_{cr}^*(H)$. We start by showing some natural upper and lower bounds.

\begin{prop}\label{lowerbound}
Let $H$ be an ordered graph on $h$ vertices and let $r:=\chi_<(H)$. Let $\mathcal{C}$ denote the set of interval $r$-colourings of $H$. Define
$$\ell_-(H):=\max_{(H_1<\cdots<H_r)\in\mathcal{C}} |H_1|\ \ \ \text{and} \ \ \
\ell^*_-(H):=\max_{(H_1<\cdots<H_r)\in\mathcal{C}} |H_r|.$$
Then  $\chi_{cr}^*(H)\geq \max  \{ h/\ell_-(H) \, , \,  h/\ell ^*_-(H) \}$.
\end{prop}
\begin{proof}
We only prove that $\chi_{cr}^*(H)\geq h/\ell_-(H)$ as  the proof that $\chi_{cr}^*(H)\geq h/\ell ^*_-(H)$ is analogous.
Let $B$ be a bottlegraph of $H$ with parts $B_1,\dots,B_k$ for some $k\in\mathbb{N}$.
Our aim is to show that $\chi_{cr}(B)\geq h/\ell_-(H)$. By Proposition~\ref{propbottle}, we may assume there exists some $m\in\mathbb{N}$ such that
$$|B_k|\leq m \quad\text{and}\quad  |B_i|=m$$
for every $i\in[k-1]$. Let $\phi$ be an interval labelling of $B$ such that the ordered graph $(B,\phi)$ has parts $B_1<\cdots<B_k$. Since $B$ is a bottlegraph, there exists a $t \in \mathbb N$ so that $(B(t),\phi)$ contains a perfect $H$-tiling $\mathcal{H}$. Note $\mathcal{H}$ consists of $|B|t/h$ copies of $H$ and every copy of $H$ has at most $\ell_-(H)$ vertices in $B_1$, thus
$$mt=|B_1|t\leq\ell_-(H)|\mathcal{H}|=\ell_-(H)|B|t/h.$$
Since $\chi_{cr}(B)=|B|/m$, the above implies $\chi_{cr}(B)\geq h/\ell_-(H)$
and so $\chi_{cr}^*(H)\geq h/\ell_-(H).$
\end{proof}
Note Proposition~\ref{lowerbound} implies that if $H$ is such that $1$ and $2$ are adjacent or $h-1$ and $h$ are adjacent then $\chi_{cr}^*(H)=h$.

\begin{prop}\label{upperbound}
Let $H$ be an ordered graph on $h$ vertices and let $r:=\chi_<(H)$. Let $\mathcal{C}$ denote the set of interval $r$-colourings of $H$. Define
$$\ell_+(H):=\max_{(H_1<\cdots<H_r)\in\mathcal{C}}\left\{\min_{i\in[r]} |H_i|\right\}.$$
Then $\chi_{cr}^*(H)\leq h/\ell_+(H)$.
\end{prop}
\begin{proof}
Let $H_1<\cdots<H_r$ be an interval $r$-colouring of $H$ such that $$\ell_+(H)=\min_{i\in[r]}|H_i|.$$ 
Let $k:=\lfloor h/\ell_+(H)\rfloor$ and let $B$ be the complete $(k+1)$-partite unordered graph with parts $B_1,\dots,B_{k+1}$ such that
$$|B_{k+1}|=(h-k\ell_+(H))\cdot(h!) \quad\text{and}\quad  |B_s|=\ell_+(H)\cdot(h!)$$
for every $s\in[k]$. 
Note that $B_{k+1}$ is the smallest part and may be empty.
We have that $|B|=h \cdot h!$ and $\chi_{cr}(B)=h/\ell_+(H)$.

We will now prove that $B$ is a simple bottlegraph of  $H$. Let $\sigma$ be a permutation of $[k+1]$ and let $\phi$ be an interval labelling of $B$ with respect to $\sigma$. Note that $V((B,\phi))=[h\cdot(h!)]$. Set $t_0:=0$ and 
$$t_i:=(|H_1|+\cdots+|H_i|)\cdot(h!)$$
for every $i\in[r]$. For every $j\in[h!]$ and $i\in[r]$, let 
$$T_j^i:=[t_{i-1}+j|H_i|]\setminus[t_{i-1}+(j-1)|H_i|].$$
Observe that the set of intervals $\{T_j^i\}_{j\in[h!]}$ is a partition of $[t_i]\setminus[t_{i-1}]$ and in particular the set of all intervals  $T_j^i$ is a partition of $[h\cdot(h!)]$. 

Relabel the parts of $(B,\phi)$
by $B'_1,\dots, B'_{k+1}$ so that
$B'_1<\dots<B'_{k+1}$.
Suppose that $T_j^i$ intersects both $B'_s$ and $B'_{s+1}$ for some $i\in[r]$, $j\in[h!]$ and $s\in[k]$. Since $|T_{j'}^i|=|H_i|$ for every $j'\in[h!]$, it follows that
$$|(B'_1\cup\cdots\cup B'_s)\setminus [t_{i-1}]|$$
is not divisible by $|H_i|$, as otherwise no $T_{j'}^i$ would intersect both $B'_s$ and $B'_{s+1}$ for any $j'\in[h!]$, and in particular for $j'=j$, contradicting our assumption. However, $|B'_1\cup\cdots\cup B'_s|$ and $t_{i-1}$ are both divisible by $h!$ and thus 
$|(B'_1\cup\cdots\cup B'_s)\setminus [t_{i-1}]|$ is divisible 
by $|H_i|$, reaching a contradiction. Therefore, every $T_j^i$ is a subset of some $B'_s$.

Suppose $T_j^i$ and $T_j^{i+1}$ are both subsets of some $B'_s$, then
\begin{align*}
|B'_s|&\geq(t_i+j|H_{i+1}|)-(t_{i-1}+(j-1)|H_i|) =(h!-j+1)|H_i|+j|H_{i+1}| \geq(h!+1)\cdot\ell_+(H),
\end{align*}
again reaching a contradiction. It follows that the ordered graph $T_j$ spanned by $T_j^1<\cdots<T_j^r$ is a complete $r$-partite ordered subgraph of $(B,\phi)$. Since $|T_j^i|=|H_i|$ then $T_j$ spans a copy of $H$ in $(B,\phi)$ for every $j\in[h!]$. The $T_j$'s are disjoint and cover all the vertices of $(B,\phi)$, thus they yield a perfect $H$-tiling. Since $\sigma,\phi$ are arbitrary, $B$ is a simple bottlegraph of $H$. In particular,
$$\chi_{cr}^*(H)\leq\chi_{cr}(B)=h/\ell_+(H).$$
\end{proof}

The next result follows easily from the previous bounds.

\begin{prop}\label{complete1class}
Let $H$ be a complete $r$-partite ordered graph on $h$ vertices with parts $H_1<\cdots<H_r$ such that (i) $|H_1|\leq|H_i|$ for every $i\in[r]$ or (ii) $|H_r|\leq|H_i|$ for every $i\in[r]$. Then $\chi_{cr}^*(H)=h/|H_1| \geq \chi_<(H)$ in case (i) and $\chi_{cr}^*(H)=h/|H_r| \geq \chi_<(H)$ in case (ii).
\end{prop}
\begin{proof}
For case (i),
let $\ell_-(H)$ and $\ell_+(H)$ be as in Proposition~\ref{lowerbound} and Proposition~\ref{upperbound} respectively. Note that $\ell_-(H)=\ell_+(H)=|H_1|$, hence Propositions~\ref{lowerbound} and~\ref{upperbound} imply that $\chi_{cr}^*(H)=h/|H_1|$. Case (ii) follows similarly.
\end{proof}

In \cite{blt}, Balogh, Li and the second author (implicitly) computed $\chi_{cr}^*(H)$ for any $H$ such that $\chi_<(H)=2$. Their result can be easily recovered using Propositions~\ref{lowerbound} and \ref{upperbound}.

\begin{prop}[Balogh, Li and Treglown \cite{blt}]\label{bottlegraphbipartite}
Let $H$ be an ordered graph with vertex set $[h]$ such that $\chi_<(H)=2$. Define $\alpha^+(H)$ to be the largest integer $t\in[h]$ so that $[1,t]$ is an independent set in $H$; $\alpha^-(H)$ to be the largest integer $t\in[h]$ so that $[h-t+1,h]$ is an independent set in $H$;  $\alpha(H):=\min\{\alpha^+(H),\alpha^-(H)\}$. Then
$$\chi_{cr}^*(H)=\frac{h}{\alpha(H)}.$$
\end{prop}
\begin{proof}
Let $\ell_-(H),\ell_-^*(H),\ell_+(H)$ be as in Propositions~\ref{lowerbound} and \ref{upperbound}.

Observe that both $[1,\alpha^+(H)]<[\alpha^+(H)+1,h]$ and  $[1,h-\alpha^-(H)]<[h-\alpha^-(H)+1,h]$ are interval $2$-colourings of $H$. In particular,
$$\ell_-(H)=\alpha^+(H)\quad\text{and}\quad \ell_-^*(H)=\alpha^-(H),$$
and so $\chi_{cr}^*(H)\geq h/\alpha(H)$ by Proposition~\ref{lowerbound}. It remains to show that $\chi_{cr}^*(H)\leq h/\alpha(H)$. If $\alpha(H)\leq h/2$ then $\ell _+(H)\geq\alpha(H)$\footnote{In fact, if $\alpha(H)\leq h/2$ then $\ell_+(H)=\alpha(H)$.} and so $\chi_{cr}^*(H)\leq h/\alpha(H)$ by Proposition~\ref{upperbound}. If $\alpha(H)>h/2$, then both $[1,\alpha(H)]<[\alpha(H)+1,h]$ and  $[1,h-\alpha(H)]<[h-\alpha(H)+1,h]$ are interval $2$-colourings of $H$. Therefore the complete bipartite graph $B$ with parts $U,V$ of size $|U|=\alpha(H)$ and $|V|=h-\alpha(H)$ is a simple bottlegraph of $H$. We have $\chi_{cr}(B)=h/\alpha(H)$ and thus $\chi_{cr}^*(H)\leq h/\alpha(H)$.
\end{proof}

If $\chi_<(H)=3$, finding a closed formula for $\chi_{cr}^*(H)$ already proves challenging. In the next result we determine $\chi_{cr}^*(H)$ for any complete $3$-partite ordered graph $H$.

\begin{prop}\label{comp3partitecase}
Let $H$ be a complete $3$-partite ordered graph on $h$  vertices with parts $H_1<H_2<H_3$ of size $h_1,h_2,h_3$ respectively and let
$$g(H):=\left(2-\frac{\min\{h_1,h_2,h_3\}}{\min\{h_1,h_3\}}\right)\cdot\frac{h}{\min\{h_1,h_3\}}.$$
Then $\chi^*_{cr}(H)=g(H)$.
\end{prop}
\begin{proof}
We may assume that $h_1\leq h_3$; the case $h_1\geq h_3$ follows by the symmetry of the argument. If $h_1\leq h_2$ then $\chi_{cr}^*(H)=h/h_1$ by Proposition~\ref{complete1class} and the result follows. So for the rest of the proof we assume $h_2<h_1\leq h_3$. Under these assumptions, $g(H)=(2-h_2/h_1)h/h_1>3$.

\begin{claim}\label{3partiteclaim1}
$\chi_{cr}^*(H)\geq g(H)$.
\end{claim}
Let $B$ be a bottlegraph of $H$ with parts $B_1,\dots,B_k$ for some $k\in\mathbb{N}$. Note that $k\geq3$ since $H$ is $3$-partite.
Our aim is to show that $\chi_{cr}(B) \geq g(H)$.
By Proposition~\ref{propbottle}, we may assume that there exists some $m\in\mathbb{N}$ such that
$$  |B_k|\leq m   \quad\text{and}\quad  |B_i|=m$$
for every $i\in[k-1]$. Pick an interval labelling $\phi$ of $B$ such that the ordered graph $(B,\phi)$ has parts $B_1<B_2<\cdots<B_k$. By definition there exists some $t\in\mathbb{N}$ such that the ordered blow-up $(B(t),\phi)$ contains a perfect $H$-tiling $\mathcal{M}$. Recall that we denote the parts of $(B(t),\phi)$ by $B_1(t)<\cdots<B_k(t)$.

Let $\mathcal{M}_1$ be the set of copies of $H\in\mathcal{M}$ whose first part $H_1$ lies completely in $B_1(t)$ and let $\mathcal{M}_2:=\mathcal{M}\setminus\mathcal{M}_1$. Note that every copy of $H\in\mathcal{M}_1$ has exactly $h_1$ vertices in $B_1(t)$ and at most $h_1+h_2$ vertices in $B_1(t)\cup B_2(t)$, while every copy of $H\in\mathcal{M}_2$ has  at most $h_1$ vertices in $B_1(t) \cup B_2(t)$. Since $|B_1(t)|=|B_2(t)|=mt$, it follows that
$$mt \geq h_1\cdot|\mathcal{M}_1|\quad\text{and}\quad 2mt\leq(h_1+h_2)\cdot|\mathcal{M}_1|+h_1\cdot|\mathcal{M}_2|.$$
The above implies that
$$|\mathcal{M}|=|\mathcal{M}_1|+|\mathcal{M}_2|\geq\left(2-\frac{h_2}{h_1}\right)\frac{mt}{h_1}.$$
Since $\chi_{cr}(B)=|B|/m=|\mathcal{M}|h/(mt)$ then $\chi_{cr}(B)\geq g(H)$, proving the claim.

\smallskip

\begin{claim}\label{3partiteclaim2}
If $h_1=h_3$ then there exists a simple bottlegraph $B$ of $H$ with four parts such that  three parts of $B$ have size $h_1^2$; one part has size $h_1^2-h_2^2$; $\chi_{cr}(B)=g(H)$.
\end{claim}
Suppose $h_1=h_3$. Then
$$g(H)=\left(2-\frac{h_2}{h_1}\right)\frac{h}{h_1}=4-\left(\frac{h_2}{h_1}\right)^2$$
and so $3<g(H)<4$. Let $B$ be the complete $4$-partite graph with parts $B_1,B_2,B_3,B_4$ where
$$|B_1|=h_1^2-h_2^2  \quad\text{and}\quad   |B_i|=h_1^2$$
for every $i=2,3,4$. Note that $\chi_{cr}(B)=|B|/h_1^2=g(H)$. It remains to show that $B$ is a simple bottlegraph of $H$.

Let $\sigma$ be a permutation of $[4]$ and $\phi$ an interval labelling of $B$ with respect to $\sigma$. Recall that the ordered graph $(B,\phi)$ has parts $B_{\sigma^{-1}(1)}<B_{\sigma^{-1}(2)}<B_{\sigma^{-1}(3)}<B_{\sigma^{-1}(4)}$.
Either $\sigma^{-1}(1), \sigma^{-1}(2) \ne 1$ or $\sigma^{-1}(3), \sigma^{-1}(4) \ne 1$.
We will now deal with the former case. 

 Observe that $V(B)$ can be partitioned into two sets $X,Y$ of size $|X|=h_1(2h_1+h_2)$ and $|Y|=(h_1-h_2)(2h_1+h_2)$ such that
$$|Y\cap B_{\sigma^{-1}(i)}|=\begin{cases}
0 & \text{if $i=1$,}\\
h_1(h_1-h_2) & \text{if $i=2$,}\\
h_2(h_1-h_2) & \text{if $i=3$,}\\
h_1(h_1-h_2) & \text{if $i=4$,}
\end{cases}
$$
and
$$|X\cap B_{\sigma^{-1}(i)}|=\begin{cases}
h_1^2 & \text{if $i=1$,}\\
h_1h_2 & \text{if $i=2$,}\\
|B_{\sigma^{-1}(3)}|-h_1h_2+h_2^2 & \text{if $i=3$,}
\end{cases}$$
and all remaining vertices of $X$  lie in $B_{\sigma^{-1}(4)}$.

In particular, (i) $|X\cap B_{\sigma^{-1}(1)}|=h^2_1$; $|X\cap B_{\sigma^{-1}(2)}|=h_1h_2$; 
$|X\cap B_{\sigma^{-1}(3)}|=h_1^2-h_1h_2+h_2^2$; $|X\cap B_{\sigma^{-1}(4)}|=h_1h_2-h^2_2$ or 
(ii) $|X\cap B_{\sigma^{-1}(1)}|=h^2_1$; $|X\cap B_{\sigma^{-1}(2)}|=h_1h_2$; 
$|X\cap B_{\sigma^{-1}(3)}|=h_1^2-h_1h_2$; $|X\cap B_{\sigma^{-1}(4)}|=h_1h_2.$
Thus, we have that $(B,\phi)[X]$ contains a perfect $H$-tiling consisting of $h_1$ copies of $H$ 
where the first  and second  parts $H_1$, $H_2$ of every such copy of $H$ lie in $B_{\sigma^{-1}(1)}$ and 
$B_{\sigma^{-1}(2)}$ respectively.
Further,
 $(B,\phi)[Y]$ contains a perfect $H$-tiling consisting of $h_1-h_2$ copies of $H$. The union of these two $H$-tilings yields a perfect $H$-tiling in $(B,\phi)$. 
 
The only remaining  case is when $\sigma^{-1}(3), \sigma^{-1}(4) \ne 1$. However,
 since $h_1=h_3$, this case will follow by  a symmetric argument to that of the previous case.
Thus $B$ is a simple bottlegraph of $H$.

\smallskip

\begin{claim}\label{3partiteclaim3}
If $g(H)$ is an integer then there exists a simple bottlegraph $B$ of $H$ such that $\chi_{cr}(B)=g(H)$ (i.e., $B$ has $g(H)$ parts) and all parts of $B$ have size $h_1^2$.
\end{claim}
Set $k:=g(H)$. Recall that $k>3$
and so $k\geq4$ since $k\in\mathbb{N}$. Let $B$ be the complete $k$-partite graph with parts $B_1,\dots,B_k$ all of size $h_1^2$. Notice  $\chi_{cr}(B)=k$, so it remains to show that $B$ is a simple bottlegraph of $H$.

Let $\sigma$ be a permutation of $[k]$ and $\phi$ an interval labelling of $B$ with respect to $\sigma$. Recall that the ordered graph $(B,\phi)$ has parts $B_{\sigma^{-1}(1)}<\cdots<B_{\sigma^{-1}(k)}$. Let $X\subseteq V(B)$ be a set of size $4h_1^2-h_2^2$ such that
$$|X\cap B_{\sigma^{-1}(i)}|=\begin{cases}
h_1^2 & \text{if $i=1,2,3$},\\
h_1^2-h_2^2 & \text{if $i=4$}.
\end{cases}$$
Let $H'$ be the complete $3$-partite ordered graph with parts $H'_1<H'_2<H'_3$ of size $h_1,h_2,h_1$ respectively. By Claim~\ref{3partiteclaim2}, $(B,\phi)[X]$ contains a perfect $H'$-tiling consisting of $2h_1-h_2$ copies of $H'$. Observe that by definition of $k$,
$$|V(B)\setminus X|=kh_1^2-(4h_1^2-h_2^2)=(2h_1-h_2)(h_3-h_1).$$
Thus we can partition $V(B)\setminus X$ into $2h_1-h_2$ sets of size $h_3-h_1$ and assign each set to a copy of $H'$. Notice that every copy of $H'$ together with its assigned set forms a copy of $H$; so we constructed a perfect $H$-tiling in $(B,\phi)$. Since $\sigma$ and $\phi$ are arbitrary, $B$ is a simple bottlegraph of $H$.

\smallskip

\begin{claim}\label{3partiteclaim4}
There exists a simple bottlegraph $B$ of $H$ such that $\chi_{cr}(B)=g(H)$.
\end{claim}
Recall that $h_2<h_1\leq h_3$.
We may assume that $g(H)$ is not an integer as otherwise we are done by Claim~\ref{3partiteclaim3}.
Given $t\in\mathbb{N}$ and $\ell\in\mathbb{N} \cup \{0\}$, let $H(t,\ell)$ be the complete $3$-partite ordered graph with parts $H'_1<H'_2<H'_3$ of size $th_1,th_2,th_3-\ell$ respectively.

Set $k:=\lfloor g(H)\rfloor$. We define $t,\ell,s$ and $B$ as follows:
\begin{itemize}

\item If $g(H)\geq 4$, there exist $t, \ell \in\mathbb{N}$  such that 
$$\frac{\ell}{t}=\frac{(g(H)-k)h_1}{\left(2-{h_2}/{h_1}\right)},$$
since the right hand side of the above inequality is a positive rational number. Thus,
\begin{align}\label{eqtl1}
k=g(H)-\left(2-\frac{h_2}{h_1}\right)\frac{\ell/t}{h_1}=\left(2-\frac{h_2}{h_1}\right)\frac{h-\ell/t}{h_1}=\left(2-\frac{th_2}{th_1}\right)\frac{th_1+th_2+th_3-\ell}{th_1}.
\end{align}
Furthermore, we have
\begin{align}\label{eqtl2}
k\geq4>4-\left(\frac{h_2}{h_1}\right)^2=\left(2-\frac{th_2}{th_1}\right)\frac{th_1+th_2+th_1}{th_1}.
\end{align}
Equations (\ref{eqtl1}) and (\ref{eqtl2}) imply that $th_3-\ell>th_1$. Hence $th_2< th_1<th_3-\ell$ and so
$$g(H(t,\ell))=\left(2-\frac{th_2}{th_1}\right)\frac{th_1+th_2+th_3-\ell}{th_1}=k.$$
Let $B$ be the simple bottlegraph of $H(t,\ell)$ as in Claim~\ref{3partiteclaim3} and set $s:=0$.

\item If $3<g(H)<4$, let $t:=1$ and $\ell:=h_3-h_1$. Observe that the parts $H'_1<H'_2<H'_3$ of $H(t,\ell)$ have size $h_1,h_2,h_1$ respectively. Let $B$ be the simple bottlegraph of $H(t,\ell)$ as in Claim~\ref{3partiteclaim2}. Set $s:=h_1^2-h_2^2$.
\end{itemize}
Note that in both cases, $B$ is a complete $(k+1)$-partite graph where each part has size $(th_1)^2$ except one smaller part of size $t^2s$ (which is empty if  $g(H)\geq4$).

Observe that $g(H)-k$ is a rational number and
\begin{align}\label{newnew}
\frac{s}{h_1^2}\leq g(H)-k<1.
\end{align}
Indeed, the lower bound is trivial if $g(H)\geq 4$; if $3<g(H)<4$ then $s=h^2 _1-h^2_2$
and $k=3$, so $g(H)-k \geq (2-h_2/h_1)(2h_1+h_2)/h_1-k= 4-(h_2/h_1)^2 -3=1-(h_2/h_1)^2=s/h^2_1$.

Thus, by (\ref{newnew}), there exist $a,b\in\mathbb{N}\cup \{0\}$ such that $(a,b)\not=(0,0)$ and
$$a(1-g(H)+k)=b(g(H)-k-s/h^2_1) \implies a+b\frac{s}{h_1^2}=(a+b)(g(H)-k)$$
\begin{align}\label{new3}
\implies a(th_1)^2+bt^2s=(a+b)(th_1)^2(g(H)-k).
\end{align}
Let $B'$ be the complete $(k+1)$-partite graph with parts $B'_1,\dots,B'_{k+1}$ where
$$|B'_1|=(a+b)(th_1)^2(g(H)-k) \quad\text{and}\quad   |B'_i|=(a+b)(th_1)^2$$
for every $i>1$. Notice $\chi_{cr}(B')=g(H)$, so it remains to show that $B'$ is a simple bottlegraph of $H$.

Let $\sigma$ be a permutation of $[k+1]$ and let $\phi$ be an interval labelling of $B'$ with respect to $\sigma$. Recall that the ordered graph $(B',\phi)$ has parts $B'_{\sigma^{-1}(1)}<\cdots<B'_{\sigma^{-1}(k+1)}$. Let $X,Y$ be two disjoint sets in $V(B')$ of size $a|B|$ and $b|B|$ respectively such that
$$|X\cap B'_{\sigma^{-1}(i)}|=\begin{cases}
a(th_1)^2 & \text{if $i<k+1$}, \\
a(t^2s)& \text{if $i=k+1$},
\end{cases}$$
and 
$$|Y\cap B'_{\sigma^{-1}(i)}|=\begin{cases}
b(th_1)^2& \text{if $\sigma^{-1}(i)\not=1$}, \\
b(t^2s) & \text{if $\sigma^{-1}(i)=1$}.
\end{cases}$$
Notice that by (\ref{new3}) and the choice of the sizes of the parts of $B'$, all vertices in $B'_{\sigma^{-1}(i)}$  are in $X \cup Y$ for $i \leq k$; as  $s\leq h^2_1$ it may be that some vertices
in $B'_{\sigma^{-1}(k+1)}$ are not in $X \cup Y$.

Note that $(B',\phi)[X]$ is a copy of $(B(a),\phi')$ for some interval labelling $\phi'$. 
So as $B$ is a simple bottlegraph of $H(t,\ell)$, $(B(a),\phi')$ contains a perfect $H(t,\ell)$-tiling $\mathcal{M}_1$ consisting of 
\begin{align}\label{new4}
\frac{a|B|}{|H(t,\ell)|}=\frac{a(th_1)^2g(H(t,\ell))}{|H(t,\ell)|}=a(th_1)^2\left(2-\frac{h_2}{h_1}\right)\frac{h-\ell/t}{h_1}\frac{1}{th-\ell}=a(2h_1-h_2)t
\end{align}
copies of $H(t,\ell)$. Similarly, $(B',\phi)[Y]$ is a copy of $(B(b),\phi'')$ for some interval labelling $\phi''$ and it contains a perfect $H(t,\ell)$-tiling $\mathcal{M}_2$ consisting of 
\begin{align}\label{new5}
\frac{b|B|}{|H(t,\ell)|}=b(2h_1-h_2)t
\end{align}
copies of $H(t,\ell)$. Note that  the vertices in $V(B')\setminus(X\cup Y)$ lie in $B'_{\sigma^{-1}(k+1)}$. Moreover, as $g(H)=(2-h_2/h_1)h/h_1$,
\begin{align*}
|V(B')\setminus(X\cup Y)|&= |B'|-(a+b)|B|=(a+b)(th_1)^2g(H)-(a+b)|B| \\
&\stackrel{(\ref{new4}),(\ref{new5})}{=}(a+b)\left((2h_1-h_2)ht^2-(2h_1-h_2)(th-\ell)t\right)=(a+b)(2h_1-h_2)\ell t.
\end{align*}
Thus, we can divide $V(B')\setminus(X\cup Y)$ into $(a+b)(2h_1-h_2)t$ sets of size $\ell$ and assign each set to a copy of $H(t,\ell)$ in $\mathcal{M}_1\cup\mathcal{M}_2$. Note that each copy of $H(t,\ell)$ together with its assigned set forms a copy of $H(t,0)$. Thus we constructed a perfect $H(t,0)$-tiling in $(B',\phi)$. Since $H(t,0)=H(t)$, this yields a perfect $H$-tiling in $(B',\phi)$, proving that $B'$ is indeed a simple bottlegraph of $H$.

\smallskip

Claim~\ref{3partiteclaim4} implies that $\chi_{cr}^*(H)\leq g(H)$. Together with Claim~\ref{3partiteclaim1}, this concludes the proof.
\end{proof}

Observe that Propositions~\ref{complete1class} and~\ref{comp3partitecase} combined with Theorem~\ref{mainthm} yield the  following result which makes the threshold in Theorem~\ref{mainthm} explicit for all complete $3$-partite ordered graphs.

\begin{thm}
Let $H$ be a complete $r$-partite ordered graph on $h$ vertices with parts $H_1<\cdots<H_r$ where $|H_i|=h_i$ for every $i\in[r]$.
\begin{itemize}
\item If $h_1\leq h_i$ for every $i>1$, then 
$\delta_<(H,n)=\left(1-\frac{h_1}{h}+o(1)\right)n.$
\item If $h_r\leq h_i$ for every $i\in [r-1]$, then 
$\delta_<(H,n)=\left(1-\frac{h_r}{h}+o(1)\right)n.$
\item If $r=3$ and $h_2\leq h_1\leq h_3$ then
$\delta_<(H,n)=\left(1-\frac{h_1^2}{(2h_1-h_2)h}+o(1)\right)n.$
\item If $r=3$ and $h_2\leq h_3\leq h_1$ then
$\delta_<(H,n)=\left(1-\frac{h_3^2}{(2h_3-h_2)h}+o(1)\right)n.$
\end{itemize}
\end{thm}

\section{The tightness of  Theorem~\ref{BalLiTre1}}\label{lowersec}
 In this section we show that the minimum degree condition in Theorem~\ref{BalLiTre1} is  best possible.
Given an ordered graph $H$, let $c_<(H)$ denote the smallest non-negative  number which satisfies the following: for every $\eta>0$, there exists an integer $n_0\in\mathbb{N}$ such that if $G$ is an ordered graph on $n\geq n_0$ vertices and with minimum degree $\delta(G)\geq c_<(H)n$ then $G$ contains an $H$-tiling covering all but at most $\eta n$ vertices. Observe  Theorem~\ref{BalLiTre1} immediately implies
that
\begin{align}\label{cbound}
c_<(H)\leq\left(1-\frac{1}{\chi^*_{cr}(H)}\right).
\end{align}
In fact, we will show that equality holds. 

\begin{thm}\label{aptsthm}
Let $H$ be an ordered graph on $h$ vertices. Then
$$c_<(H)=\left(1-\frac{1}{\chi^*_{cr}(H)}\right).$$
\end{thm}

\begin{proof}
By (\ref{cbound}) it remains to prove that $c_<(H)\geq (1-{1}/{\chi^*_{cr}(H)}).$
Throughout the proof, we let $c:=c_<(H)$ and $r:=\chi_<(H)$.  Suppose for a contradiction that $c<(1-1/\chi^*_{cr}(H))$. We split the proof into three different cases.

\smallskip

{\it Case 1: $\chi^*_{cr}(H)>r$. }

 Let $\eta>0$ be sufficiently small and define $c'$ such that
\begin{align}\label{etac'}
c':=\max\left\{\left(1-\frac{1}{r}\right),c\right\}+\eta<1-\frac{1}{\chi^*_{cr}(H)}.
\end{align}
Let $ 0<\nu\ll\eta$ and let $n \in \mathbb N$ be sufficiently large and divisible by  $h$.

Let $G$ be any ordered graph on $n$ vertices with minimum degree 
$$\delta(G)\geq c'n\stackrel{(\ref{etac'})}{\geq}\left(1-\frac{1}{r}+\eta\right)n.$$
By Theorem~\ref{BalLiTre2}, there exists a set  $Abs\subseteq V(G)$ such that $|Abs|\leq\nu n$ and $Abs$ is an $H$-absorbing set for every $W\subseteq V(G)\setminus Abs$ such that $|W|\in h\mathbb{N}$ and $|W|\leq\nu^3n$. Let $G':=G\setminus Abs$. Note that
$$\delta(G')\geq\left(c'-\eta\right)n\stackrel{(\ref{etac'})}{\geq} cn.$$
By the definition of $c$, and as $n$ is sufficiently large, there exists an $H$-tiling $\mathcal{M}_1$ covering all but at most $\nu^3n$ vertices in $G'$. Let $W\subseteq V(G')$ be the set of vertices not covered by $\mathcal{M}_1$. Since $|W|\leq\nu^3 n$ and $h$ divides $|W|$ (as $h$ divides $n,|Abs|,|V(\mathcal{M}_1)|$), there exists a perfect $H$-tiling $\mathcal{M}_2$ in $G[Abs\cup W]$. Thus, $\mathcal{M}_1\cup\mathcal{M}_2$ is a perfect $H$-tiling in $G$. Recall that $G$ is an arbitrary ordered graph on $n$ vertices with  $\delta(G)\geq c'n$, hence 
$$\delta_<(H,n)\leq c'n<\left(1-\frac{1}{\chi_{cr}^*(H)}\right)n.$$
This contradicts Corollary~\ref{ex2col}. Therefore, the initial assumption that $c<(1-1/\chi^*_{cr}(H))$ is  false and the result follows in this case.

\smallskip

{\it Case 2: $\chi^*_{cr}(H)\leq r$ and $H$ is flexible. }

As $H$ is flexible, there exists a complete $r$-partite ordered graph $F$ with parts $F_1<\cdots<F_r$ as in Lemma~\ref{flexibilitylemma1}. Pick $\eta>0$ 
 sufficiently small
 so that
\begin{align}\label{eta}
c+\eta+\eta|F|\leq 1-\frac{1}{\chi_{cr}^*(H)}
\end{align}
and 
\begin{align}\label{eta2}
\frac{r-1}{1-\eta |F|}< \frac{\chi_{cr}^*(H)}{1+\eta\chi_{cr}^*(H)}.
\end{align}
Let $n \in \mathbb N$ be sufficiently large and divisible by  $h$;  we may choose $\eta $ and $n$ so that $\eta n \in \mathbb N$. 

Recall that by (\ref{cruciallower}), $r-1<\chi^*_{cr}(H)\leq r$.
Thus, it is not difficult to show that there exist a complete $r$-partite unordered graph $B$ on $n$ vertices with parts $B_1,\dots,B_r$ such that all parts have the same size, except one smaller part, and
\begin{align}\label{chi}
\frac{\chi_{cr}^*(H)}{1+\eta\chi_{cr}^*(H)}\leq\chi_{cr}(B)<\chi_{cr}^*(H).
\end{align}
Note that
$$\delta(B)=\left(1-\frac{1}{\chi_{cr}(B)}\right)n\stackrel{(\ref{chi})}{\geq}\left(1-\frac{1}{\chi_{cr}^*(H)}-\eta\right)n.$$
Pick any permutation $\sigma$ of $[r]$ and any interval labelling $\phi$ of $B$ with respect to $\sigma$. Up to relabelling, we may assume that the complete $r$-partite ordered graph $(B,\phi)$ has parts $B_1<\cdots<B_r$. 

Notice that (\ref{eta2}) and (\ref{chi}) imply  $|B_k|\geq \eta |F| n$ for all $k \in [r]$.
Let $X\subseteq V(B)$ such that $|X\cap B_k|=\eta |F_k|n$ for every $k\in[r]$. Clearly, $(B, \phi)[X]$ is a copy of the ordered blow-up $F(\eta n)$. Let $B'$ be the ordered graph obtained from $(B,\phi)$ by deleting the vertices in $X$; so $B'$ is a complete $r$-partite ordered graph with parts $B'_1<\cdots<B'_r$ where $B'_k\subseteq B_k$ for every $k\in[r]$. Note that
$$\delta(B')\geq\delta(B)-\eta n|F|\geq\left(1-\frac{1}{\chi_{cr}^*(H)}-\eta-\eta|F|\right)n\stackrel{(\ref{eta})}{\geq}cn.$$
By the definition of $c$, there exists an $H$-tiling $\mathcal{M}_1$ in $B'$ covering all but at most $(\eta h) n$ vertices. Let $W\subseteq V(B')$ be the set of vertices not covered by $\mathcal{M}_1$. Furthermore, let $s_k:=|W\cap B'_k|$ for every $k\in[r]$. Note that $(B,\phi)[X\cup W]$ is a complete $r$-partite ordered graph with parts $L_1<\cdots<L_r$ such that $|L_k|=\eta n|F_k|+s_k$. Since $s_1+\cdots+s_r=|W|\leq(\eta n)h$ and $h$ divides $|W|$, then by Corollary~\ref{flexibilitycol1} there exists a perfect $H$-tiling $\mathcal{M}_2$ in $(B,\phi)[X\cup W]$. Finally, note that $\mathcal{M}_1\cup\mathcal{M}_2$ is a perfect $H$-tiling in $(B,\phi)$. Since $\sigma,\phi$ are arbitrary, this implies that $B$ is a simple bottlegraph of $H$. This is a contradiction since $\chi_{cr}(B)<\chi_{cr}^*(H)$. Thus, the initial 
assumption that $c<(1-1/\chi^*_{cr}(H))$ is false and the result follows in this case.

\smallskip

{\it Case 3: $\chi^*_{cr}(H)\leq r$ and $H$ is not flexible. }

Note that if $\chi^*_{cr}(H)<r$ then $H$ is flexible by Corollary~\ref{flexibilitycol2}, a contradiction.
Thus, $\chi^*_{cr}(H)=r$. In this case we  actually   produce an explicit extremal example to show that $c\geq 1-1/r$.

Since $H$ is not flexible, by Lemma~\ref{flexibilitylemma2} there exist some $i\in[r-1]$ such that the number of vertices in the first $i$ intervals of any given interval $r$-colouring of $H$ is fixed. Clearly, this implies that the number of vertices in the last $(r-i)$ intervals of any given interval $r$-colouring of $H$ is also fixed. Since $H$ has $h$ vertices, the number of vertices in the first $i$ intervals is at least $ih/r$ or the number of vertices in the last $(r-i)$ intervals is at least $(r-i)h/r$. Without loss of generality we may assume the former.

Fix $0<\eta<1$ and $n\in\mathbb{N}$ such that ${n}/{r}$ and ${\eta n}/{r}$ are integers. Let $G$ be the complete $r$-partite ordered graph on $n$ vertices with parts $G_1<\cdots<G_r$ of size
$$|G_j|=\begin{cases}
\frac{n}{r}(1-\eta) & \text{ if $j<r$}, \\
\frac{n}{r}\left(1+(r-1)\eta\right) & \text{ if $j=r$}.
\end{cases}$$
Observe that
$$\delta(G)=n-\frac{n}{r}\left(1+(r-1)\eta\right)=\left(1-\frac{1}{r}-\frac{(r-1)\eta}{r}\right)n.$$
By assumption, a copy of $H$ in $G$ has at least $ih/r$ vertices in $G_1\cup\cdots\cup G_i$. Suppose $\mathcal{M}$ is an $H$-tiling in $G$, then the number of copies of $H$ in $\mathcal{M}$ is at most 
$$\frac{|G_1\cup\cdots\cup G_i|}{(ih/r)}=\frac{n(1-\eta)}{h}.$$
In particular, there are at least $\eta n$ vertices in $G$ which are not covered by $\mathcal{M}$. 

In summary what we have shown is that, given any $c'< 1-1/r$, we can choose $\eta>0$ sufficiently small so that for an infinite number of choices of $n \in \mathbb N$, we can produce an $n$-vertex ordered graph $G$ with $\delta(G) \geq c'n$
and so that $G$ does not contain an $H$-tiling covering more than $(1-\eta)n$ vertices. So by definition $c \geq 1-1/r$, as desired.
\end{proof}

\section{Proof of Theorem~\ref{OrderedKomlos}}\label{sec:ok}
This section is devoted to the proof of Theorem~\ref{OrderedKomlos}. We will need the following result from \cite[Lemma~6.2]{blt} which itself is a special case of a lemma of B\'ar\'any and Valtr~\cite{bv}.

\begin{lemma}[Balogh, Li and Treglown \cite{blt}]\label{ordertrick}
For $n\geq k\geq2$, let $A_1,A_2,\dots,A_k$ be nonempty disjoint subsets of $[n]$. Then there exist sets $S_1,S_2,\dots,S_k$, where $S_i\subseteq A_i$, and a permutation $\sigma=(\sigma(1), \sigma(2),\dots,\sigma(k))$ of the set $[k]$, such that the following conditions hold for all $i, j\in[k]$:
\begin{itemize}
\item $|S_i|\geq\lfloor|A_i|/k\rfloor$;
\item $S_i<S_j$ if $\sigma(i)<\sigma(j)$.
\end{itemize}
\end{lemma}

\smallskip

\begin{proofofOrderedKomlos}
Let $H$ be an ordered graph on $h$ vertices and let $x\in(0,1)$. Throughout the proof we set $r:=\chi_<(H)$. If $r=1$ the statement of the theorem holds trivially, so we may assume that $r\geq 2$. Observe that it suffices to prove that the theorem holds for any $\eta>0$ sufficiently small. Fix constants $0<\eta\ll 1/\chi_{cr}^*(x,H)$ and 
\begin{align}\label{vareps}
0<\varepsilon\ll(1-x)\eta/(xh).
\end{align}
By definition of $\chi_{cr}^*(x,H)$ there is 
 an $x$-bottlegraph $B$ of $H$ with 
\begin{align}\label{xeta}
\chi_{cr}^*(x,H)\leq \chi_{cr}(B)<\frac{\chi_{cr}^*(x,H)}{1-(\eta/12)\chi_{cr}^*(x,H)}.
\end{align}
Let $k:=\chi(B)\geq r$ and
let $B_1,\dots, B_k $ be the parts of $B$.
Fix $t\in\mathbb{N}$ such that $k/t<\varepsilon$.

Let $n\in\mathbb{N}$ be sufficiently large and let $G$ be an ordered graph on $n$ vertices with minimum degree
$$\delta(G)\geq\left(1-\frac{1}{\chi_{cr}^*(x,H)}+\eta\right)n.$$
Recall by (\ref{goodlowerbound}),  $\chi^*_{cr}(x,H)\geq r-1$; so $\delta(G)\geq(1-1/(r-1)+\eta)n$. By the {\it{Erd\H{o}s--Stone--Simonovits theorem for ordered graphs}}~\cite{pt} there exists a copy of $H$ in $G$. We remove the vertices of $H$ from $G$ and repeat the same process until we obtain a set $W\subseteq V(G)$ such that $G[W]$ contains a perfect $H$-tiling $\mathcal{W}$ and $|W|=\left\lfloor\frac{\eta n}{2h}\right\rfloor h$. Let $G':=G\setminus W$. Observe that
$$\delta(G')\geq\left(1-\frac{1}{\chi_{cr}^*(x,H)}+\frac{\eta}{3}\right)|G'|\stackrel{(\ref{xeta})}{\geq}\left(1-\frac{1}{\chi_{cr}(B)}+\frac{\eta}{4}\right)|G'|.$$
Since $\chi_{cr}(B(t))=\chi_{cr}(B)$ (and ignoring the ordering of $V(G')$),  Theorem~\ref{Komlos} implies there exists a $B(t)$-tiling $\mathcal{B}$ in $G'$ which covers all but at most $\varepsilon |G'|$ vertices. 

Consider a fixed copy of $B(t)\in\mathcal{B}$ whose parts are  $A_1,\dots,A_k$ with $|A_i|=t|B_i|$ for every $i\in[k]$. By Lemma~\ref{ordertrick}, there exist sets $S_i\subseteq A_i$ for every $i\in[k]$ such that $|S_i|\geq\lfloor|A_i|/k\rfloor\geq |B_i|$ and a permutation $\sigma$ of $[k]$ such that $S_i<S_j$ if $\sigma(i)<\sigma(j)$. By discarding some vertices if necessary, we may assume that $|S_i|=|B_i|$ for every $i\in[k]$. Note that the sets $S_1,\dots,S_k$ span a copy of $B$ and the ordering of $V(G')$ induces an interval labelling of $B$ with respect to $\sigma$. 

Crucially, $|A_i\setminus S_i|=(t-1)|B_i|$, so we can repeatedly apply Lemma~\ref{ordertrick} to find a $B$-tiling in $G'$ covering all but at most
$$k|B|=(k/t)|B(t)|$$ 
vertices in our fixed $B(t)\in \mathcal B$. Furthermore, by Lemma~\ref{ordertrick}, the ordering of $V(G')$ induces an interval labelling on each of these copies of $B$. Since $B$ is an $x$-bottlegraph, each of these copies contains an $(x,H)$-tiling. Repeat this process for all $B(t)\in \mathcal B$ and
denote the union of all these $(x,H)$-tilings as $\mathcal{M}$. The number of vertices in $G$ covered by $\mathcal{M} \cup \mathcal{W}$ is at least
\begin{align*}
x(1-k/t)(1-\varepsilon)|G'|+|W|&\geq x(1-\varepsilon)^2(n-|W|)+|W| \\
&\geq xn -2x\varepsilon n -x|W|+|W|\stackrel{(\ref{vareps})}{\geq} xn.
\end{align*}
Hence $\mathcal{M} \cup \mathcal{W}$ is an $(x,H)$-tiling in $G$, as desired.\qed
\end{proofofOrderedKomlos}

\section{Computing the threshold in Theorem~\ref{OrderedKomlos} for some choices of $H$}\label{sec:xtiling}
In this section we investigate the behaviour of the function $f(x,H)$ as defined in Theorem~\ref{OrderedKomlos}. Recall that for an unordered graph $H$, the analogous function $g(x,H)$ (defined in Theorem~\ref{Komlos}) is linear in $x$. On the other hand, for a vertex-ordered graph $H$, the function $f(x,H)$ can behave rather differently.
However, we first give an instance where $f(x,H)$ does grow linearly in $x$. 

\begin{prop}\label{linear}
Let $H$ be an ordered graph with vertex set $[h]$ such that $\chi_<(H)=2$. Define $\alpha^+(H)$ to be the largest integer $t\in[h]$ such that $[1,t]$ is an independent set in $H$; $\alpha^-(H)$ to be the largest integer $t\in[h]$ such that $[h-t+1,h]$ is an independent set in $H$; $\alpha(H):=\min\{\alpha^+(H),\alpha^-(H)\}$. For any $x\in(0,1)$, we have
$$f(x,H)=\frac{x(h-\alpha(H))}{h}.$$

\end{prop}

\begin{proof}
Let $H$ be an ordered graph with vertex set $[h]$ so that $\chi_<(H)=2$, and let $x\in(0,1)$.
Additionally, fix constants $0<1/N\ll\eta<1$ where $N\in\mathbb{N}$. We first prove that $f(x,H)\leq x(h-\alpha(H))/h$.

By the definition of $\alpha^+(H)$ and $\alpha^-(H)$ and the fact that $\chi_<(H)=2$, $[1,\alpha^+(H)]<[\alpha^+(H)+1,h]$ and $[1,h-\alpha^-(H)]<[h-\alpha^-(H)+1,h]$ are  interval $2$-colourings of $H$. The maximality of $\alpha^+(H)$ implies $\alpha^+(H)\geq h-\alpha^-(H)$, i.e., $\alpha^+(H)+\alpha^-(H)-h\geq0$. Furthermore, the previous observations imply that the vertices in $[h-\alpha^-(H)+1,\alpha^+(H)]$ are isolated in $H$ and in particular $[1,h-\alpha^-(H)]<[h-\alpha^-(H)+1,\alpha^+(H)]<[\alpha^+(H)+1,h]$ is an interval $3$-colouring of $H$. (Note that $[h-\alpha^-(H)+1,\alpha^+(H)]$ may be empty.)

Let $H'$ be an ordered graph such that:
\begin{itemize}
\item $V(H')=[h']$ with $h':=Nh+\lfloor (1-x)Nh/x\rfloor$;
\item $V(H')$ can be partitioned into four intervals $D_1<D_2<D_3<D_4$ where
$$|D_1|=(h-\alpha^-(H))N,\quad    |D_2|=(\alpha^+(H)+\alpha^-(H)-h)N,$$
$$|D_3|=\lfloor (1-x)Nh/x\rfloor,  \quad  |D_4|=(h-\alpha^+(H))N,   $$
such that $H'[D_1,D_4]$ is a complete $2$-partite ordered graph and all vertices in $D_2\cup D_3$ are isolated in $H'$.
\end{itemize}
As $[1,h-\alpha^-(H)]<[h-\alpha^-(H)+1,\alpha^+(H)]<[\alpha^+(H)+1,h]$ is an interval $3$-colouring of $H$ and $[h-\alpha^-(H)+1,\alpha^+(H)]$ is a set of isolated vertices, we have that $H'[D_1\cup D_2\cup D_4]$ contains $N$ disjoint copies of $H$. These copies form an $(x,H)$-tiling in $H'$ since $Nh\geq xh'$.

Notice $\alpha^+(H')=|D_1|+|D_2|+|D_3|$ and $\alpha^-(H')=|D_2|+|D_3|+|D_4|$, thus
\begin{align}\label{alpha*}
\alpha(H')=h'-\max\{|D_1|,|D_4|\}=h'-N(h-\alpha(H)).
\end{align}
As $\chi_<(H')=2$, Proposition~\ref{bottlegraphbipartite} implies that
\begin{align}\label{chiH'}
\chi_{cr}^*(H')=\frac{h'}{\alpha(H')}.
\end{align}
Since $H'$ contains an $(x,H)$-tiling, it follows that every $1$-bottlegraph of $H'$ is an $x$-bottlegraph of $H$ and so $\chi_{cr}^*(x,H)\leq\chi_{cr}^*(1,H')$. Proposition~\ref{propeq} implies that $\chi_{cr}^*(1,H')=\chi_{cr}^*(H')$, so 
\begin{align*}
f(x,H)=\left(1-\frac{1}{\chi_{cr}^*(x,H)}\right)\leq\left(1-\frac{1}{\chi_{cr}^*(H')}\right)\stackrel{(\ref{chiH'})}{=}\left(1-\frac{\alpha(H')}{h'}\right)&\stackrel{(\ref{alpha*})}{=}\left(1-\frac{h'-N(h-\alpha(H))}{h'}\right) \\
&=\frac{x(h-\alpha(H))}{xh'/N}.
\end{align*}
Since $N$ is arbitrarily large and $xh'/N\to h$ as $N\to\infty$, the above implies $f(x,H)\leq x(h-\alpha(H))/h$.

\smallskip

Next we prove that $f(x,H)\geq x(h-\alpha(H))/h$. Without loss of generality, we may assume that $\alpha(H)=\alpha^+(H)$. Let $G$ be the ordered graph obtained by taking the complete $2$-partite ordered graph with classes $U<V$ where
$$|U|=N\alpha^+(H)+\left\lceil\frac{1-x}{x}Nh\right\rceil+1\quad\text{and}\quad|V|=N(h-\alpha^+(H))-1,$$
and adding all possible edges to $V$. Suppose $G$ contains an $(x,H)$-tiling $\mathcal{H}$. Since $|G|=Nh+\left\lceil\frac{1-x}{x}Nh\right\rceil$ then $\mathcal{H}$ covers at least $xNh+(1-x)Nh=Nh$ vertices of $G$. Let $\mathcal{H}'\subseteq\mathcal{H}$ where $\mathcal{H}'$  is an $H$-tiling consisting of exactly $N$ copies of $H$ in $G$. Every copy of $H$ has at most $\alpha^+(H)$ vertices in $U$. Furthermore, there are precisely $\left\lceil\frac{1-x}{x}Nh\right\rceil$ vertices in $G$ which are not covered by $\mathcal{H}'$, thus 
$$|U|\leq N\alpha^+(H)+\left\lceil\frac{1-x}{x}Nh\right\rceil<|U|,$$ a contradiction. Therefore, the assumption that $G$ contains an $(x,H)$-tiling is false. It follows from Theorem~\ref{OrderedKomlos} that 
$$\delta(G)< (f(x,H)+\eta)|G|.$$
Finally, we have
\begin{align*}
f(x,H)+\eta>\frac{\delta(G)}{|G|}=\frac{N(h-\alpha^+(H))-1}{Nh+\left\lceil\frac{1-x}{x}Nh\right\rceil}&\geq\frac{N(h-\alpha^+(H))-1}{Nh+\frac{1-x}{x}Nh+1}=\frac{x(h-\alpha^+(H))-x/N}{h+x/N}.
\end{align*}
As $N$ can be chosen arbitrarily large and $\eta$  arbitrarily small, it follows that
$f(x,H)\geq x(h-\alpha(H))/h$.
\end{proof}

In general, for any ordered graph $H$, if $x$ is not too big then the function $f(x,H)$ is linear in $x$.

\begin{prop}\label{linearresult}
Let $H$ be an ordered graph on $h$ vertices and $r:=\chi_<(H)$. Let $\mathcal{C}$ be the set of interval $r$-colourings of $H$. Additionally, set
$$T:=\max_{1\leq i\leq r}\left\{\min_{(H_1< \cdots < H_r)\in\mathcal{C}} |H_i|\right\},\quad J:=\max_{1\leq i\leq r}\left\{\max_{(H_1<\cdots < H_r)\in\mathcal{C}} |H_i|\right\}$$
and $x_0:=\frac{h}{(r-1)J+T}$. Then for any $x\in(0,x_0]$ where $x<1$,
$$f(x,H)=1-\frac{h-xT}{h(r-1)}.$$
\end{prop}
\begin{proof}
Let $H$ be an ordered graph on $h$ vertices with $r:=\chi_<(H)$ and $x\in(0,x_0]$ where $x<1$. Fix constants $0<1/N\ll\eta<1$ where $N\in\mathbb{N}$. We first show that $f(x,H)\leq 1-(h-xT)/(h(r-1))$.

Let $B$ be the complete $r$-partite unordered graph with parts $B_1,\dots,B_r$ where
$$|B_1|=NT\quad\text{ and }\quad |B_i|=\left\lfloor\frac{N}{r-1}\left(\frac{h}{x}-T\right)\right\rfloor$$
for every $i\not=1$. We now prove that $B$ is an $x$-bottlegraph of $H$. Note that for every $i\not=1$,
\begin{align}\label{JT}
|B_i|\geq \left\lfloor\frac{N}{r-1}\left(\frac{h}{x_0}-T\right)\right\rfloor=NJ\geq NT=|B_1|.
\end{align}
Let $\sigma$ be a permutation of $[r]$ and $\phi$ be an interval labelling of $B$ with respect to $\sigma$. Recall that the ordered graph $(B,\phi)$ has parts $B_{\sigma^{-1}(1)}<\cdots<B_{\sigma^{-1}(r)}$. Let $H_1<\cdots <H_r$ be an interval $r$-colouring of $H$ which minimises $|H_{\sigma(1)}|$. By the definition of $T$ we have that 
$$|B_1|=NT\geq N|H_{\sigma(1)}|.$$ 
Furthermore, by the definition of $J$ we have that for every $i\not=1$,
$$|B_i|\stackrel{(\ref{JT})}{\geq}NJ\geq N|H_{\sigma(i)}|.$$
Hence, the ordered graph $(B,\phi)$ contains an $H$-tiling $\mathcal{H}$ consisting of $N$ disjoint copies of $H$. These copies cover exactly $Nh$ vertices of $(B,\phi)$. In particular 
$$x|B|\leq xNT+xN\left(\frac{h}{x}-T\right)=Nh,$$
therefore, $\mathcal{H}$ is an $(x,H)$-tiling in $(B,\phi)$. Since $\sigma,\phi$ are arbitrary, $B$ is indeed an $x$-bottlegraph of $H$ and thus $\chi_{cr}^*(x,H)\leq\chi_{cr}(B)$. Note that
$$\chi_{cr}(B)=\frac{|B|}{|B_2|}
=\frac{N T+(r-1)\left\lfloor\frac{N}{r-1}\left(\frac{h}{x}-T\right)\right\rfloor}{\left\lfloor\frac{N}{r-1}\left(\frac{h}{x}-T\right)\right\rfloor}
< \frac{N T+N\left(\frac{h}{x}-T\right)}{\frac{N}{r-1}\left(\frac{h}{x}-T\right)-1}
=\frac{h(r-1)}{(h-xT)-\frac{x(r-1)}{N}}$$
and so
$$f(x,H)=1-\frac{1}{\chi_{cr}^*(x,H)}
\leq 1-\frac{1}{\chi_{cr}(B)}
<1-\frac{(h-xT)-\frac{x(r-1)}{N}}{h(r-1)}.$$
As $N$ is arbitrarily large, the above implies 
$f(x,H ) \leq 1-({h-xT})/(h(r-1)).$

\smallskip

Next we show that $f(x,H)\geq 1-(h-xT)/(h(r-1))$. Suppose the value of $T$ is achieved for some interval $r$-colouring $H^*_1<\cdots < H^*_r$ of $H$ and $i^*\in[r]$; that is, $T=|H^*_{i^*}|$. Let $G$ be the complete $r$-partite ordered graph with parts $G_1<\cdots<G_r$ of where
$$|G_{i^*}|=NT\quad\text{ and }\quad |G_i|=\left\lceil\frac{N}{r-1}\left(\frac{h}{x}-T\right)\right\rceil+1$$
for every $i\not=i^*$. Notice that $|G_i|  \geq |G_{i^*}|$ for every $i \in [r]$ since $x \leq x_0$. Suppose there exists an $(x,H)$-tiling $\mathcal{H}$ in $G$. By definition of $T$, every copy of $H\in\mathcal{H}$ has at least $T$ vertices in $G_{i^*}$, thus
$$|\mathcal{H}|\leq\frac{|G_{i^*}|}{T}=N.$$
Hence, $\mathcal{H}$ covers at most $Nh$ vertices of $G$. However,
$$|G|>NT+N\left(\frac{h}{x}-T\right)=\frac{Nh}{x}\implies x|G|>Nh,$$
contradicting the assumption that $\mathcal{H}$ is an $(x,H)$-tiling. In particular, $G$ does not contain an $(x,H)$-tiling and so Theorem~\ref{OrderedKomlos} implies
$$\delta(G)<(f(x,H)+\eta)|G|.$$
Let $i\not=i^*$. We have that
$$f(x,H)+\eta>\frac{\delta(G)}{|G|}=\frac{|G|-|G_i|}{|G|}=\left(1-\frac{\left\lceil\frac{N}{r-1}\left(\frac{h}{x}-T\right)\right\rceil+1}{NT+(r-1)\left\lceil\frac{N}{r-1}\left(\frac{h}{x}-T\right)\right\rceil+r-1}\right)$$
$$\geq\left(1-\frac{\frac{N}{r-1}\left(\frac{h}{x}-T\right)+2}{NT+N\left(\frac{h}{x}-T\right)}\right)=1-\frac{(h-xT)+\frac{2x(r-1)}{N}}{h(r-1)}.$$
As $N$ can be chosen arbitrarily large and $\eta$  arbitrarily small, it follows that
$f(x,H)\geq 1-(h-xT)/(h(r-1))$.
\end{proof}

The following result states that $f(x,H)$ is piecewise linear for certain types of $H$. These different behaviours already suggest that computing $f(x,H)$  is likely to be a difficult task in general.

\begin{prop}\label{pwlinear}
Let  $\ell_1\leq\dots\leq\ell_r$ be positive integers and $x\in(0,1]$.
Let $H$ be a complete $r$-partite ordered graph on $h$ vertices with parts $H_1<\cdots<H_r$ where $|H_i|=\ell_i$ for every $i\in[r]$. 
\begin{itemize}
\item If $\;t\ell_t-\sum\limits_{i=1}^t \ell_i\leq\frac{1-x}{x}h<(t+1) \ell_{t+1}-\sum\limits_{i=1}^{t+1} \ell_i$ for some $t\in[1,r-1]$ then
$$f(x,H)=1-\frac{1}{ht}\left((1-x)h+x\sum\limits_{i=1}^t \ell_i\right).$$
\item If $\;\frac{(1-x)h}{x}\geq r\ell_r-\sum\limits_{i=1}^{r} \ell_i$ then
$$ f(x,H)=1-\frac{h-x\ell_r}{h(r-1)}.\footnote{Observe that this value of $f(x,H)$ coincides with the value of $f(x,H)$ in the previous case when $t=r-1$.}$$
\end{itemize}
In particular, for $H$ fixed, $f(x,H)$ is continuous and piecewise linear in $x$ with $r'$ pieces where $r'$ is the number of strict inequalities in $\ell_1\leq\dots\leq\ell_r$.
\end{prop}

\begin{proof}
If $x=1$, then the result follows from Propositions~\ref{propeq} and \ref{complete1class}; so assume that $x <1$.
Let $T,J,x_0$ be as in Proposition~\ref{linearresult}. For our choice of $H$, we have $T=J=\ell_r$ and $x_0=h/(r\ell_r)$. Note that if
$$\frac{(1-x)h}{x}\geq r\ell_r-\sum\limits_{i=1}^{r} \ell_i,$$
then $x\leq x_0$. It follows from Proposition~\ref{linearresult} that
$$f(x,H)=1-\frac{h-xT}{h(r-1)}=1-\frac{h-x\ell_r}{h(r-1)}.$$
Throughout the rest of the proof, we assume that
\begin{align}\label{eqlinearcase1}
t\ell_t-\sum\limits_{i=1}^t \ell_i\leq\frac{1-x}{x}h<(t+1) \ell_{t+1}-\sum\limits_{i=1}^{t+1} \ell_i
\end{align}
for some $t\in[1,r-1]$. Fix constants $0<1/N\ll\eta\ll1$ where $N\in\mathbb{N}$. We first prove that $f(x,H)$ is at most the claimed value. Let $H'$ be the complete $r$-partite ordered graph with parts $H'_1<\cdots<H'_r$ such that the following holds:
\begin{itemize}
\item for every $i\leq t$,     
$$|H'_i|=\left\lfloor\frac{N}{t}\left(\frac{1-x}{x}h+\sum_{i=1}^t \ell_i\right)\right\rfloor$$
and so $N\ell_t\leq|H'_i|<N\ell_{t+1}$ by (\ref{eqlinearcase1});
\item  for every $i\geq t+1$, 
$$|H'_i|=N\ell_i.$$
\end{itemize}
As $N\ell_i\leq|H'_i|$ for every $i\in[r]$,  there exists an $H$-tiling $\mathcal{H}$ in $H'$ consisting of exactly $N$ copies of $H$. Note that $\mathcal{H}$ covers exactly $Nh$ vertices in $H'$ and 
$$x|H'|\leq xN\left(\frac{1-x}{x}h+\sum_{i=1}^t \ell_i\right)+xN\sum_{i=t+1}^r \ell_i=xN\left(\frac{1-x}{x}h+h\right)=Nh.$$
Thus, $\mathcal{H}$ is an $(x,H)$-tiling in $H'$. Since $H'$ contains an $(x,H)$-tiling,  every $1$-bottlegraph of $H'$ is an $x$-bottlegraph of $H$, therefore $\chi_{cr}^*(x,H)\leq\chi_{cr}^*(H')$.  

Since $|H'_1|\leq\dots\leq |H'_r|$, Proposition~\ref{complete1class} implies
$
   \chi_{cr}^*(H')={|H'|}/{|H'_1|}. 
$
Therefore,
\begin{align*}
f(x,H)&=\left(1-\frac{1}{\chi_{cr}^*(x,H)}\right)\leq\left(1-\frac{1}{\chi_{cr}^*(H')}\right)=1-\frac{|H'_1|}{|H'|}
=1-\frac{\left\lfloor\frac{N}{t}\left(\frac{1-x}{x}h+\sum\limits_{i=1}^t \ell_i\right)\right\rfloor}{t\left\lfloor\frac{N}{t}\left(\frac{1-x}{x}h+\sum\limits_{i=1}^t \ell_i\right)\right\rfloor+N\sum\limits_{i=t+1}^r\ell_i} \\
&<
1-\frac{\frac{N}{t}\left(\frac{1-x}{x}h+\sum\limits_{i=1}^t \ell_i\right)-1}{N\left(\frac{1-x}{x}h+\sum\limits_{i=1}^t \ell_i\right)+N\sum\limits_{i=t+1}^r\ell_i}=1-\frac{1}{ht}\left((1-x)h+x\sum\limits_{i=1}^t \ell_i\right)+\frac{x}{Nh}.
\end{align*}
Since $N$ can be chosen arbitrarily large, the above implies
$$f(x,H)\leq 1-\frac{1}{ht}\left((1-x)h+x\sum\limits_{i=1}^t \ell_i\right).$$

Next we prove that $f(x,H)$ is at least the claimed value. Let $G$ be the ordered graph obtained by taking the complete $(t+1)$-partite ordered graph with parts $G_1<\cdots<G_{t+1}$ where
$$|G_{t+1}|=N\left(h-\sum\limits_{i=1}^t \ell_i\right)-t \quad \text{and} \quad 
|G_i|=\left\lceil\frac{N}{t}\left(\frac{1-x}{x}h+\sum\limits_{i=1}^t \ell_i\right)\right\rceil+1
$$
for every $i\leq t$, and adding all missing edges to $G_{t+1}$. Suppose $G$ contains an $(x,H)$-tiling $\mathcal{H}$. Observe that
$$|G|\geq N\left(\frac{1-x}{x}h+\sum\limits_{i=1}^t \ell_i\right)+t+N\left(h-\sum\limits_{i=1}^t \ell_i\right)-t=\frac{Nh}{x}\implies x|G|\geq Nh.$$ 
Hence, $\mathcal{H}$ covers at least $Nh$ vertices in $G$. Let $\mathcal{H'}\subseteq\mathcal{H}$ be an $H$-tiling consisting of exactly $N$ copies of $H$. Each copy of $H$ has at most $\sum\limits_{i=1}^t \ell_i$ vertices in $G_1\cup\cdots\cup G_t$. Also,
$$|G|<N\left(\frac{1-x}{x}h+\sum\limits_{i=1}^t \ell_i\right)+2t+N\left(h-\sum\limits_{i=1}^t \ell_i\right)-t=\frac{Nh}{x}+t,$$
and so there are fewer than $\frac{1-x}{x}Nh+t$ vertices in $G$ which are not covered by $\mathcal{H'}$. This implies
$$|G_1\cup\cdots\cup G_t|<N\sum_{i=1}^t \ell_i+\frac{(1-x)Nh}{x}+t=N\left(\frac{1-x}{x}h+\sum\limits_{i=1}^t \ell_i\right)+t\leq|G_1\cup\cdots\cup G_t|,$$
which is a contradiction; hence $G$ does not contain an $(x,H)$-tiling. Therefore,  Theorem~\ref{OrderedKomlos} implies that
$\delta(G)<(f(x,H)+\eta)|G|$
and so
$$f(x,H)+\eta > \frac{\delta(G)}{|G|}=\frac{|G|-|G_1|}{|G|}=1-\frac{\left\lceil\frac{N}{t}\left(\frac{1-x}{x}h+\sum\limits_{i=1}^t \ell_i\right)\right\rceil+1}{t\left\lceil\frac{N}{t}\left(\frac{1-x}{x}h+\sum\limits_{i=1}^t \ell_i\right)\right\rceil+t+N\left(h-\sum\limits_{i=1}^t \ell_i\right)-t}$$

$$\geq 1-\frac{\frac{N}{t}\left(\frac{1-x}{x}h+\sum\limits_{i=1}^t \ell_i\right)+2}{N\left(\frac{1-x}{x}h+\sum\limits_{i=1}^t \ell_i\right)+N\left(h-\sum\limits_{i=1}^t \ell_i\right)}=1-\frac{1}{ht}\left((1-x)h+x\sum\limits_{i=1}^t \ell_i\right)+\frac{2x}{Nh}.$$
Since $\eta$ can be chosen arbitrarily small and $N$  arbitrarily large, the above implies that
$$f(x,H)\geq1-\frac{1}{ht}\left((1-x)h+x\sum\limits_{i=1}^t \ell_i\right).$$
\end{proof}

We conclude this section with the following result.
\begin{prop}
Let $H$ be an ordered graph, then
$$\lim_{x\to 1}\chi_{cr}^*(x,H)=\chi^*_{cr}(H)\quad\text{and}\quad\lim_{x\to 0}\chi_{cr}^*(x,H)=\chi_<(H)-1.$$
\end{prop}
\begin{proof}
Let $T,J,x_0$ be as in Proposition~\ref{linearresult}. If $x\in(0,x_0)$ then Proposition~\ref{linearresult} implies
$$f(x,H)=1-\frac{h-xT}{h(\chi_<(H)-1)}\implies\lim_{x\to 0} f(x,H)=1-\frac{1}{\chi_<(H)-1}.$$
Since $f(x,H) =1-1/\chi_{cr}^*(x,H)$, it follows from the above that
$$\lim_{x\to 0}\chi_{cr}^*(x,H)=\chi_<(H)-1.$$

Let $y,z\in(0,1]$ with $y<z$. Let $B$ be a $z$-bottlegraph of $H$. Clearly every $z$-bottlegraph of $H$ is a $y$-bottlegraph of $H$, so
$\chi_{cr}^*(y,H)\leq\chi_{cr}^*(z,H)$. 
It follows that, for $x\in(0,1]$, the function $\chi_{cr}^*(x,H)$ is non-decreasing. In particular, $$\chi^*_{cr}(x,H)\leq\chi^*_{cr}(1,H)=\chi_{cr}^*(H)$$
for every $x\in(0,1]$. The monotone convergence theorem implies that the limit $\ell:=\lim_{x\to 1}\chi_{cr}^*(x,H)$
exists and $\ell\leq\chi_{cr}^*(H)$. Fix arbitrary constants $0<\eta,\varepsilon<1$. Pick $n\in\mathbb{N}$ sufficiently large and let $G$ be an ordered graph on $n$ vertices with minimum degree $$\delta(G)\geq\left(1-\frac{1}{\ell}+\varepsilon\right)n\geq\left(1-\frac{1}{\chi_{cr}^*(1-\eta,H)}+\varepsilon\right)n.$$
By Theorem~\ref{OrderedKomlos}, $G$ contains a $(1-\eta,H)$-tiling, i.e., an $H$-tiling covering all but at most $\eta n$ vertices. Recall the definition of $c_<(H)$ given at the beginning of Section~\ref{lowersec}: $c_<(H)$ denotes the smallest non-negative number such that for every $\eta>0$, there exists an integer $n_0\in\mathbb{N}$ so that if $G$ is an ordered graph on $n\geq n_0$ vertices and with minimum degree $\delta(G)\geq c_<(H)n$ then $G$ contains an $H$-tiling covering all but at most $\eta n$ vertices.

It follows that $c_<(H)\leq1-1/\ell+\varepsilon$, and so $c_<(H)\leq1-1/\ell$ since $\varepsilon>0$ is arbitrary. By Theorem~\ref{aptsthm}, 
$c_<(H)=1-{1}/{\chi^*_{cr}(H)}$
and thus $\chi_{cr}^*(H)\leq\ell$. In particular, $\chi_{cr}^*(H)=\ell$.

\end{proof}

\section{Concluding remarks and open problems}\label{sec:conc}
Theorem~\ref{mainthm} together with~\cite[Theorem~1.9]{blt} asymptotically determine the minimum degree threshold for forcing a perfect $H$-tiling in an ordered graph, for any fixed ordered graph $H$. Depending on the structure of $H$, this threshold depends on one of three factors: ($C1$) the existence of an almost perfect $H$-tiling;  ($C2$) the avoidance of divisibility barriers; ($C3$) the existence of an $H$-cover. Analogous factors govern the threshold for other perfect $H$-tiling problems in a range of settings too. 

Therefore, it would be extremely interesting to find a natural `local' density condition (e.g., minimum degree, Ore-type, degree sequence) for an (ordered) graph, directed graph or hypergraph for which the corresponding perfect $H$-tiling threshold depends on another factor. We suspect no such problem exists. An alternative way to think about this question is as follows: are there barriers, other than local and divisibility barriers, that prevent absorbing for a perfect $H$-tiling problem?

Other than this general `meta problem', it would be interesting to establish the Ore-type degree threshold that forces a perfect $H$-tiling in an ordered graph $G$ (and compare this threshold to the corresponding Ore-type degree threshold for unordered graphs~\cite{orekot}).

In light of Theorem~\ref{BalLiTre1} it is natural to raise the following ordered graph analogue of
the theorem of Shokoufandeh and  Zhao~\cite{szhao}.

\begin{conj}
Let $H$ be an ordered graph.  Then there is a constant $C=C(H)\in \mathbb N$ so that the following holds. 
If $G$ is an $n$-vertex ordered graph  with
$$\delta(G)\geq\left(1-\frac{1}{\chi^*_{cr}(H)}\right)n$$
then $G$ contains an $H$-tiling covering all but at most $C$ vertices.
\end{conj}

Finally, whilst we have obtained some understanding of the function $f(x,H)$, it would be interesting to obtain a more complete understanding of how this function can behave in general. In particular, is it true that for any fixed ordered graph $H$, $f(x,H)$ is piecewise linear?

\section*{Acknowledgment}
The authors are grateful to the referees for their helpful and careful reviews.

{\noindent \bf Conflicts of interest:} none. 

{\noindent \bf Financial Support:} The second author is supported by EPSRC grant EP/V002279/1.  There are no additional data beyond that contained within the main manuscript.

\end{document}